\documentclass[12pt]{amsart}
\usepackage{amscd,amssymb,amsmath,verbatim,amsthm, graphicx}

\usepackage{color}

\newcommand{\nwBase}{\widetilde{\calE}_{K}}
\newcommand{\nwTotl}{\widetilde{\calC}_{K}}

\newcommand{\bS}{\mathbb S}
\newcommand{\EE}{\mathbb E}

\newcommand{\CC}{\mathbb C}

\newcommand{\NN}{\mathbb N}
\newcommand{\RR}{\mathbb R}
\newcommand{\ZZ}{\mathbb Z}
\newcommand{\del}{\partial}

\newcommand{\al}{\alpha}

\newcommand{\calA}{{\mathcal A}}
\newcommand{\calB}{{\mathcal B}}
\newcommand{\calC}{{\mathcal C}}
\newcommand{\calD}{{\mathcal D}}
\newcommand{\calE}{{\mathcal E}}
\newcommand{\calF}{{\mathcal F}}

\newcommand{\calI}{{\mathcal I}}
\newcommand{\calJ}{{\mathcal J}}

\newcommand{\calL}{{\mathcal L}}
\newcommand{\calM}{{\mathcal M}}
\newcommand{\calN}{{\mathcal N}}
\newcommand{\calO}{{\mathcal O}}
\newcommand{\calP}{{\mathcal P}}
\newcommand{\calS}{{\mathcal S}}
\newcommand{\calT}{{\mathcal T}}
\newcommand{\calU}{{\mathcal U}}
\newcommand{\calV}{{\mathcal V}}
\newcommand{\calW}{{\mathcal W}}
\newcommand{\MP}{\calM\!\calP}
\newcommand{\SCM}{\calS\calC\!\calM}
\newcommand{\frakb}{{\mathfrak b}}

\newcommand{\spec}{\mbox{spec\,}}
\newcommand{\frakc}{\mathfrak c}
\newcommand{\frakd}{\mathfrak d}
\newcommand{\frakp}{\mathfrak p}
\newcommand{\frakq}{\mathfrak q}
\newcommand{\frakr}{\mathfrak r}
\newcommand{\fraks}{\mathfrak s}

\newcommand{\Metcc}{{\mathfrak M}{\mathfrak e}{\mathfrak t}_{\mathrm{cc}}}
\newcommand{\frakC}{\mathfrak C}

\newcommand{\Fr}{\mathrm{Fr}}

\newcommand{\tfrakq}{\tilde{\frakq}}
\newcommand{\hfrakq}{\hat{\frakq}}

\newtheorem{theorem}{Theorem}
\newtheorem{proposition}{Proposition}

\newtheorem{lemma}{Lemma}
\theoremstyle{definition}
\newtheorem{definition}{Definition}
\theoremstyle{remark}
\newtheorem*{remark}{Remark}

\title{Conical metrics on Riemann surfaces, II: spherical metrics} 

\author[R. Mazzeo]{Rafe Mazzeo}
\address{Department of Mathematics, Stanford University, 380 Serra Mall, Stanford CA 94305}
\email{rmazzeo@stanford.edu}
\author[X. Zhu]{Xuwen Zhu}
\address{Department of Mathematics, Northeastern University, 549 Lake Hall, Boston MA 02115}
\email{x.zhu@northeastern.edu}

\date{}

\begin{document}

\maketitle

\begin{abstract}
We continue our study, initiated in \cite{MZ}, of Riemann surfaces with constant curvature and isolated 
conic singularities. Using the machinery developed in that earlier paper of extended configuration families 
of simple divisors, we study the existence and deformation theory for spherical conic metrics with some 
or all of the cone angles greater than $2\pi$.  Deformations are obstructed precisely when the number $2$ 
lies in the spectrum of the Friedrichs extension of the Laplacian. Our main result is that, in this case, it is
possible to find a smooth local moduli space of solutions by allowing the cone points to split. 
This analytic fact reflects geometric constructions in \cite{MP, MP2}. 
\end{abstract}

\section{Introduction}
We shall study the following problem: given a compact Riemann surface $M$, a collection
of distinct points $\frakp = \{p_1, \ldots, p_k\} \subset M$ and a collection of positive real numbers
$\beta_1, \ldots, \beta_k$, is it possible to find a metric $g$ on $M$ with constant
curvature and with conic singularities with prescribed cone angles $2\pi \beta_j$ at
the points $p_j$?   If there is a solution, the sign of its curvature is the same as that of 
the conic Euler characteristic 
\begin{equation}\label{e:GB}
\chi(M, \vec \beta) = \chi(M) + \sum_{j=1}^k (\beta_j-1), 
\end{equation}
by virtue of the `conic' Gau\ss-Bonnet formula
\[
\int_M K\, dA = 2\pi \chi(M, \vec \beta).
\]
We always normalize by assuming that $K \in \{-1, 0, 1\}$. 

When $\chi(M, \vec \beta) \leq 0$, existence and uniqueness of solutions for any $\vec \beta \in \RR^k_+$ is
easy to prove using barrier arguments \cite{Mc}.  The spherical case, $K = 1$, has proved more challenging, and
many questions remain open. For cone angles lying in $(0,2\pi)$, Troyanov \cite{Tr} discovered an auxiliary set of 
linear inequalities on the $\beta_j$ 
which are necessary and sufficient for existence; later, Luo and Tian \cite{LT} proved uniqueness of the solution
in this angle regime. When $M \neq \bS^2$, existence was recently proved by Mondello and Panov \cite{MP2} 
for any $\vec \beta$ with $\chi(M, \vec \beta) > 0$, at the expense of not being able to specify the conformal 
class on $M$, see also \cite{BMM}. When $M = \bS^2$, the same two authors \cite{MP} gave necessary conditions 
on $\vec \beta$ for existence, again in the form of a set of linear inequalities, and proved existence in the interior
of this region. In this case one is unable to specify the marked conformal class, i.e., the location of the points 
$\frakp$ on $\bS^{2}$.  In either of these settings, uniqueness sometimes fails.   We also wish to understand the deformation 
theory, i.e., how solutions depend on the `conic data', i.e., the conformal class, the set $\frakp$ and the cone angle
parameters $\vec \beta$.  This is understood when $\chi(M, \vec \beta) \leq 0$, and also in the spherical
case when all $\beta_j < 1$ \cite{MW, MZ}. However, for all of these questions, the complete story in the 
spherical case with at least some of the cone angles greater than $2\pi$ still has many gaps. We review the history 
and further literature for this problem in \S 2.  

The main results in this paper provide new perspectives and insight into these existence and moduli questions 
and indicate potential new intricacies. Our focus in this paper is the local deformation theory for this problem, following 
the work of the first author and Wei\ss~\cite{MW}, but relying heavily on the geometric tools developed in our earlier 
paper~\cite{MZ}.   More specifically, suppose that $g$ is a spherical cone metric on $M$ with `conic data set' 
$\frakd_r(g) = \{\frakc, \frakp, \vec \beta\}$: $\frakc$ is the conformal class of $g$ on $M$, $\frakp = \{p_1,
\ldots, p_k\}$ is an ordered $k$-tuple of points on $M$, and $g$ has a conic singularity with cone angle $2\pi \beta_j$ at 
$p_j$, $j = 1, \ldots, k$. We consider the question of whether all nearby data sets $\{\frakc', \frakp', \vec \beta'\}$ 
are attained by nearby spherical cone metrics, and whether these metrics depend smoothly on these conic data sets.

In this paper we consider these questions at the `premoduli' level, i.e., before taking the quotient by the relevant 
diffeomorphism group.  Indeed, an important feature of the work below involves analyzing families of solutions
when the $k$-tuples of points either merge or split, and there are subtleties in passing to the diffeomorphism 
quotient in these circumstances.  A careful discussion of this problem is deferred to elsewhere.  
Thus we associate to $g$ its {\it unreduced} conic data set $\frakd(g)$; this consists 
of the triple $\{ \fraks(g), \frakp, \vec \beta\}$ where $\fraks(g)$ is the {\it smooth} constant curvature metric 
(normalized to have area $|2\pi \chi(M)|$ when $\chi(M) \neq 0$ and area $1$ when $\chi(M) = 0$) in the (unmarked) conformal 
class of $g$, and as before, $\frakp$ and $\vec \beta$ indicate the locations of the conic singularities and 
the cone angles. We denote by $\Metcc$ the Banach manifold of all smooth constant curvature metrics; this is
infinite dimensional since we are not taking the quotient by diffeomorphisms. Each $k$-tuple $\frakp$ lies in the 
$k$-fold product $M^k$ away from any of the partial diagonals. As we explain in \S 3, this open set in $M^k$ is 
identified with the interior of the `extended configuration space' $\calE_k$.  Altogether then, unreduced conic data 
sets lie in $\Metcc \times \mathrm{int}\, \calE_k \times \RR^k_+$.
In the following we usually refer to unreduced conic data sets simply as conic data sets.

Our first result is a consequence of the analysis in \cite{MW}:
\begin{theorem} Let $g$ be a spherical cone metric as above with conic data set $\frakd(g)$. Let 
$\Delta_g$ denote the Friedrichs realization of the scalar Laplace-Beltrami operator associated to $g$. 
If $2 \not\in \mathrm{spec}\,(\Delta_g)$, then the premoduli space of spherical cone metrics is a smooth 
Banach manifold near $g$ which projects diffeomorphically to an open set in the space of data sets 
$\Metcc \times \mathrm{int}\, \calE_k \times \RR^k_+$ containing $\frakd(g)$. 
\end{theorem}

When all $\beta_j < 1$, the premoduli space of spherical cone metrics is globally diffeomorphic to the space of 
data sets where the $\beta_j$ satisfy the Troyanov condition \eqref{Troy}, see \cite{MW}.  
For larger cone angles, one might expect the moduli space of spherical cone metrics to `fold', e.g., the projection from
the space of solutions to the space of data sets may no longer be one-to-one. If the moduli space is a smooth 
manifold, one might even expect to use degree theory to take a signed count of solutions, thus quantifying the 
lack of uniqueness. Our main result indicates that the (pre)moduli space is not a smooth manifold, but only stratified, 
which means that any such enumeration of solutions may be difficult. The key problem is that the deformation theory 
is obstructed when $2 \in \mathrm{spec}\,(\Delta_g)$. We show that this spectral condition is unavoidable. In fact, 
the set of cone angle data $\vec \beta$ for which there exists a solution metric with $2$ in the spectrum is 
unbounded in $(\RR^+)^k$.  Furthermore, if this spectral condition holds, then there are explicit examples which 
exhibit that the local deformation theory is obstructed, see \cite{Zhu19}. Our main result is that even if $2$ does 
lie in the spectrum, there is an unobstructed deformation space if we allow for more drastic deformations which permit
the individual points $p_j$ to `splinter' into a collection of conic points with smaller cone angles.  This splitting of cone points
already appears in the purely geometric arguments in \cite{MP}, but enters our analytic arguments in an apparently
different way.   

An alternative perspective on our work here is that we determine the behavior of families of spherical cone metrics as 
the underlying marked conformal structure degenerates in the sense that various subcollections of points coalesce.

To state the following theorem, we introduce some notation. Fix a $k$-tuple $\frakp_{0} \in M^k$ with $p_i \neq p_j$
for all $i, j$. Choose positive integers $N_i$, $i = 1, \ldots, k$, and set $K = N_1 + \ldots + N_k$ (see Theorem 2 below for their definition). Define a new $K$-tuple 
$\frakq = \{q_1, \ldots, q_K\}$ by repeating the point $p_1$ $N_1$ times, $p_2$ $N_2$ times, and so on. 
This point $\frakq$ lies in some intersection of partial diagonals in $M^K$. The extended configuration space $\calE_K$ 
is a resolution of $M^K$ obtained by blowing up these partial diagonals (see \S 3 for a precise definition) and there is a 
boundary face $F_0$ of $\calE_K$ (it is a boundary hypersurface if only one $N_i > 1$, and a corner otherwise) 
which lies above $\frakq$.  As we describe carefully in \S 6, it turns out to be necessary to perform an additional set 
of blowups on $\calE_{K}$, leading to a slightly larger space $\widetilde{\calE}_K$. We then consider points 
$\tilde{\frakq}$ lying in the interior of the front face $\widetilde{F}_0$ of this new space over the point $\frakq$. 

We also specify the choice of cone angle parameters for these extended sets of points.  Let $\vec \beta$ be the angle 
parameter vector for $\frakp$; we say that the $K$-tuple $\vec B \in (\RR^+)^K$ is admissible if the Gau\ss-Bonnet sum is preserved, i.e.
\begin{equation}
\beta_j - 1 = \sum_{i=N_{1}+\dots + N_{j-1}+1}^{N_{1}+\dots + N_{j}}  (B_i - 1).
\label{admissible}
\end{equation}
and no subcluster merge to $2\pi$ (see Definition~\ref{def:B} for details.)  We then fix any admissible $\vec B$ as the set of cone angle parameters 
for $K$-tuples $\frakq'$ near $\frakq$.

Our main result can now be stated, albeit slightly imprecisely:
\begin{theorem}
Let $g$ be a spherical cone metric with conic data set $\frakd(g) = (\fraks(g), \frakp, \vec \beta)$ and suppose
that $2 \in \mathrm{spec}\,(\Delta_g)$. Define $K =  \sum_{j=1}^k N_j$ where $N_j = \max\{[\beta_j], 1\}$ and 
consider all points $\tilde{\frakq}' \in \nwBase$ which lie in a small neighborhood of the point $\tilde{\frakq}$,
i.e., the (not necessarily distinct) points $q_i'$, $N_{j} +1 \leq i \leq N_{j+1}$, lie in a small cluster around 
the point $p_j$. Let $\fraks(g)'$ be a conformal structure close to $\fraks(g)$ and $\vec B$ an admissible 
$K$-tuple of cone angle parameters for the points $\frakq'$.  Then there exists a $p$-submanifold $X \subset \nwBase$ 
containing $\tilde{\frakq}$, the tangent space of which at $\tilde{\frakq}$ is determined by data drawn from the elements
of the eigenspace of $\Delta_g$ with eigenvalue $2$, and a diffeomorphism from $X$ to the premoduli 
space of spherical cone metrics with $K$ cone points near $g$ with cone angle parameters $\vec B$ and 
background conformal class $\fraks(g)'$. 
\end{theorem}
The more precise statement will require further definitions. The idea is simply that, having fixed $\vec B$, 
there is a `good' space $X$ of $K$-tuples of conic points $\frakq$ which arises by splitting various of the 
individual cone points in $\frakp$ into small clusters.  The locations of these clustering families is
encoded by the configuration space $\nwBase$.  To say that $X$ is a $p$-submanifold means simply
that it intersects the boundaries and corners of $\nwBase$ cleanly.  We are thus asserting that for a given $\vec B$ 
and $\fraks(g)'$, there exist smooth families of spherical cone metrics $g'$ near to $g$ and with conic points at the 
$K$-tuples $\frakq'$ near (in the sense of merging) to $\frakp$ if and only if $\tilde{\frakq}' \in X \subset \nwBase$. 

Key tools here are the use of the extended configuration spaces $\calE_K$ (and later, $\nwBase$), as well as the
associated extended configuration families $\calC_K$; these were defined and studied in great detail, and 
play a central role in our earlier paper \cite{MZ}.  We review their geometry carefully in \S 3,  but refer to \cite{MZ}
for a more definitive treatment. For now, however, we recall that each $\calE_K$ is a manifold with corners which is 
a compactification of the open set in $M^K$ consisting of all {\it distinct} ordered $K$-tuples $\{p_1, \ldots, p_K\}$;
$\calC_K$ is a universal curve over this configuration space in the sense that it too is a manifold with corners equipped
with a singular fibration over $\calE_K$. Over the interior of $\calE_K$, the fiber of any $\{p_1, \ldots, p_K\}$ is a copy 
of $M$ blown up at these $K$ points.  The heart of our method is to construct families of fiberwise metrics on $\calC_K$ 
solving the curvature equation to infinite order at the faces of $\calC_K$ which correspond to the collapse of a $K$-tuple 
$\{q_1, \ldots, q_K\}$ to a $k$-tuple $\{p_1, \ldots, p_k\}$. We also show that the infinitesimal deformations of these 
approximate solutions fill out the cokernel of the linearization. This main result then follows from the implicit function theorem (see \S 7 and \S 8).

The geometry of these spaces is quite complicated, but they capture rather complete information about families of constant 
curvature conic metrics. For example, one of the main results of \cite{MZ} states that when $\chi(M, \vec \beta) \leq 0$, hence 
solutions exist for all choices of data sets, then the solution families are polyhomogeneous, i.e., maximally smooth, as a 
family of fiber metrics on $\calC_k$. The analogous regularity holds in this spherical setting too.   Our methods for analyzing 
the family of conic Laplacians on the fibers of $\calC_K$ should be useful in other problems 
involving families of elliptic operators with merging regular singularities. 

This paper is organized as follows. In \S 2 we give a review of existing literatures on spherical conic metrics. In \S 3 we describe the two configuration spaces $\calE_{k}$ and $\calC_{k}$. In \S 4 we discuss the mapping and regularity properties of the linearized Liouville operator, and in particular we prove the deformation theory in the unobstructed case when $2\notin \spec(\Delta_{g})$. In \S 5 we describe the locus of degenerate spherical conic metrics and show there are many spherical cone metrics with 2 in the spectrum of $\Delta_{g}$. In \S 6 we describe the local and global behavior of families of metrics with cone points splitting into clusters which generates a family of functions that will unobstruct the main deformation problem.  We also show the construction of spaces $\tilde \calE_{K}$ and $\calC_{K}$, which is necessary to desingularize the parametrization. In \S 7 we construct projected solutions that solve the Liouville equation modulo the finite dimensional space of 2-eigenfunctions, and show these solutions are polyhomogeneous on the configuration spaces. In \S 8 we finally identify those configurations of conic points on $\tilde \calE_{K}$ that remove the error from \S 7, determined asymptotically by a symplectic pairing formula using eigenfunctions and functions generated by splitting of cone points from \S 6, therefore proving the final deformation theorem. 

\section{Spherical conic metrics}
We now review at least some of the rather extensive history of the study of spherical conic metrics. These fundamental
objects have the beguiling feature that they arise in many places in mathematics and may be approached from many 
different points of view, including synthetic geometry, complex analysis, the theory of character varieties, 
calculus of variations and other methods of geometric analysis. 

As noted earlier, it is quite easy to prove \cite{Mc} that there exist hyperbolic or flat conic metrics 
with any prescribed data sets $\{\fraks, \frakp, \vec \beta\}$, where the sign of $\chi(M, \vec \beta) \leq 0$
determines the curvature $K \in \{-1, 0\}$.  Recall that, as in the introduction, $\fraks$ is a smooth 
constant curvature metric uniformizing its conformal class, $\frakp = \{p_1, \ldots, p_k\}$ is
a $k$-tuple of distinct points on $M$ and $\vec \beta \in \RR^k_+$.   Indeed, let $h_0$ be a smooth uniformizing
(nonconic) metric representing any given conformal structure on $M$. Then this problem reduces to solving Liouville's equation:
\begin{equation}
\Delta_{h_0} u + K_{h_0} - K e^{2u} = 0
\label{Liouville}
\end{equation}
where 
\[
\Delta = -\, \mathrm{div}\, \nabla
\]
and $K_{h_0}$ and $K$ are the Gauss curvatures of $g_0$ and $g = e^{2u} h_0$.  The conic singularities arise 
from the `boundary value' 
\begin{equation}
u(z) = (\beta_j-1) \log |z|  + \calO(1)\ \ \ \mbox{near}\ \ p_j,
\label{bdrycond}
\end{equation}
where $z$ is a local holomorphic coordinate centered at $p_j$.   Nonpositive curvature of $h_0$ makes
the signs favorable so that one can find solutions by the method of barriers.  

As already explained, first in \cite{MW} and then in \cite{MZ}, the deformation theory is unobstructed in
these two  cases. (Actually, when $K=0$ there is a minor issue related to indeterminacy of scale which 
can be remedied by an area normalization.) This means that if $g$ is any hyperbolic or flat conic metric, 
and if we assign to $g$ its conic data
\begin{equation}
\frakd(g) = \{\fraks(g),  \frakp, \vec \beta\}, 
\label{conicdata}
\end{equation}
then for any data set near to this given one (subject to the constraint that $\chi(M,\vec \beta)$ either remains negative
or remains equal to $0$) there exists a unique solution of the problem with the same
curvature, and this solution depends smoothly on this data. The point of view in \cite{MW} is that if we fix the area, then 
as the cone angles vary the solution may change smoothly from hyperbolic to flat to spherical. 
This argument relies only on the surjectivity of the {\it scalar} operator $\Delta_g - 2K$, which is obvious when $K =-1$, 
and true once we factor out the constants when $K = 0$.  In \cite{MW} it is shown that this operator is also invertible
when $K = +1$ provided the cone angles are all less than $2\pi$, so that the moduli space of
solutions is smooth in this case too. As we show below, these arguments can be extended to handle the
case of spherical cone metrics for which some or all of the angles are greater than $2\pi$ provided 
we {\it assume} that $\Delta_g - 2$ is invertible, i.e., provided that $2 \not\in  \mathrm{spec}\,(\Delta_g)$. 

On the other hand, $2$ often does lie in the spectrum. For example, if $F: M \to \bS^2$
is a branched cover and $g_0$ is the round metric on $\bS^2$, then $F^* g_0$ is a spherical
conic metric on $M$, where the ramification points and ramification indices give the cone points and
cone angles (which are therefore all integer multiples of $2\pi$).  If $\phi$ is an eigenfunction
of the Laplacian on $\bS^2$ with eigenvalue $2$, then $F^* \phi$ is an eigenfunction on $M$
for the Friedrichs extension of $\Delta_g$, also with eigenvalue $2$.   As another example,
the football, with any cone angle, has eigenvalue $2$, as do certain connected sums 
along short geodesics of footballs with one another (see \S 5), or with these ramified covers.   
Thus there are many spherical cone metrics for which $2$ does lie in the spectrum. 

The recent work of Bin Xu and the second author~\cite{XuZhu} shows that coaxial metrics always have eigenvalue $2$ (see~\S\ref{ss:coaxial} for the definition of coaxial/reducible metrics). On the other hand, the works of Mondello and Panov~\cite{MP2}, and Eremenko, Gabrielov and Tarasov~\cite{EGT3}, indicate that there also exist non-coaxial metrics with eigenvalue $2$.

We now discuss other methods and results which have been used to study spherical cone metrics.

We begin by clarifying the existence theory when the cone angles are less than $2\pi$. 
By the conic Gau\ss-Bonnet formula~\eqref{e:GB}, if all the $\beta_{i}$ are less than $1$, then, assuming
$M$ is orientable, a spherical metric with these cone angles exists only if $M = \bS^{2}$.  A straightforward fact,
observed by Troyanov \cite{Tr2}, is that when $k=2$, a solution exists if and only if $\beta_{1}=\beta_{2}$, and 
in this case $(M,g)$ is a spherical football. A significantly more substantial result of his \cite{Tr} proves existence 
when $k\geq 3$ by a variational argument which involves a strengthening of the classical Moser-Trudinger 
inequality adapted to this conic setting. 
A solution exists in this case if and only if either $M \neq \bS^2$, or else $M = \bS^2$ and 
\begin{equation}
\beta_j-1 > \sum_{i \neq j} (\beta_i - 1) \quad \mbox{for each}\ j.
\label{Troy}
\end{equation}
A later result by Luo and Tian \cite{LT} shows that Troyanov's solution is unique. This `Troyanov condition' has been 
interpreted \cite{RossThomas} as a version of the famous K-stability condition in complex geometry.

Troyanov's argument relies heavily on the fact that under these angle conditions, the associated Liouville energy (for which
the PDE associated to this problem is the Euler-Lagrange equation) is bounded below, so that one can look for solutions as 
minima of this energy.  This argument also 
works in a very limited range of $\RR^k_+$ where some of the $\beta_i$ are greater than $1$. However, for most 
$\vec \beta$, the energy is unbounded below.  An early breakthrough was a generalization of Troyanov's variational method, due to Bartolucci and Tarantello~\cite{BT}, later generalized by Bartolucci, De Marchis and Malchiodi~\cite{BMM},
who prove the existence of minimax solutions with an arbitrary angle combination except away from a critical set of cone angles.
The critical points found by this approach are not (local) minima.  They use a subtle mountain pass
lemma, and along the way, assume crucially that $M \neq \bS^2$. The paper~\cite{BMM} also shows that 
in certain cases, solution with a given conic data set are not unique. This variational method has been pushed much further 
by Malchiodi and his collaborators, see \cite{BM, Carlotto, CM1, CM2, MRuiz} and the citations therein.  
A quite general result of this kind was announced recently by Carlotto and Malchiodi~\cite{Ma1, Ma2}, but details
have not appeared. 

A related method involves the computation of the Leray-Schauder degree for the curvature equation. We mention 
the work of Chen and Lin~\cite{CL1, CL2, CL3} and further papers with their collaborators~\cite{CCW, LW1, LW2}; these 
give the existence and nonexistence of solutions to the curvature equation when the angle parameters are away from a certain critical set. 

There is a classical approach to this problem involving complex analysis. Indeed, as already discussed, the 
special case when the cone angles are integer multiples of $2\pi$ is closely related to the  theory of ramified coverings of 
Riemann surfaces. Even here, the full story is not known, see~\cite{EG2, EGSV, Goldberg, Sch, Zhu18}. We also mention the papers 
of Eremenko \cite{Ere2} and Umehara-Yamada \cite{UY}, which give a complete description when 
$M=\bS^2$ and $k=3$, and~\cite{EG, EGT, EGT2} for some symmetric cases when $k=4$, \cite{ET} for the case of three noninteger angles and any number of integer angles. Recently Eremenko has also 
showed that the number of solutions is finite when $k=4$ and none of the angles is a multiple of $2\pi$~\cite{Ere19}. 
For  metrics with special monodromy, we also mention the recent papers by Xu et al.~\cite{CWWX, SCLX}.

A breakthrough using purely geometric (completely non-analytic!) methods was obtained recently 
by Mondello and Panov \cite{MP}. Their main result provides the generalization of the Troyanov region \eqref{Troy}, 
i.e., they describe a region $\mathcal M\! \mathcal P_k \subset \RR^k_+$ which characterizes the set of allowable 
cone angles of spherical cone metrics on $\bS^2$.  This region is described as
\begin{equation}
\mathcal M\! \mathcal P_k := \{ \vec\beta \in \RR^k_+:  d_1( \vec \beta - \vec 1, \ZZ^k_{\mathrm{odd}}) \geq 1\},
\label{MonPan}
\end{equation}
where $\ZZ^k_{\mathrm{odd}} = \{ \gamma \in \ZZ^k:  \sum \gamma_i \in 2\ZZ + 1\}$ and $d_1$ is the
$\ell^1$ norm on $\RR^k$. 

Their main result, a tour-de-force in classical geometry, is that any point in the interior of $\mathcal M \! \mathcal P_k$ 
is the vector of cone angles for at least one spherical cone metric on $\bS^{2}$; they are not able, however, to specify the 
locations of the conic points $p_j$, and do not address whether solutions exist when $\vec \beta$ is on the boundary 
of this region. After partial results of Dey~\cite{Dey} and Kapovich~\cite{Ka}, a complete answer was obtained by Eremenko~\cite{Ere1} on which $\vec \beta \in \partial \mathcal M\!\mathcal P_k$ are possible.

A more recent paper by Mondello and Panov~\cite{MP2} extends the result in~\cite{MP} to surfaces with higher genus 
and shows that when $M \neq \bS^2$, there exists a spherical conic metric for {\it any} $\vec \beta$ with
$\chi(M, \vec \beta) > 0$. They also show that it is not always possible to prescribe the underlying conformal class 
on $M$.  They also show that the moduli space of solutions when $M=\bS^{2}$ and $k$ is large has many 
connected components.   They identify a set of cone angle vectors $\vec \beta$, called the `bubbling set', 
near which families of solutions are expected to diverge. They prove that away from this set, the moduli space 
is smooth and the forgetful map which carries a spherical cone metric to its underlying marked conformal structure is 
proper.  This bubbling set is in fact the same as the set of critical angles appearing in the variational approach in~\cite{Ma2}.
Furthermore, when $M=\bS^{2}$, this bubbling set {{\it strictly contains} the set of cone angle vectors associated to 
spherical cone metrics with coaxial monodromy, see \cite{Ere1}. 

As explained earlier, we approach this problem via deformation theory, and focus on whether it is possible 
to freely deform the unreduced conic data near given set corresponding to an initial spherical cone metric $g$.  
The answer to this depends on the 
spectral behavior of the Friedrichs extension of the Laplacian $\Delta_g$.  Following \cite{MW} and \cite{MZ}, we 
prove that if $2$ does not lie in the spectrum of this operator, then the answer to this question is affirmative.  This 
relatively easy result motivates our key problem, which is to understand the local deformation theory when $2$ 
does lie in the spectrum. Our main result states that if we allow the cone points $p_j$ to break apart into clusters,
then even near these degenerate metrics there is a submanifold in the space of all conic data 
whose elements correspond to a branch in the space of spherical cone metrics.

We have already noted our central use of the extended configuration space $\calE_K$ and extended configuration family $\calC_K$,
both of which are defined and studied in our earlier paper \cite{MZ}. The former is a compactification of the space of $K$-tuples
of distinct points on $M$, while the latter is the bundle with fibre at $\hfrakq \in \calE_K$ the surface $M$ 
blown up at the points of $\frakq$.  Actually, the natural map $\calC_K \to \calE_K$ is a singular fibration (technically,
it is an example of a $b$-fibration), and the fibres over the boundary faces of $\calE_K$ are unions of surfaces with
boundary.  We describe this more carefully in the next section.   One of the main results in \cite{MZ} is that 
the families of hyperbolic and flat cone metrics with $k$ singular points extends to a polyhomogeneous family 
of fiberwise metrics on $\calC_k$. This is a sharp regularity statement: polyhomogeneity is a slight extension of 
the notion of smoothness which allows for series expansions with noninteger exponents.   In the cases studied 
in \cite{MZ}, existence was already known, but in the spherical setting that is no longer the case. Here 
it is not always possible to find a smooth family of spherical conic metrics near those with any given set of conic 
points $\frakp$.  Considering $\frakp$ as point on a corresponding face of $\tilde\calE_K$ (which is a space constructed from an extra blow up $\calE_{K}$, see \S 6), we show that there exists a 
smooth submanifold $X \subset \tilde\calE_K$ containing the initial $k$-tuple $\frakp$ as a coalescing limit such that 
spherical cone metrics exist for $K$-tuples $\hfrakq \in X$.

\section{Configuration spaces}\label{s:ekck}
There are two configuration spaces, $\calE_k$ and $\calC_k$ at the center of our construction. These are obtained 
by resolving the spaces $M^k$ and $M^k \times M$, respectively. These resolutions are obtained by the process of 
{\it real blowup}, i.e., where a $p$-submanifold $S$ of a manifold with corners $Z$ is replaced by its (inward 
pointing) spherical normal bundle.  (The $p$-submanifolds are the natural submanifolds in manifolds with corners 
for which tubular neighborhoods around them are represented by their normal bundles.)  We refer to \cite{MZ}
for a review of these notions, and for many further details about these spaces than we can recall here.

The key points we review here concern how the boundary faces of $\calE_{k}$ and $\calC_{k}$ encode information about the 
various ways in which subclusters of points can collide. 

\subsection{The extended configuration space $\calE_k$}
We start with $M^{k}$, the space of ordered $k$-tuples of not necessarily distinct points $p_1, \ldots, p_k \in M$. The extended 
configuration space $\calE_{k}$ is a canonically defined space defined by iteratively blowing up all the partial diagonals 
\[
\Delta_{\calI} = \{ \frakp = (p_1, \ldots, p_k) \in M^{k}:  p_{i} = p_{j}, \ \forall\ i,j \in \calI\},
\]
where $\calI$ is any subset of $\{1, \ldots, k\}$ with $|\calI| \geq 2$.  Thus, using the notation that $[X;S]$ denotes
the blowup of a manifold $X$ around a submanifold $S$, and compressing the fact that there is a chain of blowups, we write
\begin{equation}
{\calE_{k}}=[M^{k}; \cup_{\calI } \Delta_{\calI} ].
\end{equation}
The order of blowup is important, and we perform these in order of reverse inclusion, i.e., first blowing up 
the smaller partial diagonals, with larger $|\calI|$, see \cite{MZ}.   

The space $\calE_{k}$ is a manifold with corners; its boundary hypersurfaces $F_{\calI}$ are the `front faces' created by
blowing up each $\Delta_{\calI}$. The interior of $\calE_k$ is naturally identified with the open subset $\calU \subset M^k$
of $k$-tuples of distinct points. This identification extends smoothly to the blow-down map
$$
\beta_k: \calE_{k}\rightarrow M^k.
$$
Points of $\calE_k$ are sometimes denoted $\frakq$, so $\beta_{k}(\frakq) = \frakp$ is the underlying $k$-tuple
of points which may lie along one or more of the partial diagonals.  Finally, each face $F_{\calI}$ has a boundary
defining function, which we write as $\rho_{\calI}$. Thus $\rho_\calI(\frakq)$ measures the `clustering radius' of the
subcluster of points in $\frakp$ with indices in $\calI$. 

\subsection{The extended configuration family $\calC_k$}
We next describe the universal curve over $\calE_{k}$.  Consider the product $\calE_{k}\times M$; points in this space 
are $(\frakq, z)$, with $z$ lying in the fiber. The space $\calC_k$ is obtained by resolving the graph of the 
canonical multi-valued section $\sigma$ of this bundle defined as
$$
\bigcup_{\calI} \{(\frakq,z) \in \calE_{k}\times M: z \in \sigma^\calI(\frakp)\}. 
$$
If $\frakq \in F_\calI$, $\calI = (i_1, \ldots, i_r)$ (so $p_{i_1} = \ldots = p_{i_r}$), then $\sigma^\calI(\frakp)$ 
denotes the $(k-r+1)$-tuple obtained by adjoining this single `$r$-fold' point with 
the remaining $k-r$ points. Now define the coincidence set: 
\begin{equation}
F_{\calI}^\sigma=\{\rho_{\calI}=0, z\in \sigma^\calI(\frakp) \}.
\label{liftedincset}
\end{equation}
Abusing notation slightly, we write $F^\sigma_i$ for the nonsingular parts of the graph of $\sigma$, i.e.,
the sets $\{ z = p_i\}$, $i = 1, \ldots, k$ where $\frakp$ does not lie in any partial diagonal. 

The extended configuration family $\calC_k$ is defined as the iterated blowup
\begin{equation}
\calC_k = \big[ \calE_k \times M; \ \cup_{\calI} F^\sigma_\calI \big]
\label{extconfspace}
\end{equation}
where, once again, we blow up in order of reverse inclusion of the subsets $\calI$. This canonical space is again a manifold with corners. 
\subsection{The map $\calC_k \to \calE_k$} 
The trivial fibration $\calE_{k}\times M\rightarrow \calE_{k}$ lifts to a $b$-fibration:
\begin{equation}
\pi_{k}: \calC_k \longrightarrow \calE_k,
\end{equation}
whose geometry we now recall. 

If $\frakq$ lies in $\mathrm{int}\, \calE_k$, then $\pi_{k}^{-1}( \frakq):=M_{\frakp}$ is the surface $M$ blown up at the 
$k$ points of $\beta_k(\frakq) = \frakp=\{p_{1}, \dots, p_{k}\}$. This fiber is a surface with $k$ boundary components, 
each a copy of $\bS^1$.   However, if $\frakq\in F_{\calI}$, $|\calI| = r$, then the fiber $\pi_{k}^{-1}( \frakq)$ is a
`tied manifold', i.e., a union of surfaces with boundary identified in a certain pattern along their boundaries.
One of these surfaces is $M_{\beta_k(\frakq)}$, while the others are a certain number of copies of the hemisphere $\bS^2_+$
blown up at a collection of points. We usually write $\frakp = \beta_k(\frakq)$ below. 
The combinatorial structure of how these surfaces fit together encode the 
various regimes by which $r$ points can cluster. 

Let $\frakC_{\calI}$ denote the collection of new boundary faces generated by blowing up $F^\sigma_\calI$, cf. 
Figure~\ref{f:C2} for the simplest case $k=2$. 
\begin{figure}[h]
\includegraphics[width=\textwidth]{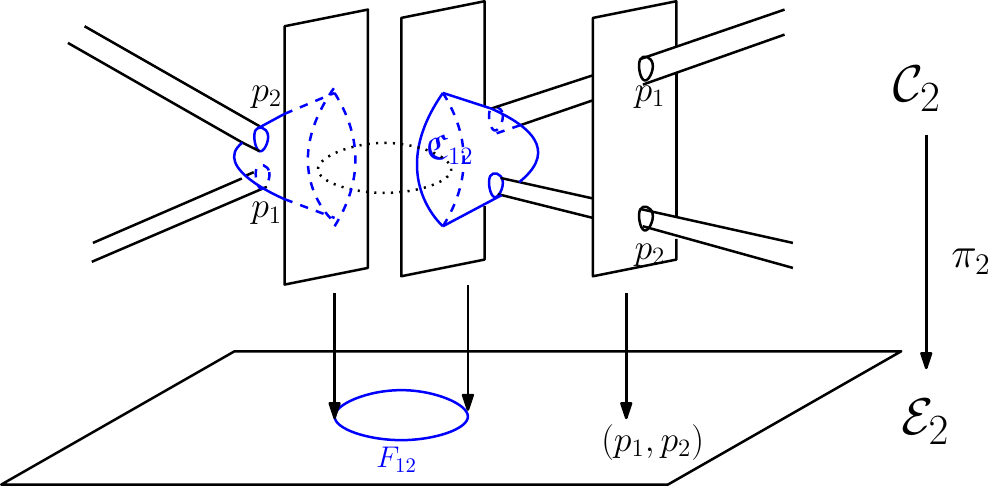}
 \caption{The singular fibration of $\calC_{2}\rightarrow \calE_{2}$. Here we removed the center of mass for $\frakp=(p_{1}, p_{2})$. 
The boundary face in the base, $F_{12}$, is a circle encoding the colliding direction of $p_{1}$ and $p_{2}$.}
 \label{f:C2}
\end{figure}
In this picture, $\frakC_{12}$ is the preimage of the central face $F_{12}$ of $\calE_2$; this fibers over $\bS^1$ (the
`direction of approach' of the pair $p_1, p_2$) and each fiber is a copy of $M_{\frakp}$ and 
$\bS^2_+$ blown up at two points. 

For larger $k$, if $\frakq$ lies on a boundary face of $\calE_{k}$, then the preimage $\pi_k^{-1}(\frakq)$ 
is a tower of hemispheres, each one attached to a previous (or lower) hemisphere at the circle boundary created 
by blowing up a point in that previous hemisphere. The lowest hemisphere in the chain is attached to $M_{\frakp}$ 
at the circle created by blowing up the point where the corresponding points have collided.   Altogether, this tower encodes how subclusters 
of $\frakp$ collide. The images of the nonsingular graphs $\{z=p_{i}\}$ are completely separated, each one intersecting 
one of these hemispheres (or else $M_{\frakp}$ if that point is not part of a cluster). 

We illustrate this further by considering the case $k=3$, see Figure~\ref{figurek=3}. Above a generic point $\frakq\in F_{123}$ the 
fiber $\pi_{k}^{-1}(\frakq)$ is a hemisphere blown up at three points attached to $M_{\frakp}$ at its outer boundary, much 
as in Figure~\ref{f:C2}. However, when $\frakq$ lies on $F_{123} \cap F_{12}$, for example, then the fiber is a tower of 
two hemispheres, cf.\ Figure~\ref{figurek=3} below. 
\begin{figure}[h]
\centering
\includegraphics[width = 0.5\textwidth]{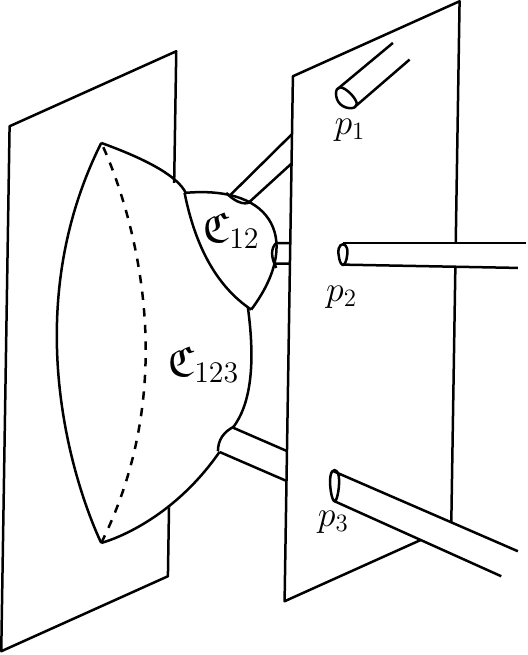}
\caption{One of the singular fibers in $\calC_{3}$, where two of the points collide faster than the third one}
\label{figurek=3}\end{figure}
Here the lower hemisphere, $\frakC_{123}$, is attached as before to $M_{\frakp}$, while the upper 
one, $\frakC_{12}$, is attached to the blowup of the point in $\frakC_{123}$ where $\{z = p_1\}$ and $\{z = p_2\}$ intersect;
note that the submanifold $\{z = p_3\}$ intersects $\frakC_{123}$ but not $\frakC_{12}$. This corresponds to 
the three points coalescing, but with points $1$ and $2$ closer than either is to point $3$. 
When $k$ is even larger, the fibers over points lying in the various edges and corners of $F_{1 \ldots k}$ are more complicated
towers of hemispheres which encode the way the $k$ points coalesce with certain subclusters coalescing faster. 

The (somewhat intricate) combinatorics of the boundary faces and corners of $\calE_{k}$ and $\calC_{k}$ are 
described carefully in~\cite[Chap 2]{MZ}. 

\subsection{Fiber metrics restricting to the boundary faces}
It is proved in~\cite{MZ} that the space $\calC_k$ fully captures the asymptotic behavior of families of flat or hyperbolic
conic metrics on $M$ as the cone points coalesce. More precisely, fix $\vec \beta=(\beta_{1}, \dots, \beta_{k})$ and 
parametrize the family of flat or hyperbolic metrics by elements $\frakq$ of the extended configuration space $\calE_k$.
When $\frakq$ lies in the interior of $\calE_k$, so the cone points are all distinct, then the corresponding metric is 
a metric on the fiber $\pi_k^{-1}(\frakq) = M_{\frakp}$. The main result is that this family of fiber metrics over the
interior of $\calE_k$ extends to a polyhomogeneous family of fiber metrics on $\calC_k$. Over $\mathrm{int}\, \calE_k$, 
this simply asserts that the constant curvature metric depends smoothly on $\frakq$, as already proved in \cite{MW}. 
However, when $\frakq$ lies in some boundary component of $\calE_k$, then $\pi_k^{-1}(\frakq)$ is the union 
of $M_\frakp$ (where now $\frakp = \beta_k(\frakq)$ contains only $k' < k$ distinct points)
and a tower of hemispheres over that $k'$-fold point. The family of fiber metrics restricts to a flat or hyperbolic 
(depending on the initial family) metric on $M_{\frakp}$. On each hemisphere $\frakC_{\calI}$ in this fiber, the restriction
is a flat metric with a certain number of interior conic singularities (at the points where `higher' hemispheres are attached) and 
with a {\it complete} conic singularity at its outer boundary. (Note that each of these metrics is flat regardless 
of whether the initial family is flat or hyperbolic.)   At the conic point (or rather, its $\bS^1$ blowup) in $M_{\frakp}$ 
where the points $\{p_{i}, i\in \calJ\}$ are all equal, the angle parameter equals
$$
\beta_{\calJ}:= 1 + \sum_{i\in \calJ}(\beta_{i}-1).
$$
At any of the other circular boundary components, either on $M_{\frakp}$ or on one of these inner hemispheres, 
where $\frakC_i$ intersects that face, the cone angle parameter is just $\beta_{i}$.  Finally, at the outer boundary of 
each hemisphere in $\frakC_{\calI}$, the metric is asymptotic to the large end of a flat cone with cone angle parameter $\beta_{\calI}$. 

The same description holds for {\it spherical} cone metrics on the fibers of $\calC_k$ so long as all the cone angles are 
less than $2\pi$ and $\vec \beta$ lies in the Troyanov region. (There is a minor exception when the metric above the 
central fiber is a spherical football.)  The restriction of this spherical metric family to each of the hemisphere faces in
$\frakC_{\calI}$ is still a flat metric, exactly as above.  

We shall extend this regularity result in~\S\ref{ss:poly} to include families of spherical cone metrics with at least some of the 
cone angles bigger than $2\pi$.

\section{The linearized Liouville operator} 
Our main analysis involves the Liouville operator
\[
N_{g_0}(u) := \Delta_{g_0} u + K_{g_0} - e^{2u}
\]
(recall that $\Delta = -\mathrm{div}\, \nabla$); solutions to $N_{g_0}(u) = 0$ correspond to 
spherical metrics $e^{2u}g_0$.   In this section we recall the 
basic mapping and regularity properties of its linearization 
\[
L_{g_0}  :=  \Delta_{g_0} - 2.
\]

\subsection{Function spaces}
Given a $k$-tuple of distinct points $\frakp \in \calU \subset M^k$, the blowup $M_{\frakp}$ is a surface
with $k$ boundaries, each a copy of $\bS^1$. Choose a local holomorphic coordinate $z$ near each conic point $p_j$,
with corresponding polar coordinates $(r,\theta)$. A conic differential operator of order $m$ on $M_{\frakp}$ is an 
operator of the form
\[
A = r^{-m} \sum_{j + \ell \leq m} a_{j \ell}(r,\theta) (r\del_r)^j \del_\theta^\ell,
\]
where each $a_{j\ell} \in \calC^\infty(M_\frakp)$.  It is called elliptic (in this conic category) if
$\sum_{j+\ell = m} a_{j\ell} \rho^j \eta^\ell \neq 0$ for $(\rho,\eta) \neq (0,0)$.   In suitable coordinates,
$g_0 = dr^2 + \beta^2 \sin^2 r\,  d\theta^2$ and 
$$
\Delta_{g_0} =  \del_r^2 + \frac{\cos r}{\sin r} \del_r + \frac{1}{\beta^{2}\sin^2 r} \del_\theta^2,
$$
or equivalently, 
$$
\Delta_{g_0} = r^{-2}\left( (r\partial_{r})^{2}+\beta^{-2} \partial_{\theta}^{2}  + \ldots \right) ,
$$
where the remainder terms are smooth multiples of $r^2\del_r$ and $r \del_\theta$, hence lower order.
Thus $\Delta_{g_0}$ is a conic elliptic operator. 

The detailed theory of conic elliptic operators is discussed in \cite{MW, MZ}, and described in complete detail in
\cite{Melrose-APS}, \cite{M-edge} and \cite{GilKrainerMendoza}. We review here the mapping and 
regularity properties of $\Delta_{g_0}$ on weighted $b$-H\"older spaces, and the closely related definition of 
the H\"older-Friedrichs domain. 

The most convenient scale of function spaces for conic operators are those with certain dilation invariance properties. 
\begin{definition}[$b$-H\"older spaces]
The space $\calC^{0,\al}_b(M_{\frakp})$ consists of all bounded functions on $M_{\frakp}$ which are in $\calC^{0,\alpha}$ 
in the interior of $M_{\frakp}$ and such that near each $p_j$, 
\[ 
[u]_{b; 0, \alpha} := \sup_{0 < R < R_0} \sup_{R \leq r,r' \leq 2R \atop (r,\theta) \neq (r',\theta')} \frac{|u(r,\theta) - u(r',\theta')|R^\al}{|(r,\theta)-(r',\theta')|^\al} \leq C,
\] 
with associated norm
\[
||u||_{b; 0, \alpha} = ||u||_{L^\infty} + [u]_{b;0,\alpha}.
\]
The space $\calC^{m,\al}_b(M_{\frakp})$ consists of all functions $u$ such that near each $p_j$,  $(r\del_r)^j \del_\theta^\ell u \in
\calC^{0,\al}_b$ when $j+\ell\leq m$. Finally, $r^\mu \calC^{m,\al}_b(M_{\frakp}) = \{u = r^\mu v: v \in \calC^{m,\al}_b(M_{\frakp})\}$. 
\end{definition}

Directly from the definition, 
\begin{equation}
\Delta_{g_0}:  r^\mu \calC^{m+2,\alpha}_b \longrightarrow r^{\mu-2} \calC^{m,\alpha}_b
\label{Lapbdd}
\end{equation}
is bounded for every $m \in \NN$ and $\mu \in \RR$.  

There are two possible choices for the space of `smooth' functions in this setting.
\begin{definition}[Conormality]  The space of conormal functions (of order $\mu$) is the intersection
\[
\calA^\mu(M_\frakp) =  \bigcap_{m \geq 0} r^\mu \calC^{m,\alpha}_b(M_\frakp) = \{u: |(r\del_r)^j \del_\theta^\ell u| \leq C_{j,\ell} r^\mu, \ 
\forall\ j, \ell \geq 0\}
\]
\end{definition}
\begin{definition}[Polyhomogeneity]
An index set $I$ is a countable discrete set $\{\gamma_i,N_i\} \subset \CC \times \NN$ with $\Re \gamma_i \to \infty$. 
A function $u$ is called polyhomogeneous with index set $I$ if $u \in \calA^\mu(M_\frakp)$ for some $\mu$ and 
\[ 
u \sim \sum_{i} \sum_{\ell = 0}^{N_i} u_{i,\ell}(\theta) r^{\gamma_i} (\log r)^\ell.
\] 
This is an asymptotic expansion in the classical sense, in that the difference between $u$ and any finite partial
sum of the terms on the right lies in $\calA^{\mu+N}$ where $N \to \infty$ depends on the largest index
in the partial sum. 
\end{definition}

We may also define the $b$-H\"older spaces, as well as the spaces of conormal and polyhomogeneous
functions on any compact manifold with corners $X$.  In this more general setting we replace the vector fields
$r\del_r$ and $\del_\theta$ which appear in the definitions above by the space of $b$-vector fields, $\calV_b(X)$,
which is the space of all smooth vector fields on $X$ which are tangent to all boundary faces. Thus if $q \in X$
lies on a corner of codimension $n$, then there is a coordinate system $(x_1, \ldots, x_n, y_1, \ldots, y_m)$,
$n+m = \dim X$, with each $x_i \geq 0$ and $y_j \in (-\epsilon, \epsilon)$. Near $q$, 
\[
\calV_b(X) = \{ V = \sum_{i=1}^n a_i(x,y) x_i \del_{x_i} + \sum_{j=1}^m  b_j(x,y) \del_{y_j},\ \ a_i, b_j \in \calC^\infty(X)\}.
\]
Then $\calC^{0,\alpha}_b(X)$ is defined via a H\"older seminorm similar to the one above, which is invariant
under the partial dilation $(x,y) \mapsto (\lambda x, y)$, and 
\begin{equation}\label{e:Holder}
\calC^{m,\alpha}_b(X) = \{u: V_1 \ldots V_\ell u \in \calC^{0,\alpha}_b
\ \forall\ \ell \leq m\ \mbox{and}\ V_j \in \calV_b(X)\}.
\end{equation}
If $H_1, \ldots, H_N$ are the boundary hypersurfaces of $X$, then we denote by $\rho_j$, $ j = 1, \ldots, N$, 
a smooth function which satisfies $\rho_j > 0$ on $X \setminus H_j$, and $\rho_j = 0$, $d\rho_j \neq 0$ 
on $H_j$. These are called boundary defining functions.  Using multi-index notation, we write
\[
\rho^\mu \calC^{m,\alpha}_b(X) =  \{u = \rho_1^{\mu_1} \ldots \rho_N^{\mu_N} v: v \in \calC^{m,\alpha}_b(X)\}
\]
and
\[
\calA^\mu(X) = \bigcap_{m \geq 0} \rho^\mu \calC^{m,\alpha}_b(X).
\]
Finally, an index family $\calI$ is an $N$-tuple of index sets $(I_1, \ldots, I_N)$, and $u$ is 
polyhomogeneous with index family $\calI$ on $X$ if it is conormal and has asymptotic expansion
with index set $I_j$ near $H_j$, with all coefficients conormal on $H_j$.  It is not hard to
prove that under these conditions, $u$ has a product-type expansion at the corners of $X$. 

\subsection{Indicial roots and mapping properties}
Sharp mapping and regularity for $L_{g_0}$ (or indeed any other conic elliptic operator), are naturally captured by
these spaces.  These properties are stated in terms of the set of indicial data associated to this operator.

\begin{definition}[Indicial roots]
The number $\gamma \in \CC$ is called an indicial root of multiplicity $N$ for a conic elliptic operator $A$ at $p_j$
if there exists some $\phi \in \calC^\infty(\mathbb{S}^1)$ such that $A(r^\gamma (\log r)^{N-1} \phi(\theta)) = 
\calO(r^{\gamma - 1}(\log r)^{N-1})$ (as opposed to the expected $\calO(r^{\gamma-2}(\log r)^{N-1})$), but this estimate fails if $N-1$ 
is replaced by $N$.

The set of functions $r^{\gamma} (\log r)^\ell \phi$, $\ell = 0, \ldots, N-1$ for which this improved estimate holds is called 
the indicial kernel of $A$ (at $p_j$ and for the indicial root $\gamma$). 

Write $\Gamma(A, p_j)$ for the set of all indicial roots of $A$ at $p_j$ and $\tilde{\Gamma}(A, p_j)$ for
the set $\{(\gamma,N-1): \gamma \in \Gamma(A, p_{j}),\ N = \mbox{logarithmic multiplicity at}\ \gamma\}$.
We often omit the $p_j$ in this notation to denote the union of these sets over all $p_j$.
\end{definition} 

The indicial roots for $L_{g_0}$ are straightforward to compute, cf. \cite[Section 5.1]{MW}: 
\begin{lemma}
$\Gamma(L_{g_0}, p_j)=\{\frac{k}{\beta_{j}}: k\in \ZZ\}$. The value $0$ is an indicial root of multiplicity two 
with indicial kernel $\{1, \log r\}$, while the other indicial roots have multiplicity $1$ and indicial kernel 
$\{r^{k/\beta_j } e^{\pm i k \theta}\}$.
\end{lemma}

We now state the first basic mapping property of $L_{g_0}$: 
\begin{proposition}\cite[Proposition 9]{MW}
Suppose $\mu \not\in \Gamma(L_{g_0})$ and denote by $K_{-\mu}$ the nullspace of $L_{g_0}$ on 
$r^{-\mu}\calC^{2,\al}_b(M_{\frakp})$. Then $K_{-\mu}$ is finite dimensional and for any $f \in r^{\mu-2}\calC^{0,\alpha}_b$, 
there exists an element $h \in K_{-\mu}$ such that $Au = f - h$ for some $u \in r^{\mu}\calC^{2,\al}_b$. 
In particular, if $K_{-\mu} = \{0\}$, then 
\label{surjuptocokernel}
\begin{equation}
L_{g_0}: r^{\mu}\calC^{2,\al}_b(M_{\frakp}) \rightarrow r^{\mu-2}\calC^{0,\al}_b(M_{\frakp}).
\label{Lsurj}
\end{equation}
is surjective.
\label{pr:duality}
\end{proposition}

We require an extension of this result, motivated by the following consideration.  Suppose $\mu$ is a weight such that \eqref{Lsurj} is
surjective.  If $\mu < 0$, this result may be of limited use in the nonlinear problem simply because the 
Liouville operator does not act nicely on functions which are unbounded near $r=0$. However, if the
right hand side does not blow up as quickly as $r^{\mu-2}$ then we can say more:
\begin{proposition}
Suppose that $\mu' > \mu$, with neither value an indicial root, and $L_{g_0}u = f$ for some $f \in r^{\mu'-2}\calC^{0,\alpha}_b$ 
and $u \in r^\mu \calC^{2,\alpha}_b$. Then 
\[
u = \sum_{j=J_1}^{ J_2}  r^{j/\beta} (a_j \cos j\theta + b_j \sin j\theta) + \tilde{u}
\]
for some constants $a_j, b_j$. Here $J_1$ is the smallest integer greater than or equal to $\mu \beta$ and  $J_2$ is the largest 
integer less than $\mu' \beta$; if $J_1 \leq 0 \leq J_2$ then the term with $j=0$ should be replaced by $a_0 + b_0 \log r$. Finally,
the remainder term $\tilde{u}$ lies in $r^{\mu'} \calC^{2,\alpha}_b$. 
\label{partexp}
\end{proposition}
This is a regularity statement: if the right hand side decays faster than expected, then the
solution has a partial expansion as $r \to 0$. 

A special case of particular importance is when $\mu = -\epsilon < 0$ and $\mu' = 2$ (so $\mu'-2 = 0$). 
We are thus searching for a solution $u \in r^{-\epsilon} \calC^{2,\alpha}_b$ to $L_{g_0} u = f$, where $f \in \calC^{0,\alpha}_b$;
Since $K_\epsilon$, the nullspace of $L_{g_0}$ on $r^\epsilon \calC^{2,\alpha}_b$, is trivial, Proposition \ref{surjuptocokernel} may
be applied to obtain that \eqref{Lsurj} is surjective, hence we may always find a solution $u$ to $L_{g_0}u = f \in \calC^{0,\alpha}_b$. 
Proposition \ref{partexp} states that 
\[
u = a_0 + b_0 \log r + \sum_{j=1}^J (a_j \cos j\theta + b_j \sin j\theta) r^{j/\beta} + \tilde{u}, 
\]
where $J$ is the largest integer strictly less than $2\beta$ and $\tilde{u} \in r^2 \calC^{2,\alpha}_b$. 

A familiar classical construction is to characterize the domain of $L_{g_0}$ as an unbounded operator on $L^2(M_\frakp)$.  
More specifically, $L_{g_0}$ is symmetric and semibounded on $\calC^\infty_0( M \setminus \frakp)$, and hence there 
is a canonical {\it Friedrichs domain} $\calD_{\mathrm{Fr}}(L_{g_0})$, which is a dense subspace in $L^2$ on which $L_{g_0}$ 
is self-adjoint and has the same lower bound.  This is the set of all functions $w\in L^2$ such that both $\nabla w, L_{g_0} w  \in L^2$.
From the asymptotic expansion above, this last condition implies that the $\log r$ term is absent.  Accordingly, we make 
the following
\begin{definition}
The H\"older-Friedrichs domain of $L_{g_0}$ is the space
$\calD^{m,\al}_{\mathrm{Fr}}(L_{g_0}) = \{u \in \calC^{m,\al}_b(M_{\frakp}):  L_{g_{0}} u \in \calC^{m,\al}_b(M_{\frakp}) \}.$
\label{HF}
\end{definition}
By the results above, 
\[
u \in \calD^{m,\al}_{\mathrm{Fr}}(L_{g_0}) \Longrightarrow u = a_0 + \sum_{j=1}^J (a_j \cos j\theta + b_j \sin j\theta) r^{j/\beta} + \tilde{u}, 
\]
where $J$ is as before and $\tilde{u} \in r^2 \calC^{m+2,\al}_b$. 

\subsection{Deformation theory -- the unobstructed case}
Let $g_0$ be a spherical cone metric with conic data $(\fraks(g_0), \frakp, \vec \beta)$, and let $\SCM_k$ 
denote the set of all spherical cone metrics with the same $\fraks(g_0)$. We show in this section that 
if $2 \not\in \mathrm{spec}\,(\Delta_{g_0})$, then the map $\SCM_k \to M^k \times \RR^k_+$ is a local 
diffeomorphism near $g_0$. In other words, the space of spherical cone metrics $g$ near $g_0$ with a 
fixed unmarked conformal class is smoothly parametrized by the data $(\frakp', \vec \beta')$ near 
$(\frakp, \vec \beta)$.  It is also the case that the dependence on the underlying unmarked conformal class
$\fraks$ is smooth. This argument is the same as the one in \cite{MW}.  That paper assumes that all 
cone angles are less than $2\pi$, in which case it turns out to be automatic that $L_{g_0}$ is invertible.  
If some or all cone angles are greater than $2\pi$, we must assume the invertibility of this operator to reach 
the same conclusion.  We review these arguments in this section and prove the
\begin{theorem}\label{t:no2}
Let $g$ be a spherical conic metric with conic data $(\fraks, \frakp, \vec \beta)$, and suppose that 
$2 \not\in \mathrm{spec}\,(\Delta_g)$. Then for each fixed constant curvature metric 
$\fraks'$ sufficiently near $\fraks$ representing a slice in the space of unmarked conformal structures, 
there exists a neighborhood $\calV$ of $(\frakp, \vec \beta)$ in 
$\mathrm{int}\, \calE_k \times \RR^k_+$ and a neighborhood $\calW$ containing $g$ in the space of spherical 
conic metrics with the same unmarked conformal class, such that the map assigning to $g' \in \calW$ its 
conic data $(\frakp', \vec \beta')$ is a diffeomorphism $\calW \to \calV$.   This map also depends
smoothly on $\fraks'$. 
\end{theorem}

The proof relies on a preliminary computation which provides a link between the geometric and analytic parts
of this result.   Namely, we compute the derivative of a family of metrics $g(\epsilon)$ with varying conic data 
$(\frakp(\epsilon), \vec \beta(\epsilon))$. The relevant information is entirely local so we may as well work in a 
disk around one conic point; furthermore, the fact that the metrics are spherical rather than flat 
only adds higher order perturbations to the answer below. Thus it suffices to consider the family of flat metrics
\begin{equation}
g(\epsilon) =  |z + \epsilon w|^{2\beta(\epsilon)-2} |dz|^2.
\label{modelmetric}
\end{equation}
Here $\beta(\epsilon)$ is any smooth function with $\beta(0) = \beta$ and $w$ is a fixed complex number indicating
the direction of motion of the family of conic points. Then 
\[
g'(0) = \bigg(2 \beta'(0) \log |z| + 2(\beta(0) -1)\Re (w/z) \bigg) |z|^{2(\beta-1)} |dz|^2.
\]
In particular, 
\begin{eqnarray*}
& w = 0 \Rightarrow  g'(0) = 2\beta'(0)\log |z| g(0),   \\ 
& \beta'(0) = 0 \Rightarrow  g'(0) = 2 (\beta-1) \Re (w/z) g(0).
\end{eqnarray*}
We have used complex coordinates here, but switching to $r = |z|^{\beta}/\beta$, we have
\[
g'(0) = \left( c_0 \log r + (c_1 \cos \theta + c_1' \sin \theta) r^{-1/\beta} \right) g(0).
\]
The constants $c_0, c_1, c_1'$ depend on $\beta(0), \beta'(0)$ and $w$; the latter two encode the angle at which the singular
point moves in this deformation. Now recall that $\log r$ and $r^{-1/\beta} \cos \theta$, $r^{-1/\beta}\sin \theta$ are 
model solutions for the indicial problem.  This calculation shows that 
these particular solutions to the indicial equation arise as derivatives of certain (local) one-parameter 
families of conic metrics.

We capitalize on this as follows.  Choose local holomorphic coordinates near each $p_j$ in $\frakp(0)$. 
The neighborhood $\calV$ around $(\frakp, \vec \beta)$ in the space of conic data is defined as the 
product of $k$ copies of $B_w^2(0) \subset \RR^2$ and a ball $B^k_\gamma(0)$ in $\RR^k$ around $\vec \beta(0)$. 
A point 
\[
\zeta = (w_1, \ldots, w_k, \gamma_1, \ldots, \gamma_k) \in \calV
\]
corresponds to conic data 
\[
((p_1 + w_1, \ldots, p_k + w_k), (\beta_1 + \gamma_1, \ldots, \beta_k + \gamma_k)).
\]

Next, for each $\zeta \in \calV$, choose a conic (but not necessarily spherical) metric $\tilde{g}(\zeta)$
which has conic data $\zeta$. For example, we can glue together a fixed metric outside the union
of balls around the $p_k$ and some varying model metrics as in \eqref{modelmetric}.   We may
also choose a smooth family of diffeomorphisms $F_\zeta: M \to M$ with the following properties:
$F_\zeta$ is the identity outside some small neighborhood of the points $p_j$ and $F_\zeta(p_j) = p_j + w_j$.
(We may as well assume that the $F_\zeta$ only depend on the $w$ but not the $\gamma$ coordinates of $\zeta$.) 
Finally, define $g(\zeta) = F_\zeta^* \tilde{g}(\zeta)$.

The point of all of this is that $g(\zeta)$ is a smooth family of conic metrics which represents the conic data
$\zeta$, but where, using the diffeomorphism action on $M$, we have arranged that the cone points remain
fixed.   We refer to \cite{MW} for an explanation of what this means in terms of the Teichm\"uller theory.
One must also modify these local families into families which leave the underlying uniformizing
metric $\fraks(g(\zeta))$ fixed, but this is straightforward and we omit details. 

We have now reduced the problem to proving that there exists a family of functions $u(\zeta)$ which
lies in one of the (possibly weighted) H\"older spaces discussed earlier, which solves
\[
\Delta_{g(\zeta)} u(\zeta) +  K_{g(\zeta)} - e^{2 u(\zeta)} = 0, \quad u(0) = 0,
\]
and which depends smoothly on $\zeta$.   This is a straightforward application of the implicit function theorem, 
by virtue of the results of the last subsection.  Namely, 
consider the Liouville operator with base metric $g(\zeta)$ as a nonlinear operator
\[
\calN: \calV \times \calD^{0,\alpha}_{\Fr} \longrightarrow \calC^{0,\alpha}_b, \qquad
\calN(\zeta, u) = \Delta_{g(\zeta)} u + K_{g(\zeta)} - e^{2u}.
\]
This is a smooth mapping and by Propositions \ref{pr:duality} and \ref{partexp} and
the computation at the beginning of this proof, the linearization of this map is surjective. Noting
further that the linearization only in the second ($u$) slot is injective, we conclude that the kernel
projects isomorphically to the tangent space of $\calV$.  This uses our main hypothesis 
that $2$ does not lie in the nullspace of $\Delta_{g(0)}$; if it were not to hold, there would be an extra 
cokernel.   Note finally that we may also let the underlying nonsingular constant curvature metric vary in some slice 
representing the family of unmarked conformal structures; the family of solutions clearly 
depends smoothly on this extra parameter too. Altogether, this proves the local deformation theorem and 
the smoothness of the moduli space of spherical conic metrics under this spectral hypothesis.
  
\section{The locus of degenerate spherical cone metrics} 
The last section explains the importance of understanding when $2 \notin \mathrm{spec}\,(\Delta_g)$, 
or equivalently, when $L_g$ is invertible, for a spherical conic metric $g$.  We begin our discussion
of this case. 

We first recall a result from \cite{MW}: 
\begin{proposition}
Let $(M_{\frakp}, g)$ be a spherical conic metric with all cone angles in $(0,2\pi)$. Let $\lambda_1$ be the  
first nonzero eigenvalue of the Friedrichs extension of $\Delta_{g}$. Then $\lambda_1 \geq 2$, 
with equality if and only if $M$ is either the round $2$-sphere or else the spherical football 
(with constant Gauss curvature $1$). Apart from these cases, 
\begin{equation}
L_{g}: \calD^{m,\al}_{\mathrm{Fr}}(M_{\frakp}) \to \calC^{m,\al}_b(M_{\frakp})
\label{Frmap}
\end{equation}
is invertible. 
\label{eigenbound}
\end{proposition}

We also record a small generalization of this.  The Liouville energy (referenced in \S 2) is the
functional
$$
E(u)=\frac{1}{2}\int (|\nabla u|^{2} + 2K_{g_{0}}u)dV_{g_{0}}-\frac{1}{2}\log \int e^{2u}\, dA_{g_{0}},
$$
cf. \cite{Ma2}.  The Euler-Lagrange equation for $E$ reduces to the Liouville equation \eqref{Liouville},
while the Hessian of $E(u)$ equals $\Delta_{g_{0}}-2-P_0$; here $P_0$ is the $L^2$ orthogonal projection off 
the constants (i.e., the lowest eigenmode for $\Delta_{g_0}$).  

In the approach to existence discussed in \S 2 using the calculus of variations, the direct
method to find minimizers of $E$ is successful in the `subcritical' case defined by the condition
$\chi(M, \vec \beta)<\min\{2, 2\min_{j}\beta_{j}\}$ (where the background metric $g_0$ has
angle parameters $\vec \beta$). This can occur even if some of the cone angles are greater
than $2\pi$. In this subcritical case, $E$ is bounded below and coercive, and
there is a unique minimizer $u$ which is nondegenerate; the metric $e^{2u}g_{0}$ is then 
spherical. In this case $L_g - P_0$ is invertible and of Morse index $0$, so in the language of the 
proposition above, $\lambda_1 > 2$.

\subsection{Metrics with reducible monodromy}\label{ss:coaxial}
Our first goal is to show that there must be many spherical cone metrics for which $2$ lies
in the (Friedrichs) spectrum of $\Delta_g$. Recall that the monodromy group of a spherical cone metric is defined by its developing map and is contained in $\operatorname{PSU}(2)$. If the monodromy is contained in a subgroup which lies in $\operatorname{U}(1)$ (up
to conjugation), then the metric is called {\em reducible}~\cite{CWWX} or {\em coaxial}~\cite{MP, Ere1}. In particular, any metric obtained by a branched cover of the sphere has trivial monodromy and is therefore reducible.  
\begin{proposition}
For any bounded set $B \subset \MP_{k}$   defined in \eqref{MonPan}, there exists a spherical cone metric $g$ on $\bS^{2}$ with cone angle 
parameters $\vec \beta \notin B$ such that $2 \in \mathrm{spec}\,(\Delta_g)$. When $k\geq 5$, $\vec \beta$ can 
be chosen to be in the interior of $\MP_{k}$.
\end{proposition}
\begin{proof}
If a spherical conical metric $g$ is reducible, then $2 \in \mathrm{spec}\,(\Delta_g)$ and at least one such eigenfunction is generated by the developing map, see~\cite{XuZhu}. From~\cite{Ere1}, the angle condition that gives a reducible metric is unbounded, that is, for any bounded $B\subset \MP_{k}$ there exists at least one $\vec \beta$ outside $B$ that admits a reducible metric. In detail, there exists a reducible metric with angles $\vec \beta$ if and only if one of the following holds:
\begin{itemize}
\item All $\beta_{i}\in \NN$, $d_{\ell^{1}}(\vec \beta-\vec 1, \ZZ_{odd}^{n})=1$, and  $2\max_{i}(\beta_{i}-1)\leq \sum_{i=1}^{n} (\beta_{i}-1)$. In this case such a metric is a branched cover of $\bS^{2}$;
\item (Up to reordering) there exists $1< m <n$ such that $\beta_{1}, \dots \beta_{m}\notin \NN$, $\beta_{m+1}, \dots, \beta_{n}\in \NN$. Moreover $\vec \beta$ satisfies the following ``coaxial conditions'':
\begin{itemize}
\item There exists $\{\epsilon_{i}\}_{i=1}^{m}$ with $\epsilon_{i}\in\{\pm 1\}$ such that 
$$k'=\sum_{i=1}^{m}\epsilon_{i}\beta_{i}\geq 0.$$
\item  $k''=\sum_{i=m+1}^{n} \beta_{i}-n-k'+2\geq 0$ and $k''$ is even. 
\item If there exist integers $\{b_{i}\}$ whose greatest common divisor is 1 and such that $(\beta_{1}, \dots, \beta_{m}, \underbrace{1,\dots, 1}_{k'+k''})=\eta(b_{1}, \dots, b_{m+k'+k''})$, then 
$$2\max_{i=m+1}^{n}\beta_{i}\leq \sum_{i=1}^{m+k'+k''} b_{i}.$$
\end{itemize}
\end{itemize}
For the second condition above, we can also see that when $k\geq 5$ such metrics exist in the interior of  $\MP_{k}$.
\end{proof}

\subsection{Spherical cone metrics with a large number of small eigenvalues}
We now prove some more general results which use a spectral flow argument to show that there 
should be many examples of spherical cone metrics with $\vec \beta$ arbitrarily large
for which $\Delta_g$ has eigenvalue $2$. There are two  main steps.  The first is to show that for any $N > 0$, there
exists some $\vec \beta \in \MP_k$ and a spherical cone metric $g$ with this conic angle data
such that $\Delta_g$ has at least $N$ eigenvalues less than $2$.  The second is to show that if the space of spherical 
cone metrics has only finitely many connected components, then by spectral flow one can find such metrics with eigenvalue 2. 
If $2$ never lies in the spectrum of Laplacians of conic surfaces with `sufficiently large' cone angles, then one can find
$g_0$ and $g_1$ in the same connected component, satisfying that $\Delta_{g_j}$ has $N_j$ eigenvalues less than $2$, $j=1,2$, 
and there is a continuous family of spherical conic metrics between $g_0$ and $g_1$.  A simple spectral flow argument 
leads to a contradiction if $N_0 \neq N_1$, hence there must be a nonempty locus of spherical cone metrics
with large cone angles and with $2$ lying in  the spectrum of the Laplacian. Therefore, either the space of spherical cone metrics has infinitely many connected components, or there are infinitely many codimension-one strata with $2\in \mathrm{spec}\,(\Delta_g)$.

We begin with the analysis of the football with arbitrary cone angle.
\begin{lemma}\label{l:football}
The eigenvalues of $\Delta_g$ on the spherical football with cone angle $2\pi\beta$ are 
\begin{equation}\label{e:footballeigen}
\{(j/\beta + \ell)(j/\beta + \ell+1):\ j, \ell \in \NN \}.
\end{equation}
The eigenspace is simple when $j = 0$ since $\log r$ does not lie in the Friedrichs domain, with eigenfunction 
$P_{\ell}^{0}(\cos(r))$, while if $j > 0$, then the eigenspace is two dimensional and spanned by 
$P_{\ell}^{j/\beta}(\cos r) \cos (j\theta)$ and $P_{\ell}^{j/\beta}(\cos r) \sin (j\theta)$.
Here $P_{\ell}^{\nu}$ is the associated Legendre function of order $\ell$ and degree $\nu$.
\end{lemma}
\begin{proof}
This is an explicit computation. Since $g= dr^{2}+\beta^{2}\sin^{2}r d\theta^{2}$, we seek solutions of 
$$
(\partial_{r}^2 +\frac{\cos r}{\sin r} \partial_{r}+\beta^{-2}\frac{1}{\sin^{2} r}\partial_{\theta}^{2}+\lambda) u=0.
$$
Inserting $u=R(r)e^{ij\theta}$ yields
\begin{equation}
\sin^{2}r \, R''(r)+\sin r \cos r \, R'(r) + \lambda \sin^{2} r R(r)-\frac{j^{2}}{\beta^{2}}R(r)=0,
\end{equation}
or, changing variable to $t=\cos (r)$, 
$$
(1-t^{2})\, R_{tt} - 2tR_{t}+[\lambda-\frac{j^{2}}{\beta^{2}(1-t^{2})}]R=0,\ \ t \in [-1,1].
$$
A basis of solutions when $\lambda > 0$ consists of the first and second associated Legendre functions 
$P_{\ell}^{j/\beta}(t)$ and $Q_{\ell}^{j/\beta}(t)$.  In order that the solution lies in the Friedrichs domain, one
of the following must hold: 
\begin{itemize}
\item $j=0$, and $R(-1), R(1) <\infty$, or
\item $j>0$, and $R(-1)=R(1)=0$.
\end{itemize}
In the first case, the equation becomes
$$
(1-t^{2})R_{tt} - 2tR_{t}+\lambda R=0.
$$
This has a solution which is regular at both $t = \pm 1$ only when $\lambda = \ell(\ell+1)$, $\ell \in \NN$;
the solution itself is $P_{\ell}^{0}(t)$. In particular, $\lambda = 2$ when $\ell=1$ and the eigenfunction is $u=\cos r$.

In the second case, when $j>0$, then using properties of $P_{\ell}^{k/\beta}$ and $Q_{\ell}^{k/\beta}$, we get that
$\lambda = (j/\beta + \ell)(j/\beta + \ell+1)$, $\ell \in \NN$ and the only admissible solution is $P_{\ell}^{j/\beta}(t)$; 
the eigenspace is spanned by $P_{\ell}^{\mu}(\cos r) \cos (j\theta)$, $P_{\ell}^{\mu}(\cos r) \sin (j\theta)$.
\end{proof}

\begin{lemma}
There are $2 + 2[\beta]$ eigenvalues $\lambda$  less than or equal to $2$ for the Friedrichs extension of $\Delta_{g_0}$ 
for a football with angle $2\pi\beta$.  
\end{lemma}
\begin{proof}
We count the numbers in~\eqref{e:footballeigen} with $\lambda \leq 2$. These occur only when $j =0$ and $\ell = 0, 1$, 
or else $j > 0$, $\ell = 0$ and $(j/\beta)(j/\beta + 1) \leq 2$, which holds if and only if $j/\beta \leq 1$. Each of these have
multiplicity $2$. This leads to $2+2[\beta]$ eigenvalues in $[0,2]$. 
\end{proof}

\begin{lemma}\label{l:eigenN}
Fix $k \geq 2$ and a bounded set $B \subset \MP_k$ of
admissible cone angles. Then for any $N \in \NN$ there exists a spherical cone metric $g$ with cone angle
parameters $\vec \beta \not\in B$ and with at least $N$ eigenvalues of $\Delta_g$ less than $2$. 
\end{lemma}
\begin{proof}
The arguments for the cases $k=2, 3, 4$ and $k \geq 5$ are somewhat different.

When $k=2$ the preceding Lemma shows just this: we simply take a football with cone angle $2\pi\beta$ 
where $[\beta]>\frac{N-2}{2}$. 

Next, any spherical cone surface with $k=3$ conic points has a $\ZZ_2$ reflection symmetry, or in other words, is
obtained by doubling a spherical triangle, see~\cite{Ere2}. Thus we need only show that there exists a spherical triangle 
with at least $N$ Dirichlet eigenvalues less than $2$; the odd reflections to the doubled surface of the corresponding
eigenfunctions will be eigenfunctions in the Friedrichs domain of the cone surface.  We consider here a spherical
triangle with angles $\frac{\pi}{2}, \frac{\pi}{2}, \pi\beta$ where $\beta \gg 1$.    Using that the double of this
triangle across the side connecting the two right angles produces `half' of a football, we see that the
the Dirichlet eigenvalues of this triangle are 
\begin{equation}\label{e:k3}
\{ (j/\beta + 2\ell) (j/\beta +2\ell + 1): j, \ell \in \NN, \ j \geq 1\}.
\end{equation} 
Hence if $[\beta]>N$, then at least $N$ of these are less than $2$. 

For $k=4$, consider the shape obtained by gluing $4$ spherical triangles as in Figure~\ref{f:k4eigen2}: 
two of these are triangles with angles $(\pi/2, \pi/2, \pi\beta)$, and the other two have angles $(\pi/2, \pi/2, \pi/2)$. 
The sides that are matched have the same length (the only sides that do not have length $\pi/2$ are
those which connect the two right angles). This yields a spherical cone polygon with angles 
$(\pi(\beta+1), \pi(\beta+1), 3\pi/2, 3\pi/2)$, see Figure~\ref{f:k4eigen2}.
\begin{figure}[h]
\centering
 \includegraphics[width=0.7\textwidth]{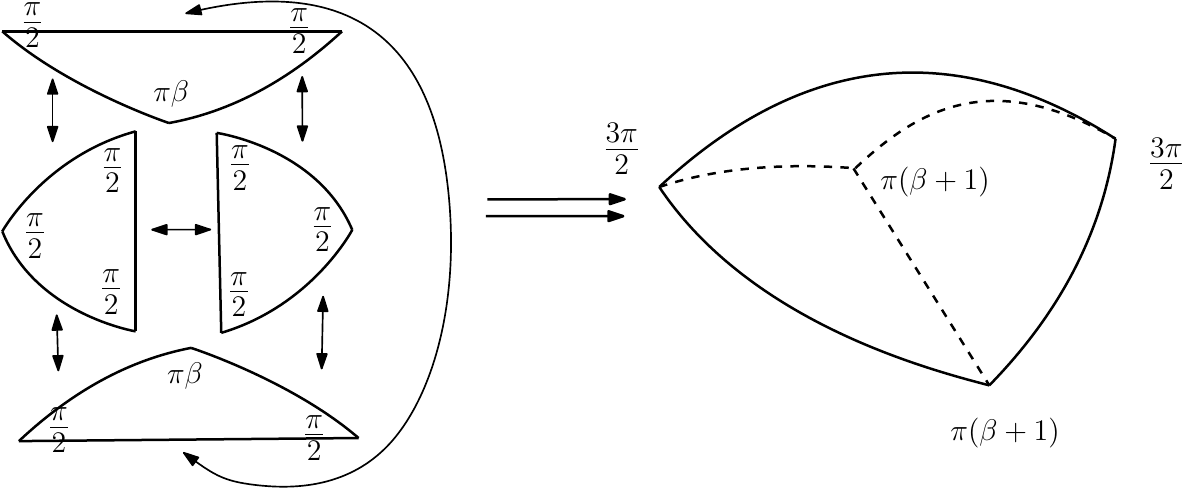}
 \caption{A spherical metric with 4 cone points $(\pi(\beta+1), \pi(\beta+1), 3\pi/2, 3\pi/2)$ with at least $N$ eigenvalues below 2}
 \label{f:k4eigen2}
\end{figure}
If $\beta$ is chosen so that the spherical triangle with one angle $\pi \beta$ has at at least $N$ Dirichlet eigenvalues less than $2$,
then this new surface has an $N$-dimensional space of functions spanned by the the functions $F_j$ which equal
the Dirichlet eigenfunctions $f_j$ on the two large triangles and $0$ on the two smaller triangles.  These functions are 
all in $H^1$ and have Dirichlet energy less than $2$.  By the minimax characterization of eigenvalues, there
must be at least $N$ eigenvalues on the spherical cone surface less than $2$. 

Finally, suppose $k\geq 5$. We use a different gluing here. Start with a football with cone angle $2\pi \beta$. By the explicit expression 
of the eigenfunctions in Lemma~\ref{l:football}, if $\beta$ is large enough, there exists $N$ eigenfunction 
$f_{1},\dots, f_{N}$ each with eigenvalue less than $2$, which vanish on a `meridian' of this football (i.e., a geodesic
curve connecting the two cone points). For example, take $f_{j}=P_{0}^{j/\beta}(\cos r)\sin(j\theta), \ j\leq [\beta]$, which
vanishes along the curve $\{\theta=0\}$. 

Now suppose that $L > 0$ is quite small and choose a slit of length $L$ in this meridian. Choose $(\beta_{1}, \dots, \beta_{k-2})$ so that the vector $\vec \beta=
(\beta, \beta, 1+\beta_{1}, 1+\beta_{2}, \beta_{3},\dots, \beta_{k-2})\in \MP_{k}$.  There exists a spherical cone metric,
obtained by gluing together two identical spherical polygons, with cone angles $2\pi(\beta_{1}, \dots, \beta_{k-2})$ at points 
$(p_1, \ldots, p_{k-2})$ arranged along the equator such that $\mathrm{dist}(p_1, p_2) = L$. 
Cut along the slit between these two points and glue this surface to the football. This new spherical cone surface
has cone angle parameters $\vec \beta=(\beta, \beta, 1+\beta_{1}, 1+\beta_{2}, \beta_{3},\dots, \beta_{k-2})$, 
see Figure~\ref{f:k5eigen2}.
\begin{figure}[h]
\centering
 \includegraphics[width=0.7\textwidth]{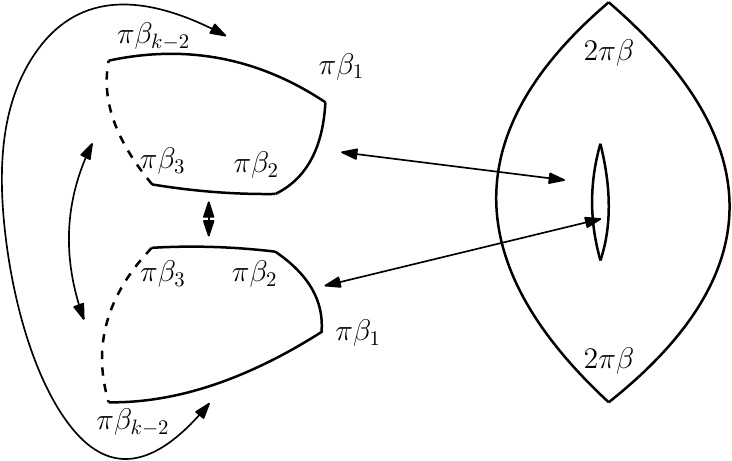}
 \caption{A spherical metric glued from three pieces such that it has $k\geq 5$ cone points $2\pi(\beta, \beta, 1+\beta_{1}, 1+\beta_{2}, \beta_{3},\dots, \beta_{k-2})$ and at least $N$ eigenvalues below 2}
 \label{f:k5eigen2}
\end{figure}

Now we proceed much as before. The eigenfunctions $f_i$ on the football which vanish on the equator extend by $0$
to $H^1$ functions $\tilde f_{i}$ on this new surface. Since
$$
\frac{\|\nabla\tilde f_{i}\|_{L^{2}}^{2}}{\|\tilde f_{i}\|^{2}_{L^{2}}}=\frac{\|\nabla f_{i}\|_{L^{2}}^{2}}{\| f_{i}\|^{2}_{L^{2}}}\leq 2
$$
for $i = 1, \ldots, N$, the min-max characterization shows that there are at least $N$ eigenvalues less than $2$ for 
this new surface.
\end{proof}

\subsection{Spectral flow}
Using the above constructions, we can get the following description of the interior of the space of spherical metrics 
with $k$ cone points on $\bS^{2}$: either it has infinitely many connected components, or else there exist infinitely
many subsets of codimension one and corresponding metrics with 2 in the spectrum. 

We can see this by the following argument. Suppose the interior of the space only has finitely many connected components.  By the results of the previous subsection, we may choose two
spherical cone metrics, $g_0$ and $g_1$, with cone angle parameters $\vec \beta^{(0)}$ and $\vec \beta^{(1)}$ 
such that $\Delta_{g_j}$ has $N_j$ eigenvalues less than $2$, and $N_0 \neq N_1$. Since there are only finitely many connected components, $g_{0}$ and $g_{1}$ can be chosen so that there is a path $g_{s}, s\in [0,1]$ connecting $g_{0}$ and $g_{1}$.
And if we denote
by $N(s)$ the number of eigenvalues in $(0,2)$, then $N(s)$ is a continuous function.  
However, this contradicts the fact that $N(0) \neq N(1)$.  Therefore there exists $s\in [0,1]$ such that $2$ is an eigenvalue of $\Delta_{g_{s}}$. Since the whole path is contained in the interior of the space, one can perturb the metrics while keeping the number of eigenvalues below 2 for $g_{0}$ and $g_{1}$, and therefore we get a codimensional-one set of such metrics with 2 in the spectrum.

The space of solutions does have many connected components \cite{MP2}, and
the preceding discussion does not rule out the possibility that it might
have infinitely many components. In that case this spectral flow would not
produce metrics with $2$ in the spectrum. 

\section{Splitting cone points -- local theory}\label{s:split}
We now take up the description of families of metrics with merging cone points, or equivalently, the construction of families 
of metrics where isolated cone points split into clusters. It suffices here to consider flat conic metrics since the change from 
flat to spherical simply adds higher order perturbations which are irrelevant for the immediate considerations. 
We first carry out a local analysis and describe a parametrization of these splitting families using 
weighted symmetric polynomials in the locations of the cone points.  The differential
of this parametrization yields a family of functions which, as we show later, unobstructs our main deformation
problem.  Unfortunately, this parametrization is singular at the front face $F_0$ of $\calE_J$, and it is necessary to 
perform an iterated blowup of the range in order to obtain a local diffeomorphism near $F_0$.  This step is 
unfortunately rather technical.    The passage from the local to the global version of these results
is straightforward.

\subsection{Weighted factorizations}
Consider the flat conic metric
$$
g_0 = |z|^{2(\beta_{0}-1)}|dz|^{2}=e^{2v_0}|dz|^{2}, \ v_0=(\beta_{0}-1)\log|z|, \quad z \in D = \{|z| < 1\},
$$
with $\beta_0 > 1$. Our aim is to parametrize the family of flat conic metrics in $D$ with conic data $(p_1, \ldots, p_J, \beta_1, 
\ldots, \beta_J)$, where all $p_j$ are near $0$ and $\sum (\beta_j-1) = \beta_0-1$, and then compute the variations of this family. 

While no local constraints prevent us from considering splittings into arbitrarily large clusters of points, we prove below that
certain global constraints dictate that we must restrict to splittings into at most $J = [\beta_0]$ points, i.e.,  the size of the initial
cone angle determines the cardinality of the cluster.  We use local versions of the spaces $\calE_J$, $\calC_J$, 
where $M$ is replaced by the (open) disk $D$, or in fact by the entire complex plane $\CC$.  Fix $\vec B =( B_1, \ldots, B_J)$ with 
each $B_i \neq 1$ such that  
\begin{equation}\label{e:splitBeta0}
\sum_{j=1}^J(B_{j}-1)=\beta_{0}-1.
\end{equation}
The equal angle case,
\begin{equation}
B_j^{\mathrm{eq}} - 1 = \frac{1}{J}(\beta_0 - 1),\  \ j = 1,\ldots, J,
\label{eqang}
\end{equation}
is of particular importance. 

We first explain how local clustering families are in bijective correspondence (away from a certain locus) with functions of the form
\begin{equation}\label{e:v0}
\begin{split}
\dot v(0) &=\sum_{j=1}^{J}( e_{j}'\cos( j\theta)+e''_{j}\sin (j\theta))\, r^{-j/\beta_{0}} \\
& =\Re \sum_{j=1}^{J}  \frac{A_j}{z^j}, \qquad A_j = \beta_{0}^{j/\beta_{0}} (e_j' + i e_j''), \qquad r=|z|^{\beta_{0}}/\beta_{0}. 
\end{split}
\end{equation}
Note that $j \leq J$ implies $-1 \leq -j/\beta_0 < 0$, so our restriction on $J$ ensures that these exponents are not less than $-1$.  

Define the constants
\[
\frakb_j =  J \frac{B_j-1}{\beta_0 - 1},
\]
and write $\vec \frakb = (\frakb_1, \ldots, \frakb_J)$. Thus $\sum \frakb_j = J$, and in the equal angle case, each $\frakb_j = 1$.
We must avoid `degenerate' $J$-tuples of cone angles lying in the set $\widehat{\Delta} = \cup_I \widehat{\Delta}_I$, where
$I = \{i_1, \ldots, i_p\} \subset \{1, \ldots, J\}$ and $\widehat{\Delta}_I = \{\vec \frakb: \frakb_{i_1} + \ldots + \frakb_{i_p} = 0\}$. (Recall that $\frakb_{i_1} + \ldots + \frakb_{i_p} = 0$ is equivalent to $\sum 2\pi (B_{i_{j}}-1)=2\pi(1-1)$, that is, this subcluster merge to a point with angle $2\pi$.)

\begin{proposition}\label{p:localsheet}
For every $\vec \frakb \not\in \widehat{\Delta}$, there is a subvariety $\calS_{\vec \frakb} \subset \CC^J$, called the
weighted discriminant locus associated to $\vec \frakb$, and a proper holomorphic mapping $\calF = \calF_{\vec \frakb}:
\CC^J \to \CC^J$ which assigns to a $J$-tuple $\vec Z = (z_1, \ldots, z_J)$ the $J$-tuple $\vec A = (A_1, \ldots, A_J)$,
as determined by \eqref{e:A} below. This map has the following properties:
\begin{itemize}
\item[i)] $\calF$ is ramified along the union of the partial diagonals in $\CC^J$, and the image of this branch locus equals
the weighted discriminant locus $\calS_{\vec \frakb}$;
\item[ii)] the restriction of this mapping to the unramified set is a $J!$-sheeted covering 
map from the interior of $\calE_J$ to $\CC^J \setminus \calS_{\vec \frakb}$;
\item[iii)] fixing any local inverse $\calF_{\vec \frakb}^{-1}: \vec A \mapsto (z_1(\vec A), \ldots, z_J(\vec A))$, then the function
\begin{equation}\label{e:vAz}
v(\vec A;z) = \sum \frakb_j \log |z-z_j(\vec A)|
\end{equation}
is differentiable at $\vec A=\vec0$ and satisfies
\[
\frac{\del v}{\del e_{\ell}'}(\vec 0)= \cos (\ell \theta)|z|^{-\ell}, \quad
\frac{\del v}{\del e_{\ell}''}(\vec 0)= \sin ( \ell \theta)|z|^{-\ell}.
\]
\end{itemize}
\end{proposition}
\begin{proof}
For any $\vec A \in \CC^J$, define the polynomial 
\begin{equation}
  P(\vec A; z)=z^{J}+A_{1}z^{J-1}+\dots +A_{J}.
  \label{expbb}
\end{equation}
Then 
\[
2\frac{\del\,}{\del A_{j}} \log |P(\vec A; z)| = \frac{z^{J-j}}{P(\vec A; z)}, 
\]
and in particular, at $\vec A = 0$, this derivative equals $z^{J-j - J} = |z|^{-j} e^{-ij \theta}$.

In the equal angle case (where all $\frakb_j = 1$), we define $A_j = \sigma_j(\vec Z)$ to be the $j^{\mathrm{th}}$
symmetric polynomial of the $z_i$, so $P(\vec A; z) = (z - z_1(\vec A)) \ldots (z - z_J(\vec A) )$ and
$\{z_1(\vec A), \ldots, z_J(\vec A)\}$ is some ordering of the set of roots of $P(\vec A; z)$. Then
\[
v(\vec A;  z) = \sum \log |z - z_j(\vec A)| =\log |P(\vec A)|,\ \ z \not\in \{z_1(\vec A), \ldots, z_J(\vec A)\},
\]
and the computation above establishes the result in this special case.

For more general angle splittings, assume that $|z| > \max |z_j|$ and expand the product 
\begin{equation}
(z - z_1)^{\frakb_1} \ldots (z-z_J)^{\frakb_J} = \sum_{j=0}^\infty d_j(z_1, \ldots, z_J) z^{J-j}
\label{expbb1}
\end{equation}
using the binomial theorem in each factor. We then write 
\begin{equation}\label{e:A}
A_j(z_1, \ldots, z_J) =  d_{j}(z_1, \ldots, z_J), \quad j = 1, \ldots, J,
\end{equation}
which defines the coefficients in $P(\vec A, z)$ as in \eqref{expbb}.  This defines the map $\calF(\vec Z) = \vec A$.  The remainder of the series
in \eqref{expbb1} is lower order as all the $z_j \to 0$ in the sense that
\begin{equation}
|P(\vec A; z)| = |z - z_1|^{\frakb_1} \ldots |z-z_J|^{\frakb_J}(1 + \calO( \max |z_j|^{J+1})),
\label{multform}
\end{equation}
where the error term is uniform for $(1+\epsilon) \max\{ |z_j|\} < |z| < 1$, say. As we show below (see the error term estimates near the end of the proof), it is also true that
\begin{equation}
  |P(\vec A; z)| = |z - z_1(\vec A)|^{\frakb_1} \ldots |z-z_J(\vec A)|^{\frakb_J}(1 + \calO(|\vec A|^{1 + \epsilon})),
  \label{multform2}
\end{equation}
and assuming this, then the derivative of $v$, defined as in \eqref{e:vAz}, with respect to $A_{J-\ell}$ at $\vec A = 0$
is equal to $r^{-\ell}e^{-i\ell \theta}$, as before. 

We next consider the local inverses of $\calF$.   Let $\{\lambda_1, \ldots, \lambda_J\}$ denote
the roots of the polynomial $P(\vec A, z)$, so $P(\vec A; z) = (z - \lambda_1) \ldots (z-\lambda_J)$ and
$A_\ell = (-1)^\ell\sigma_\ell(\lambda_1, \ldots, \lambda_J)$ are the standard symmetric polynomials
of these roots. 
Now take the Taylor expansion of $\log (z-\lambda)$ in $\lambda$ around $\lambda = 0$; in the range $2\max |\lambda_j| < |z| < 1$, 
the error term is uniform and we have
\[
\log(z-\lambda)=Q_{J,z}(\lambda) + \calO(|\lambda|^{J+1}), \quad Q_{J,z}(\lambda)=c_{0}(z)+c_{1}(z)\lambda+\dots + c_{J}(z)\lambda^{J};
\]
for some functions $c_j(z)$ (which we do not need to write out explicitly). Therefore, 
\begin{equation}\label{e:P}
  \begin{split}
    \log P &=\sum_{j=1}^{J} \log (z-\lambda_{j}) =\sum_{j=1}^{J}Q_{J,z}(\lambda_{j}) + \calO(\max {|\lambda_{j}|}^{J+1}) \\
    & =Jc_{0}(z)+c_{1}(z)\sum_{j=1}^J\lambda_{j}+\dots + c_{J}(z)\sum_{j=1}^J\lambda_{j}^{J}  + \calO(\max {|\lambda_{j}|}^{J+1}).
    \end{split}
  \end{equation}
By Newton's formula, the $\ell^{\mathrm{th}}$ power sum is a quasi-homogeneous polynomial $R_\ell$ of the elementary symmetric
functions $\sigma_1, \ldots, \sigma_\ell$, hence the previous formula can be rewritten as
  \[
    \log P =  J c_0(z) + c_1(z) R_1(\vec A) + \ldots + c_J(z) R_J(\vec A) + \calO(\max {|\lambda_{j}|}^{J+1}).
    \]

Now consider the (locally defined) holomorphic function 
\begin{equation}\label{e:vprime}
V(z) = \sum_{j=1}^{J} \frakb_{j}\log (z-z_{j})=\sum_{j} \frakb_{j}Q_{J,z}(z_{j}) + \calO( \max |z_j|^{J+1}). 
\end{equation}
Equating this to \eqref{e:P} and discarding the error terms gives
\begin{equation}\label{e:bezout}
  \sum \frakb_{j}z_{j}^{\ell}=\sum \lambda_{j}^{\ell} = R_\ell(\vec A), \quad \ell = 1, \ldots, J.
\end{equation}
We use this set of equations to determine the $z_j$ from $\vec A$.  Multiply the right side of
the $\ell^{\mathrm{th}}$ equation by $z_0^\ell$ to interpret these as homogeneous polynomials
in the variables $z_0, \ldots, z_J$.   This modified set of equations corresponds to a collection of 
projective hypersurface $\Sigma_\ell \subset \CC \mathbb P^J$, $\ell =1, \ldots, J$, with $\mathrm{deg}\,(\Sigma_\ell) = \ell$.
By Bezout's theorem, the intersection of the $\Sigma_\ell$ contains $J!$
points, counted with multiplicity.  When all the $\frakb_j = 1$, these $J!$ points of intersection are
just the orbit of a single solution under the symmetric group.   As we show momentarily, away from the
partial diagonals there are $J!$ distinct solutions to these equations, and for each of these, $z_0 \neq 0$.
After that, we analyze the error terms.

We first show that $z_0 \neq 0$ for each solution, i.e., all solutions lie in $\CC^J$ rather than in the divisor at infinity.
For this, rewrite $\sum \frakb_j z_j^\ell = 0$ as
\[
  \begin{bmatrix}
    \frakb_1 z_1 + \ldots + \frakb_J z_J\\ \frakb_1 z_1^2 + \ldots + \frakb_J z_J^2 \\ \vdots  \\
    \frakb_1 z_1^{J} + \ldots + \frakb_J z_J^{J}  \end{bmatrix} =
\begin{bmatrix}  z_1 & \ldots & z_J \\ z_1^2 & \ldots & z_J^2 \\
  \vdots & \vdots & \vdots \\  z_1^{J} & \ldots & z_J^{J}\end{bmatrix}
\begin{bmatrix}  \frakb_1  \\ \frakb_2 \\ \vdots \\\frakb_J \end{bmatrix} =
\begin{bmatrix} 0 \\ 0 \\ \vdots \\ 0 \end{bmatrix}.
\]
The first factor is a van der Monde matrix, hence is nonsingular precisely when the $z_j$ are all distinct and nonzero.
On the other hand, suppose that $z_{i_1} = \ldots = z_{i_p} \neq 0$ for some $I = \{i_1, \ldots, i_p\} \subset \{1, \ldots, J\}$ and all
other $z_{j}=0$.  Then we obtain a solution to this equation provided 
$\frakb_{i_1} + \ldots + \frakb_{i_p} = 0$, which indicates why the sets $\widehat{\Delta}_I$ are excluded. To see that these sets create 
the only problem, an inductive argument shows that nonzero solutions to this system exist only if some such relationship exists 
amongst the $\frakb_i$, i.e., $\vec \frakb \in \widehat{\Delta}$.   We have now shown that if $\vec \frakb$ does not lie in this 
finite union of subspaces, then all $J!$ solutions, $\vec Z^{(i)} = (z_1^{(i)}, \ldots, z_{J}^{(i)})$, $i = 1, \ldots, J!$, are elements of $\CC^J$.

Now observe that $\calF$ is the composition of the two maps 
\[
\vec Z = (z_1, \ldots, z_J) \longmapsto ( \sum \frakb_i z_i, \ldots, \sum \frakb_i z_i^J) = (\sum \lambda_i, \ldots, \sum \lambda_i^J)
\]
and the global polynomial biholomorphism
\[
(\sum \lambda_i, \ldots, \sum \lambda_i^J) \longmapsto (\sigma_1(\vec \lambda), \ldots, \sigma_J(\vec \lambda)) = \vec A.
\]
In particular, $\calF$ is an algebraic mapping from $\CC^J$ to $\CC^J$, which is generically a $J!$-sheeted cover. 

\medskip

\noindent {\bf Claim:}  This map is a proper ramified cover of degree $J!$ with ramification locus the image of the partial
diagonals (in $\vec Z$).

\medskip

Properness is obvious. It suffices, therefore, to show that $\calF$ is a local biholomorphism at every $\vec Z$ away
from a partial diagonal.  Since the second mapping in the composition is a biholomorphism, it suffices to examine the first mapping.
Its complex Jacobian equals
\begin{gather*}
  \begin{bmatrix}  \frakb_1 & \ldots & \frakb_J \\
    2\frakb_1 z_1 & \ldots & 2\frakb_J z_J \\
    \vdots & \vdots & \vdots \\
    J \frakb_1 z_1^{J-1} & \ldots & J \frakb_J z_J^{J-1}
  \end{bmatrix}  \qquad \qquad \qquad \qquad \qquad \qquad \qquad \qquad \qquad \qquad \qquad \qquad \qquad  \\ 
=  \begin{bmatrix} 1 & 0 & \ldots & 0 \\
    0 & 2 & \ldots & 0 \\
    \vdots & \vdots & \vdots & \vdots \\
    0 & \ldots & 0 & J \end{bmatrix}
  \begin{bmatrix}  1 & \ldots & 1 \\
     z_1 & \ldots & z_J \\
    \vdots & \vdots & \vdots \\
    z_1^{J-1} & \ldots & z_J^{J-1}
  \end{bmatrix}
  \begin{bmatrix}  \frakb_1 & 0 & \ldots & 0 \\
   0 & \frakb_2 & \ldots & 0 \\
    \vdots & \vdots & \vdots & \vdots \\
    0 & \ldots & 0 & \frakb_J \end{bmatrix};
\end{gather*}
this is nonsingular provided the $z_j$ are distinct since no $\frakb_j = 0$, and the middle term on the right is once again
a van der Monde matrix. The inverse function theorem now establishes the claim.
The weighted discriminant locus $\calS_{\vec \frakb}$ is, by definition, the image under $\calF$ of the union
of partial diagonals.

We now analyze the error terms.   Our goal is to show that  $|z_{j}|, |\lambda_{j}| =\calO(\max |A_{i}|^{1/i})$ for
all $j$.  Granting this, then comparing \eqref{e:P} and~\eqref{e:vprime}, we obtain the desired estimate
\[
  \log |P|-\sum \frakb_j\log|z-z_j| = \calO(\max\{|\lambda_j|^{J+1}, |z_j|^{J+1}\} )
  = \calO( |\vec A|^{1+\epsilon}).
\]
 
To prove this claim, set $M=\max\{|A_{i}|^{1/i}\}$ and recall that $\sum \lambda_i^\ell$ is a quasihomogeneous polynomial of degree
$\ell$ in $A_1, \ldots, A_\ell$, so 
\[
\sum |\lambda_{j}|^{\ell}\leq C_{0} M^\ell, \ \ell = 1, \ldots, J. 
\]

\medskip

\noindent {\bf Claim:} There exists a constant $C > 0$, depending on $\vec \frakb$, such that if $\vec A \in \CC^{J}$ has
$M=\max\{|A_{i}|^{1/i}\}\leq 1$, then any solution $\vec Z$ to $\calF(\vec Z) = \vec A$ satisfies  $|\vec Z| \leq CM$.

\medskip

If no such constant $C$ exists, then there exists a sequence $\vec A^{(n)}$ and corresponding solutions
$\vec Z^{(n)}$ such that $\Lambda_n := |\vec Z^{(n)}| \geq n M^{(n)}$, $n = 1, 2, 3, \ldots$.   Dividing each of the original equations
by the appropriate powers of $\Lambda_n$ yields 
\[
\sum \frakb_j (\tilde z_{j}^{(n)})^{\ell}=\sum \lambda_{j}^{(n)} \Lambda_n^{-\ell}, \quad \ell =1, \dots, J,
\]
where $\tilde z_{j}^{(n)}=z_{j}^{(n)}/\Lambda_n$. By construction, $|\tilde z_{j}^{(n)}|\leq 1$ for $j=1, \dots, J$ 
with $|\tilde z_{j}^{(n)}|=1$ for at least one $j$, for every $n$. 

Since the sequence $\vec Z^{(n)}/\Lambda_n$ is bounded and has norm bounded away from zero, some subsequence converges
to a limiting $J$-tuple $\vec Z \neq \vec 0$ satisfying
\[
\sum \frakb_j z_j^\ell = 0, \quad \ell = 1, \ldots, J.
\]
However, $\vec \frakb \not\in \widehat{\Delta}$, so these equations have no nontrivial solutions. This contradiction proves the estimate.

Notice that if all $\frakb_j = 1$, we recover that for the exact roots, 
$$
|\lambda_{j}|\leq C M, \ \ j=1, \dots, J.
$$
\end{proof}

\subsection{Desingularization of $\calF^{-1}$}
The map $\calF^{-1}$ is a local diffeomorphism from $\CC^J \setminus \calS_{\frakb}$ onto the interior of $\calE_J$, at least
locally in $0 < |\vec A| < \epsilon$.  It will be particularly useful to study one-parameter paths $t \mapsto \calF^{-1}(t \vec A)$,
at least for certain $\vec A$, or more globally, to consider (any branch of) $\calF^{-1}$ as a map from the blowup 
$[\CC^J ; \{\vec A = 0\}] := \calA_J \longrightarrow \calE_J$. 

Observe that the front face $F(\calA_J)$ is a sphere $\bS^{2J-1}$, and the intersection $\calS_{\frakb} \cap F(\calA_J) = \calT_0$ 
has real codimension two in this sphere. In the following we will identify an additional finite number of real codimension
two subsets $\calT_1, \ldots, \calT_N$ of this front face, and corresponding conic extensions $\calS_j = C(\calT_j)$ (so
$\calS_0 = \calS_{\frakb}$). Write $\calT = \calT_0 \cup \ldots \cup \calT_N$ and $\calS = C(\calT)$.  We shall study the restriction of
$\calF^{-1}$ to the set $\Omega  = \calA_J \setminus \calS$, and in particular the behavior of this map near 
$\del \Omega = \Omega \cap F(\calA_J)$. 

Fix a branch of $\calF^{-1}$ and $\vec A \in \Omega \setminus \del \Omega$, and consider the curve $\calF^{-1}( t\vec A) = \vec Z(t) = 
(z_1(t), \ldots, z_J(t))$.  As $t \to 0$, $t\vec A$ converges to the point $\vec A/|\vec A| \in  F(\calA_J)$ and $\vec Z(t)$ 
converges to some point in $F_0$.    We define $\calS_1 = \{A_J =0\}$ (and $\calT_1 = \calS_1 \cap F(\calA_J)$).
Thus if $\vec A \in \Omega$ then $A_J \neq 0$ and by the algebraic nature of $\calF$, at least one component $z_j(t)$ of 
$\vec Z(t)$ satisfies $z_j(t) \sim t^{1/J} \zeta_j + \calO(t^{2/J})$. By contrast, if $A_J = 0$, then the leading term of each of 
these components is of order $t^{1/(J-1)}$ or lower. The coefficients $\zeta_j$ are determined as follows. For each 
$\ell$, $R_\ell(t\vec A)$ is a polynomial in $t$ with no constant term; furthermore, the quasihomogeneity of $R_\ell$ implies that 
the only term with a linear power of $t$ is $tA_J$, and this occurs only in $R_J$. All other powers of $t$ in any $R_\ell(t\vec A)$ 
have exponent at least $2$. Inserting these putative expansions for $z_j$ into the algebraic system \eqref{e:bezout} and equating
the coefficients of $t$ yields the sequence of equations 
\begin{equation}
\sum_{j} \frakb_j \zeta_j^\ell = 0,\ \ell < J, \quad \sum_{j} \frakb_j \zeta_j^J = A_J.
 \label{bezout2}
\end{equation}

Clearly, the solution $\vec \zeta = (\zeta_1, \ldots, \zeta_J)^T$ depends {\it only} on $A_J$, but none of the other 
$A_\ell$, $\ell < J$.  In addition, its dependence on $A_J$ is homogeneous, i.e.,  $\vec \zeta (A_J) = \vec \zeta(1) A_J^{1/J}$.   
Hence the image of every point in the face $\del \Omega$ lies on a particular circle determined solely by $\vec \frakb$ 
and which we denote by $\sigma_{\frakb}$. We shall also see momentarily that $\sigma_{\frakb}$ lies entirely in a single
spherical fiber of $F_0$.  

There are complete expansions for each $z_j(t)$, hence (as also shown by general algebraic principles), each branch of $\calF^{-1}$ 
extends to a polyhomogeneous function on $\Omega$. However, this is not a local (polyhomogeneous) diffeomorphism near 
boundary points since it is far from surjective.  Our deformation theory will ultimately require that we somehow
extend $\calF^{-1}$ to a map with invertible differential even at $F(\calA_J)$, and we now explain
how this may be achieved by replacing $\calE_J$ by some iterated blowup along $\sigma_{\frakb}$.  The goal of
these blowups is to `separate out' the different paths $\vec Z(t)$ corresponding to different values of $\vec A$.  

\subsubsection{Directions of increasing order of vanishing}
This construction will be somewhat lengthy and occupy the remainder of this subsection. The idea is that each function 
$z_j(t)$ has an expansion where, after some preliminary analysis, we can see that the coefficient of $t^{i/J}$ for any $2 \leq i \leq J$
involves only $A_{J-i+1}, A_{J-i+2}, \ldots, A_{J-1}$.  Further study of these coefficients shows that there exists a linearly independent
set of directions in the bundle of vectors normal to $\sigma_b$ in $F_0$ which represent the directions tangent to those
paths which decay like $t^{i/J}$, $i = 2, \ldots, J$.  The iterated blowup is defined in terms of this independent set of directions. 

The first step is to examine more closely how the system 
$$
\sum \frakb_{j}z_{j}^{\ell}= R_\ell(\vec A), \quad \ell = 1, \ldots, J
$$
determines the asymptotics of the $z_{i}$.  Since $A_J \neq 0$ in $\Omega$, we can normalize by setting $\rho = |A_J|^{1/J}$,
and also write $\tilde A_{\ell}=A_{\ell}/|A_{J}|$, or equivalently $A_\ell = \rho^J \tilde{A}_\ell$, $\ell = 1, \ldots, J$. 
We also write $\tilde{A}_J = e^{i\theta}$; this angle $\theta$ will appear often below.  The entire collection of these
normalized components will be denoted $\widetilde{A} = (\tilde{A}_J = e^{i\theta}, \tilde{A}_{J-1}, \ldots, \tilde{A}_1)$. 
Finally, decompose $R_{\ell}(\vec A)=\ell A_{\ell}+e_{\ell}(A_{1}, \dots, A_{\ell-1})$; each monomial in $e_\ell$ is a constant
multiple of a product $A_1^{i_1} \ldots A_{\ell-1}^{i_{\ell-1}}$ where $i_1 + 2 i_2 + \ldots + (\ell-1)i_{\ell-1} = \ell$ and hence
has degree at least $2$. This implies that $R_{\ell}(\vec A) = \ell \tilde A_{\ell}\rho^{J} + \calO(\rho^{2J})$.  

Now substitute
$$
z_{i}=\sum_{j=1}^{J} c_{ij}\rho^{j} + \calO(\rho^{J+1})
$$
into the $\ell^{\mathrm{th}}$ equation of this system and collect terms with like powers to get
\begin{equation}\label{e:Rell}
\begin{split}
\sum \frakb_{j}z_{j}^{\ell}=P_{\ell, 0}\rho^{\ell}+\dots+P_{\ell, J-1}\rho^{\ell+J-1}+\calO(\rho^{\ell+J}), 
\end{split}
\end{equation}
where
\begin{equation}\label{e:Pellell}
\begin{split}
P_{\ell, 0 }&=\sum_{i=1}^{J} \frakb_{i} c_{i1}^{\ell}, \quad P_{\ell, 1}=\ell \sum_{i=1}^{J} \frakb_{i} c_{i1}^{\ell-1}c_{i2}, \quad  \mbox{and in general} \\
P_{\ell, k}&=\ell \sum_{i=1}^{J} \frakb_{i} c_{i1}^{\ell-1}c_{i, k+1}+ \sum_{i} \frakb_{i}Q_{\ell,k}( c_{i1}, \dots, c_{ik}),
\ \ 2 \leq k \leq J-1. 
\end{split}
\end{equation}
Each $Q_{\ell,k}$ is a sum of monomials $c_{i1}^{\ell_{1}}\dots c_{ik}^{\ell_{k}}$ with $\sum j\ell_{j}=\ell+k$, $\sum \ell_{j}=\ell$ and
$\ell_{1}<\ell-1$; for example, $Q_{\ell, 2}=\ell(\ell-1) c_{i1}^{\ell-2}c_{i2}^{2}$.  

Equating the coefficients of $\rho^{\ell+k}$ on the left and right side yields, for $\ell = 1, \ldots, J$ and $k = 0, \ldots, J-1$, that
\[
P_{\ell, k}  = \begin{cases} 0, \quad & \mbox{if}\  k \neq J \\  (J-k) \tilde{A}_{J-k},\ & \mbox{if}\ k = J \end{cases}
\]


When $k=0$, this is an algebraic system for the components of the vector $\vec c_1 = (c_{11}, \ldots, c_{J1})^T$:
\[
\sum_{i=1}^{J} \frakb_{i} c_{i1}^{\ell}=0, \ \ell \leq  J-1, \ \ \sum_{i=1}^{J} \frakb_{i} c_{i1}^{J}= J \tilde{A}_J = J e^{i\theta},
\]
and the solution is just a scalar multiple of the solution $\vec \zeta$ to \eqref{bezout2}. As in that case, there exist $J!$ 
solutions to  this system and when $\frakb_{i}=1$ for all $i$, these correspond to permuting
the components of $(J^{1/J}e^{i(2\pi k + \theta)/J})$, $k = 1, \ldots, J$. 

On the other hand, when $k=1$ the equation is now a linear system for $\vec c_{2}=(c_{12}, c_{22}, \dots, c_{J2})^{T}$:
\[
  \sum_{i=1}^{J} \frakb_{i} c_{i1}^{\ell-1}c_{i2}=0, \ \ell \neq  J-1, \ \ \sum_{i=1}^{J} \frakb_{i} c_{i1}^{J-2}c_{i2}=\tilde A_{J-1},
\]
which we write as $T \vec c_{2}=\vec x_{2}$, where $\vec x_{2}=(0,\dots, 0, \tilde A_{J-1}, 0)^{T}$ and 
\begin{multline*}
T=\left[ 
\begin{array}{cccc}
\frakb_{1} & \frakb_{2}&\dots & \frakb_{J}\\
\frakb_{1} c_{11} & \frakb_{2} c_{21} &\dots & \frakb_{J}c_{J1}\\
\dots\\
\frakb_{1} c_{11}^{J-1} & \frakb_{2} c_{21}^{J-1}&\dots & \frakb_{J}c_{J1}^{J-1}
\end{array}
\right]\\
=
\begin{bmatrix}
1 & 1&\dots &1\\
 c_{11} &  c_{21} &\dots & c_{J1}\\
\dots\\
 c_{11}^{J-1} &  c_{21}^{J-1}&\dots & c_{J1}^{J-1}
\end{bmatrix}
\begin{bmatrix}  \frakb_1 & 0 & \ldots & 0 \\
   0 & \frakb_2 & \ldots & 0 \\
    \vdots & \vdots & \vdots & \vdots \\
    0 & \ldots & 0 & \frakb_J \end{bmatrix}.
\end{multline*}
Since the $\frakb_{i}$ are all nonzero, this matrix is invertible unless $c_{i1}=c_{j1}$ for some $i\neq j$. Therefore, except for a
(real) codimension $2$ subset of values of $\vec A$ which we denote as $\calS_{2}$ to accord with previous notation, $T$ is invertible and hence there is a unique solution vector $\vec c_{2}$,
whose components are multiples of $\tilde A_{J-1}$.

Similarly, for larger values of $k$, we obtain an inhomogeneous linear system for $\vec c_{k}=(c_{1k}, \dots, c_{Jk})^{T}$ which is now slightly more
complicated because of the appearance of the lower order terms $Q_{\ell, i}$: 
\[
\begin{split}
\sum_{i=1}^{J} \frakb_{i} c_{i1}^{\ell-1}c_{ik}&=-\sum\frakb_{i}Q_{\ell, k-1}(c_{i1}, \dots, c_{i, k-1}), \ 
\ell\neq J-k+1, \ \mbox{and} \\ 
\sum_{i=1}^{J} \frakb_{i} c_{i1}^{J-k}c_{ik}&=\tilde A_{J-k+1}-\sum\frakb_{i}Q_{J-k+1, k-1}(c_{i1}, \dots, c_{i,k-1}).
\end{split}
\]
This can be written more simply as $T\vec c_{k}=\vec y_{k}$, where  $\vec y_k = \vec x_k - \vec q_k$, where
\[
\vec x_k = (0, \ldots, \tilde{A}_{J-k+1}, 0, \ldots, 0)^T,\ \  \vec q_k = (q_{k\ell}),\ q_{k\ell} = \sum_i \frakb_i Q_{\ell, k-1}(c_{i1}, \ldots, c_{i, k-1}).
\]
By the invertibility of the same matrix $T$, there exists a unique solution, at least for $\vec A$ outside of a real codimension two subvariety which is denoted as $\calS_{k+1}$. 

Now write the entire system as $TC = Y$, where
$$
C=(\vec c_{1}, \dots, \vec c_{J}), \ Y = (\vec y_1, \vec y_{2}, \dots, \vec y_{J}).
$$
The entries of $T$ depend on the $c_{i1}$, hence the first column of $TC$ is actually a nonlinear equation in these variables;
however it is convenient to think of the entries of these two matrices as uncoupled. 
\begin{lemma}
The matrix $C$ has rank $J$ when $\vec A$ lies outside of a real codimension $2$ subvariety of $\CC^J$. 
\end{lemma}
\begin{proof}
We have shown that $T$ is invertible for any $\vec A$ outside a real codimension $2$ subvariety.  Thus, restricting to such values
of $\vec A$, it suffices to prove that $X$ is also invertible, possibly restricting the set of allowable $\vec A$ further.  Now
$Y = X - Q$ where $X$ has entries $\tilde{A}_i$ on the anti-diagonal (i.e., $X_{1J} = \tilde{A}_1$, $X_{2, J-1} = \tilde{A}_2$, etc.) and
zeroes elsewhere, and $Q$ has columns $\vec q_1, \ldots, \vec q_J$. Recall that $\vec q_1 = \vec q_2 = 0$, and 
$\vec q_k$ depends only on $\tilde{A}_{J-k+1}, \ldots, \tilde{A}_J$. 

We now use column operations to reduce to a matrix with all entries below the main antidiagonal equal to $0$.
These operations involve multiplication by rational functions of the $\tilde{A}_i$, and we need to keep some
track of the dependence.

The only two nonzero entries in the first two columns are $J e^{i\theta}$ and $\tilde{A}_{J-1}$, and appropriate 
multiples of the inverses of these entries can be used to clear all the entries in the bottom two rows.
Next, use the inverse of the antidiagonal entry $\tilde{A}_{J-2} - q_{3, J-2}$ to clear all entries to its
right on row $J-2$.  Note that this introduces rational functions with denominators depending only
on $\tilde{A}_J, \tilde{A}_{J-1}$ and $\tilde{A}_{J-2}$.   Carrying on, we use the antidiagonal entry
$\tilde{A}_{J-\ell} - q_{\ell+1, J-\ell}$ to clear the entries to its right; this uses rational functions with
denominators depending only on $\tilde{A}_J, \ldots, \tilde{A}_{J-\ell}$.

Provided we restrict to the complement of the zero sets of the denominators which appear along the way,
i.e. to the union of a finite number of real codimension two varieties $\calS_{J+1}$, we obtain a matrix with all entries
below the main antidiagonal equal to $0$.   The entries along the main anti-diagonal are each of 
the form $\tilde{A}_{J-\ell}$ plus a rational function depending only on $\tilde{A}_{J-\ell_1}, \ldots, \tilde{A}_J$.
Restricting one final time to the complement of where these entries vanish denoted as $\calS_{J+2}$, we see that $Y$ is invertible,
as claimed.
\end{proof} 

By induction, each component of $\vec c_k$, $k \geq 2$, is a constant multiple of $\tilde A_{J-k+1}$ plus a polynomial depending
only on $\tilde A_{J-1}, \dots, \tilde A_{J-k+2}$, i.e.,
\[
c_{ik} = d_{ik} \tilde{A}_{J-k+1} + f_{ik}(\tilde{A}_{J-k+2}, \ldots, \tilde{A}_{J-1})
\]
(with $f_{i2} = 0$ for all $i$).  Note that, by its defining equation, $\vec c_1 = \vec d_1 \xi$, where $\vec d_1$ is a constant vector and
$\xi := A_J^{1/J} = \rho e^{i\theta/J}$.

Employing complex notation to simplify calculations, the information above allows us to compute the Jacobian of the change of variables
\[
\tilde{A} := (\xi , \tilde{A}_{J-1}, \ldots, \tilde{A}_1) \longmapsto  (z_1, \ldots, z_J).
\]
The structure of the $f_{ij}$ now shows that 
\[
D_{\tilde{A}} \vec Z = 
\begin{bmatrix}
  d_{11}+\calO(\rho) &   d_{12} \rho^2 + \calO(\rho^3) & \cdots & d_{1J}\rho^{J} +\calO(\rho^{J+1})   \\
  d_{21}+\calO(\rho)& d_{22} \rho^2 + \calO(\rho^3) & \cdots & d_{2J} \rho^J  + \calO(\rho^{J+1}) \\

  \cdots& \cdots & \cdots & \cdots \\
  d_{J1}+\calO(\rho)&  d_{J2}\rho^{2} +\calO(\rho^{3}) & \cdots & d_{JJ}\rho^{J} +\calO(\rho^{J+1})
\end{bmatrix}.
\]
By the preceding calculations, the matrix $(d_{ij})$ is nonsingular provided $\vec A$ remains outside a real codimension two variety.

\subsubsection{The final iterated blowup} 
The computations above indicate precisely how $\calF^{-1}$ becomes singular near $F(\calA_J)$, and motivate 
how this map can be desingularized by a sequence of blowups.  

We first explain this when $J=2$. Passing from the coordinates $z_1, z_2$ to $z_0 = \tfrac12 (z_1 + z_2)$ and 
$\tilde{z}_1 = \tfrac12( z_1 - z_2)$, and writing $\sqrt{\frakb_2/\frakb_1} := \bar{\frakb}$ for simplicity, we have
\[
\begin{split}
z_1&= -\tfrac12 A_1 +\tfrac12 \bar{\frakb} \sqrt{ A_1^2 - 4A_2} =  -\tfrac12 \rho^2 \tilde{A}_1 + \tfrac12 \rho \bar{\frakb}
\sqrt{ \rho^2 \tilde{A}_1^2 - 4 e^{i\theta}}  \\ & \sim \bar{\frakb} ie^{i\theta/2} \rho - \tfrac12 \tilde{A}_1 \rho^2 + \calO(\rho^3),
\end{split}
\]
and similarly 
\[
\begin{split}
z_2&= -\tfrac12 A_1 -\tfrac12 \bar{\frakb}^{-1} \sqrt{ A_1^2 - 4A_2} \sim - \bar{\frakb}^{-1} ie^{i\theta/2} \rho - \tfrac12 \tilde{A}_1 \rho^2 + \calO(\rho^3),
\end{split}
\]
and hence 
\[
\begin{split}
z_0 & =\frac{1}{2}( \bar{\frakb}  -\bar{\frakb}^{-1}) ie^{i\theta/2} \rho  -\tfrac12 \tilde A_1 \rho^{2} +\calO(\rho^{3}) =c(\theta, \frakb)\rho -\tfrac12  \tilde{A}_1 \rho^2+\calO(\rho^{3}), \\
\tilde{z}_1 & = \tfrac12 (\bar{\frakb} + \bar{\frakb}^{-1})ie^{i\theta/2} \rho +\calO(\rho^{3})=c'(\theta, \frakb) \rho e^{i\theta/2}+\calO(\rho^{3}).
\end{split}
\]
Now set $\tilde{z}_1 = R e^{i\phi}$, so that $R, \phi, z_0$ are coordinates on $\calE_2$.  We can then use these to write the lift
$P^{(0)}: \calA_2 \to \calE_2$ of $\calF^{-1}: \CC^2 \to \CC^2$ as
\[
\begin{split}
& P^{(0)}: (\rho, \theta, \tilde{A}_1) \mapsto  (R, \phi, z_0), \\ 
R = c'(\theta, \frakb) \rho & + \calO(\rho^3),\ \phi = \theta/2 + \calO(\rho^2),\ z_0 =c(\theta, \frakb)\rho -\tfrac12 \rho^2 
\tilde{A}_1 + \calO(\rho^3).
\end{split}
\]
Clearly $P^{(0)}$ has submaximal rank at $\rho = 0$ since $P^{(0)}(0, \theta, \tilde{A}_1) = (0, \theta/2, 0)$.  To remedy this, 
we perform two blowups. 

Rename $\sigma_{\frakb} = \sigma^{(0)}$ (to accord with later conventions); this equals the image $P^{(0)}(\{\rho = 0\})$.  
The first step is to blow up $\calE_2$ along $\sigma^{(0)}$, yielding the space $\calE_2^{(1)} = [\calE_2; \sigma^{(0)}]$.  
Coordinates on this new space are obtained by replacing $z_0$ with the new coordinate $z_0^{(1)} = z_0/R$ and the lift of $P^{(0)}$ equals
\[
\begin{split}
P^{(1)}: & (\rho, \theta, \tilde{A}_1) \mapsto (R, \phi, z_0^{(1)}) \\ 
& = (c'(\theta, \frakb)\rho + \calO(\rho^2), \theta/2 + \calO(\rho^2), 
c(\theta, \frakb) - \tfrac12 \rho \tilde{A}_1 + \calO(\rho^2) ).
\end{split}
\]
This is still singular at $\rho=0$ since $c(\theta, \frakb) =(\bar{\frakb} - \bar{\frakb}^{-1}) ie^{i\theta/2}$ is independent of $\tilde A_{1}$, 
but is slightly less singular than $P^{(0)}$. The image $\sigma^{(1)} := P^{(1)}(\{\rho=0\})$ is a circle in the interior of the new front face 
$F(\calE_2^{(1)})$ of this new blowup. 

Finally, blow up once again to arrive at $\calE_J^{(2)} = [\calE_J^{(1)};\sigma^{(1)}]$.  This has the new coordinate 
$z_{0}^{(2)} =(z_{0}^{(1)}-c(\theta, \frakb))/R$, and the lift of $P^{(1)}$ is
$$
P^{(2)}: (\rho, \theta, \tilde{A}_1) \mapsto (R, \phi, z_0^{(2)}) = (c'(\theta, \frakb)\rho +\calO(\rho^{2}), \theta/2+\calO(\rho^{2}),  -\tfrac12  \tilde{A}_1 +\calO(\rho)).
$$
This now {\it is}  a local diffeomorphism, even at $\rho = 0$. 

\bigskip

Let us now return to the general case and summarize the entire process before writing the steps more carefully.
First blow up $\calE_J$ at $\sigma_{\frakb} = \sigma^{(0)}$. Then $P^{(0)}$ lifts 
to a map $P^{(1)}: \Omega \to \calE_J^{(1)}$ which is slightly less degenerate at $\rho = 0$ in the sense that the 
image $\sigma^{(1)} = P^{(1)}( \{\rho = 0\})$ is now $3$-dimensional (instead of $1$-dimensional).  We continue, blowing up
$\sigma^{(1)}$ to obtain $\calE_J^{(2)}$ and a lifted map $P^{(2)}$ which is less degenerate still. The dimension of 
$P^{(2)}(\{\rho =0\})$ increases by $2$.  The later blowups $\calE^{(j+1)} = [\calE_J^{(j)}; \sigma^{(j)}]$ and maps
$P^{(j)}$ are defined the same way.  Continuing through $J-2$ steps, the image $\sigma^{(J-2)} = P^{(J-2)} (\{\rho = 0\})$ 
is $(2J-3)$-dimensional. This dimension does not increase after the next blowup, but finally, $P^{(J)}: \Omega \to
\calE_J^{(J)}$ is nondegenerate even at $\rho = 0$.  

We prepare by choosing coordinates analogous to $(R, \phi, z_0)$ on $\calE_J$. The center of mass of $(z_1, \ldots, z_J)$ is
$z_{0}=\sum z_{i}/J$, thus if we set $\tilde z_{i}=z_{i}-z_{0}$ (so $\sum_{i=1}^{J} \tilde z_{i}=0$), then $z_0, \tilde{z}_1, \ldots, 
\tilde{z}_{J-1}$ is a full coordinate system on $\CC^J$. We next pass to projective coordinates near a point in the interior of 
$F_0$ in $\calE_J$ by writing 
\begin{multline*}
\tilde{z}_1 = R e^{i\phi},\ z_j^{(0)} = \tilde{z}_j/R, \ j = 2, \ldots, J-1, \\ 
\mbox{so}\ \ (R, \phi, z_0, z_2^{(0)}, \ldots, z_{J-1}^{(0)}) \in \RR^+ \times \bS^1 \times \CC \times \CC^{J-2}.
\end{multline*} 

The expansions for each $z_i$ in $\rho$ yield 
$$
z_0 \sim \bar{c}_1 \rho + \bar{c}_2 \rho^2 + \ldots, \ \ \tilde z_{i}=  \tilde{c}_{i1} \rho + \tilde{c}_{i2} \rho^2 + \ldots,
$$
where $\tilde{c}_{ij} = c_{ij} - \bar{c}_j$.  We recall that the coefficient of $\rho$ in each of these expansions is a function 
of $\theta$ alone, the coefficient of $\rho^2$ is a function of $\theta$ multiplied by $\tilde{A}_{J-1}$, and for $j \geq 2$, 
the coefficient of $\rho^j$ takes the form 
$a(\theta) \tilde{A}_{J-j} + b(\theta, \tilde{A}_{J-1}, \ldots, \tilde{A}_{J-j+1})$.  We refer to this as the ``standard dependence''. 
It is straightforward to check that $R$, $\phi$, and the $z_i^{(0)}$ all exhibit this same standard dependence and in particular
$R = \calO(\rho)$, $\phi = \theta/J + \calO(\rho)$, and $z_i^{(0)} = \tilde{c}_{i1}(\theta)/\tilde{c}_{11}(\theta) + \calO(\rho)$.
This parametrizes the circle $\sigma_{\frakb}$ by $\theta$ via $\theta \mapsto (0, \theta/J, 0, c_{21}(\theta)/c_{11}(\theta), 
\ldots, c_{J-1,1}(\theta)/c_{11}(\theta))$. We change notation slightly, first writing $z_0 = z_0^{(0)}$, and then 
\[
R \sim \sum_{j=1}^J \kappa_j \rho^j,\ z_0^{(0)} \sim \sum_{j=2}^{J+1} e_{0j}^{(0)} \rho^{j-1},\ \ 
z_i^{(0)} \sim \sum_{j=1}^J e_{ij}^{(0)} \rho^{j-1}, \ i = 2, \ldots, J-1.
\]
This defines the lift of $\calF^{-1}$ to a map $P^{(0)}: \Omega \longrightarrow \calE_J$. 

For the next step, where we blow up $\sigma_{\frakb} = \sigma^{(0)}$ in $\calE_J$, we find projective coordinates 
on $\calE_J^{(1)} = [\calE_J; \sigma^{(0)}]$ by recentering each $z_i^{(0)}$ and then dividing by $R$, to get
\[
R, \ \phi, \ z_0^{(1)} = z_0^{(0)}/R,\ \tilde{z}_i^{(1)} = (z_i^{(0)} - e_{i1}^{(0)})/R,\ i = 2, \ldots, J-1.
\]
As will be the case at each step below, the expansions for each of these functions exhibits standard dependence, with 
\[
R = \calO(\rho),\ \phi = \theta/J + \calO(\rho),\ \ z_i^{(1)} = e_{i2}^{(0)}/\kappa_1 + \calO(\rho),\ i = 0, 2, \ldots, J-1, 
\]
or, changing notation again and writing out more of the expansion, 
\[
z_i^{(1)} \sim \sum_{j=2}^{J} e_{ij}^{(1)} \rho^{j-2},\ i = 0, 2, \ldots, J-1.
\]
Write $P^{(1)}$ for the lift of $P$. At $\rho = 0$, $e_{02}^{(1)}$ depends only on $\theta$ while $e_{i2}^{(1)} = 
(e_{i2}^{(1)})'(\theta) \tilde{A}_{J-1}$, $i \geq 2$.  Hence $P^{(1)}(\{\rho=0\})$ is a $3$-dimensional submanifold $\sigma^{(1)}$; 
it is a bundle of hemispheres $\bS^2_+$ over a circle parametrized by $\phi = \theta/J$. This circle itself is given by 
$z_0^{(1)} = e_{02}^{(1)}(\theta)$, with all other $z_i^{(1)} = 0$, and $\tilde{A}_{J-1}$ is a projective coordinate on each hemisphere.  

The pattern is now relatively clear, but we write out the next step since there is one further minor twist.  Define $\calE_J^{(2)} = 
[\calE_J^{(1)}; \sigma^{(1)}]$.  Each point of $\sigma^{(1)}$ corresponds to some values of $\theta$ and $\tilde{A}_{J-1}$, which we now fix.  

Rotate the coordinates $z_i^{(1)}$, $i = 2, \ldots, J-1$, to new coordinates $\tilde{z}_i^{(1)}$, where $\tilde{z}_2^{(1)} = 
f_{22}(\phi) \tilde{A}_{J-1} + \calO(\rho)$ and $\tilde{z}_i^{(1)} = \calO(\rho)$ for $i > 2$, still with standard dependence.  Thus 
$\tilde{z}_3^{(1)}, \ldots, \tilde{z}_{J-1}^{(1)}$ are coordinates in directions complementary to $\sigma^{(1)}$ (and of course if $J=3$, 
this latter part of the coordinate system is absent).  The blowup is realized by the new coordinates $z_0^{(2)} = 
(z_0^{(1)} - e_{02}^{(1)})/R$ and $z_i^{(2)} = \tilde{z}_i^{(1)}/R$, $i = 3, \ldots, J-1$.  Define $e_{03}^{(2)}:=e_{03}^{(1)}/\kappa_{1}$.  Then 
$P^{(1)}$ lifts to a map $P^{(2)}$ into $\calE_J^{(2)}$, which satisfies 
$z_0^{(2)} = e_{03}^{(2)} + \calO(\rho)=(e_{03}^{(2)})'(\theta) \tilde A_{J-1} + \calO(\rho)$  and 
$z_i^{(2)} = (e_{i3}^{(1)})'(\theta) \tilde{A}_{J-2} + (e_{i3}^{(1)})''(\theta, \tilde{A}_{J-1}) + \calO(\rho)$, $i \geq 3$.  Here
$(e_{i3}^{(1)})''(\theta, \tilde{A}_{J-1})$ is rational in $\tilde{A}_{J-1}$, with coefficients smooth in $\theta$. 
The small difference here is that the leading coefficients in the expansions in $\rho$ of each of these new coordinates 
is affine in $\tilde{A}_{J-2}$ instead of just linear. Both $(e_{i3}^{(1)})'$ and $(e_{i3}^{(1)})''$ are determined in terms of the 
location on $\sigma^{(1)}$!    We see from this that $P^{(2)}(\{\rho = 0\})$ is a bundle of two-dimensional 
hemispheres over $\sigma^{(1)}$, now parametrized affinely rather than just linearly by $\tilde{A}_{J-2}$. New coordinates at this step are: 
\[
(R, \phi, z_{0}^{(2)}, \tilde z_{2}^{(1)}, \tilde z_{3}^{(2)}, \dots, \tilde z_{J-1}^{(2)}).
\]

In general, there is a sequence of blowups
\[
\calE_{J}^{(j)}=[\calE_{J}^{(j-1)}; \sigma^{(j-1)}], j=1,\dots, J-1, 
\]
along with the repeated lifted maps $P^{(j)}: \Omega \rightarrow \calE_J^{(j)}$, where $\sigma^{(j-1)} = P^{(j-1)}(\{\rho = 0\})$. 
This corresponds to new coordinates 
\[
(R,\phi, z_{0}^{(j)}, \tilde z_{2}^{(1)}, \tilde z_{3}^{(2)}, \dots, \tilde z_{j}^{(j-1)}, \tilde z_{j+1}^{(j)}, \dots, \tilde z_{J-1}^{(j)})
\]
as follows. At each point of $\sigma^{(j-1)}$, the values of  $\theta, \tilde A_{J-1}, \dots \tilde A_{J-j+1}$ are fixed.  
At this $j^{\mathrm{th}}$ stage we have 
\begin{multline*}
z_{0}^{(j)}=(z_{0}^{(j-1)}-e_{0j}^{(j-1)})/R \\
\sim (e_{0(j+1)}^{(j)})'(\theta)\tilde A_{J-j+1}+(e_{0(j+1)}^{(j)})''(\theta, \tilde A_{J-1}, \dots, \tilde A_{J-j+2}) +\calO(\rho).
\end{multline*}
Because $z_i^{(j-1)}$ depends affinely on $\tilde A_{J-j+1}$ for $i \geq j$, we can rotate the coordinates $\tilde z_{i}^{(j-1)}$, $i\geq j$,
so that $\tilde z_{j}^{(j-1)}=f_{jj}(\phi)\tilde A_{J-j+1}+\calO(\rho)$, and $z_{i}^{(j-1)}=
\calO(\rho)$ for $i>j$.  Thus we can define $z_{i}^{(j)}=z_{i}^{(j-1)}/R$.   These coordinates are again affine in $\tilde A_{J-j}$, which 
guarantees that one can proceed further in this iteration. 

When $j \leq J-2$, the limiting set $\sigma^{(j)}$ is a bundle with fibers $\bS^2_+$ over $\sigma^{(j-1)}$, and we continue as before.
If $j=J-1$, then $\sigma^{(j)}$ is a graph over, and hence has the same dimension as, $\sigma^{(j-1)}$.   We blow this submanifold
up and have coordinates $(z_0^{J-1}, \tilde{z}_2^{(1)}, \tilde{z}_3^{(2)}, \ldots, \tilde{z}_{J-1}^{(J-2)})$. The final step, when $j=J$,  
involves the new coordinate $z_{0}^{(J)}=(z_{0}^{(J-1)}-e_{0J}^{(J-1)})/R$.  Just as when $J=3$ earlier, the lifted map
\begin{equation}\label{e:PJ}
P^{(J)}: (\rho, \theta, \tilde A_{1}, \dots, \tilde A_{J-1})\mapsto (\phi, z_{0}^{(J)}, \tilde z_{2}^{(1)}, \dots, \tilde z_{J-1}^{(J-2)})
\end{equation}
is a local diffeomorphism at $\rho=0$;  indeed, to leading order,  $z_{0}^{(J)}\sim \tilde A_{1}, z_{j}^{(j-1)}\sim \tilde A_{J-j+1}, 2\leq j\leq J-1$. 
Therefore, the limiting set $\sigma^{(J)}$ is open in the front face of $\calE_{J}^{(J)}$, and the map $P^{(J)}$ is a local diffeomorphism
at $\rho = 0$. 

\bigskip

The goal of this entire construction has now been realized: we have (implicitly) identified a finite number of
real codimension $2$ conic subvarieties $\calS_j \subset \calA_J$ and the space $\calE_{J}^{(J)}$, and 
have demonstrated that if $\vec A \in \Omega \setminus \del \Omega$, then $\vec Z(t) = \calF^{-1}(t \vec A)$ lifts to
a polyhomogeneous map $[0,1) \to \calE_{J}^{(J)}$, and this is a `slice' of a local diffeomorphism $\Omega 
\to \calE_J^{(J)}$. 

We emphasize that this description is `very' local, and in particular, we have not tried to describe the behavior
near the possible intersection of $\sigma_{\frakb}$ with other faces of $\calE_J$.  A careful understanding of such 
behavior is likely to be complicated and should involve a more complicated set of blowups around
the successive strata of $\calS_{\vec \frakb}$. 

In any case, in terms of all of this, we can now define, locally in $\Omega$, a suitable family of conformal factors 
$v( \vec A; z)$ as a fiberwise function on $\calC_J$. Our earlier calculations produce the derivatives of $v$. 

Recall that 
\begin{definition}\label{d:psub}
For a manifold with corner $M$, a subset $X$ is called a \emph{p-submanifold} if for any $p\in X$ there is a neighborhood $p\in U\subset M$ and local coordinates $\{x_{1}, \dots, x_{n}, y_{1}, \dots, y_{m}\}$ on $U$ such that $X\cap U$ is given by $\{x_{1}=\dots=x_{n}=0\}$.
\end{definition}

We have proved the 
\begin{proposition}\label{p:localv}
Fix any point $\frakq\in \del \Omega \subset \calA_J$ and suppose that $W$ is a subspace of $\RR^{2J} \cong 
\{ \sum_{\ell=1}^J A_\ell z^{-\ell},\ A_\ell \in \CC\}$.  Then there is a p-submanifold $\calW \subset \calE_{J}^{(J)}$ 
containing $P^{(J)}(\frakq)$ such that the differential of the function $v(\vec A; z)$ restricted to $\calW$ is equal 
to the subspace $W$.  
\end{proposition}

\subsection{The global parametrization}\label{ss:global}
We now formulate the global version of this result. Let $g_{0}$ be a spherical conical metric with conic data 
$\frakp_{0}=(p_{1}, \dots, p_{k})$ and $\vec \beta_{0}=(\beta_{1}, \dots, \beta_{k})$.  We reindex so as to 
list the cone angle parameters in decreasing order, i.e., so that 
\begin{equation}\label{e:angle}
\beta_{1}\geq \beta_{2}\geq \dots \geq \beta_{k_{0}}>1>\beta_{k_{0}+1}\geq \dots \geq \beta_{k}. 
\end{equation}
For each $j \leq k_0$, we allow $p_j$ to split into $[\beta_j]$ points, so altogether there are 
\begin{equation}\label{e:K}
K=k+\sum_{i=1}^{k_{0}} ([\beta_{i}]-1)=\sum_{i=1}^{k}\max\{[\beta_{i}], 1\}
\end{equation}
points after splitting.  For each $j \leq k_0$ choose splitting parameters $\vec B^{(j)} = (B_1^{(j)}, \ldots, B_{[\beta_j]}^{(j)})$ with
$\sum_i (B_i^{(j)}-1) = \beta_j-1$. We also set $\vec B^{(j)}=(B_1^{(j)}) = (\beta_j)$, $j > k_0$ and decompose the 
entire set $\vec B$ into clusters associated to each $p_j$, 
\begin{equation}\label{e:B}
\vec B = ( \vec B^{(1)}, \ldots, \vec B^{(k)}) \in (0,\infty)^{K},
\end{equation}
where each cluster $\vec B^{(j)}$ is interpreted as above.   Points $p_i$ with $i>k_{0}$ or $\beta_{i}<2$ do not split.  Same as the single cluster case, we require $\vec B$ to avoid  $\widehat{\Delta} = \cup_{j=1}^{k}\cup_{I} \widehat{\Delta}_I^{(j)}$, where
$I = \{i_1, \ldots, i_p\} \subset \{1, \ldots, [\beta_{j}]\}$ and $\widehat{\Delta}_I^{(j)} =\{\vec B: \, \sum_{i\in I} 2\pi (B_i^{(j)}-1)=0\}$. (That is, no subcluster merge to a point with angle $2\pi$.) 

\begin{definition}\label{def:B}
An angle vector $\vec B\in (0, \infty)^{K}$ is called admissible if it satisfies the constraints above, and the set of all such $\vec B$ is denoted $B$. 
\end{definition}

We next define a lift of $\frakp_0$, first to the point 
\[
(p_1, \ldots, p_1, p_2, \ldots, p_2, \ldots, p_{k_0}, \ldots, p_{k_0}, p_{k_0+1}, \ldots, p_k) \in M^K,
\]
where each $p_j$ with $j \leq k_0$ is repeated $[\beta_j]$ times. Finally, we choose a lift of this point to $\hat{\frakp}
=(\frakq^{(1)}, \ldots, \frakq^{(k)}) \in \calE_K(M)$, where each $\frakq^{(j)}$ is a lift of $(p_j, \ldots, p_j)$ to the interior of the 
central front face of $\calE_{[\beta_j]}$ for $j\le k_{0}$.  We can certainly assume that each $\frakq^{(j)}$ lies in the admissible
set $\Omega = \Omega(\calE_{[\beta_j]})$. (We are abusing notation slightly here, regarding each $\calE_{[\beta_j]}$ 
as a local factor in $\calE_K$.) For $j>k_{0}$, $\frakq^{(j)}=p_{j} \in \calE_{1}(M)=M$.  Recall that the lift of $\frakp_{0}$ to $M^{K}$ lies on the intersection of partial diagonals, where the blow up is done within each local factor. That is,  $\hat \frakp$ lies in a corner $F_{0}\subset \calE_{K}$ where locally there is a product structure, and the factors are $\calE_{[\beta_{j}]}, j\le k_{0}$ and $(k-k_{0})$ copies of $M$.  

Analogous to the map $\calF^{-1}$, we lift (one branch of) the initial map $\calF^{-1}$ as follows. Write 
$\CC^{K}=\prod_{j=1}^{k_{0}}\CC^{[\beta_{j}]} \oplus  \CC^{k-k_0}$, and define
\[
\calF^{-1}: (\vec A^{(1)}, \dots, \vec A^{(k_{0})}, A^{(k_{0}+1)}, \dots, A^{(k)}) \rightarrow (\vec z^{(1)}, \dots, \vec z^{(k_{0})}, z^{(k_{0}+1)}, \dots, z^{(k)}),
\]
where 
\[
\vec z^{(j)}=\calF^{-1}(\vec A^{(j)}), \ j\leq k_{0}\ \mbox{and}\ z^{(j)}=-A^{(j)}, \ j>k_{0}. 
\]
Taking each $\vec z^{(j)}$ as the local coordinate in $M^{[\vec \beta_{j}]}$, then $\calF^{-1}$ lifts to
\[
P^{(0)}: \calA_K \rightarrow \calE_{K}.
\]

As in the local case, $P^{(0)}$ is not a diffeomorphism at $\del \Omega$, so we perform the additional sequence of blowups in 
$\calE_{K}$ near $\hat \frakp$. Indeed, we replicate the iterative blowup from the single-cluster case in each factor 
$\calE_{[\beta_{j}]}, j\leq k_{0}$. When $j>k_{0}$, we simply blow up $p_{j}\in M$ in that factor. Because of the transversality, 
these operations can be performed in any order. We call the final space $\nwBase$;  it is locally given by 
\[
\prod_{j=1}^{k_{0}} \calE_{[\beta_{j}]}^{([\beta_{j}])} \times \prod_{j=k_{0}+1}^{k} [M; \{p_{j}\}], 
\] 
and the final lift of $\calF^{-1}$ is 
\begin{equation}\label{e:PK}
P^{(K)}: \prod_{j=1}^{k_{0}}[\CC^{[ \beta_{j}]};0] \times \prod_{j=k_{0}+1}^{k}[\CC;0] \rightarrow \calE_K^{(K)}.
\end{equation}
This is a local diffeomorphism, including up to $\del \Omega$.

\begin{proposition}\label{p:globalv}
Fix any point $\frakq$ on the front face of $\calA_K(M) := \prod_{j=1}^{k_{0}}[\CC^{[\beta_{j}]};0] \times \prod_{j=k_{0}+1}^{k}[\CC;0]$,
and not lying on the codimension two subvarieties $\calT$. Suppose furthermore that $W$ is a subspace of
\begin{equation}\label{e:AK}
\RR^{2K} \cong 
\prod_{j=1}^{k_{0}}\{ \sum_{\ell=1}^{[\beta_{j}]} A_\ell^{(j)} z^{-\ell},\ A_\ell^{(j)} \in \CC\} \times \prod_{j=k_{0}+1}^{k}\{A^{(j)}z^{-1}, 
\ A^{(j)}\in \CC\}.
\end{equation}
Then there is a p-submanifold $\calW \subset \nwBase$ containing  $P^{(K)}(\frakq)$ such that the differential of the function $v(\vec A; z)$
restricted to $\calW$ equal the subspace $W$.  
\end{proposition}

\subsection{Examples of cone point splitting}
We now illustrate the ideas and calculations above with some explicit calculations when $J =1, 2$ or $3$.
As before, we work locally in the disk near a single cone point. 

\quad

\noindent\textbf{Example 1 \ \  $2\pi\beta_{0} \in (2\pi, 4\pi)$: } In this simplest case, $J=[\beta_{0}]=1$,
and hence the point $p_0$ moves rather than splits. The family of conformal factors in this case is 
$$
v(\vec A; z)=\log|z+A_{1}|, 
$$
so, writing $A_{1}=\beta_{0}^{1/\beta_{0}}(e_{1}'+ie_{1}'')$, this has infinitesimal variation
\begin{gather*}
\frac{\partial v}{\partial A_1}(0)  =\tfrac12 \Re(1/z),\ \ \mbox{i.e.,}\ \ 
\frac{\partial v}{\partial  e_1'}(0)= \cos\theta \, r^{-1/\beta_{0}},  \ 
\frac{\partial v}{\partial  e_1''}(0)= \sin\theta \, r^{-1/\beta_{0}}.
\end{gather*}
This computation is independent of the phase of $A_{1}$. 

\quad 

\noindent\textbf{Example 2 \ \  $2\pi\beta_{0}\in (4\pi, 6\pi)$:}  Now $p_0$ splits into $2$ cone points with
an admissible pair of cone angles $2\pi B^{0}_{1}$ and $2\pi B^{0}_{2}$, i.e., $(B^{0}_{1}-1)+(B^{0}_{2}-1)=\beta_{0}-1$.
Set $2\frac{B_{0}^{i}-1}{\beta_{0}-1}=\frakb_{i}$ and solve for the functions $z_{i}(\vec A)$, $i = 1,2$, such that 
$$
v(\vec A; z)=\frakb_{1}\log|z-z_{1}(\vec A)|+\frakb_{2}\log|z-z_{2}(\vec A)|.
$$
satisfies 
$$
\log |P(\vec A;z)|=\log |z^{2}+A_{1}z+A_{2}|=v(\vec A;z) + \calO(|\vec A|^{1+\epsilon}). 
$$
This leads to the system of equations
\[
\begin{split}
\frakb_{1}z_{1}+\frakb_{2}z_{2}&=\lambda_{1}+\lambda_{2}=R_{1}(\vec A)=-A_{1}\\
\frakb_{1}z_{1}^{2}+\frakb_{2}z_{2}^{2}&=\lambda_{1}^{2}+\lambda_{2}^{2}=R_{2}(\vec A)=(A_{1})^{2}-2A_{2}.
\end{split}
\]
Since $\frakb_{1}+\frakb_{2}=2$, the restriction that $\vec b\notin \widehat{\Delta}=\{\vec \frakb: \frakb_{1}+\frakb_{2}=0\}$
is vacuous. This system has two solutions: 
\[
 ( z_{1}^{(1)}, z_2^{(1)}) = \left(\frac{-A_{1}+ \sqrt{((A_{1})^{2}-4A_{2})\frakb_{2}/\frakb_{1}}}{2}, \frac{-A_{1}- \sqrt{((A_{1})^{2}-4A_{2})\frakb_{1}/\frakb_{2}}}{2}\right)
\]
and
\[
 ( z_{1}^{(2)}, z_2^{(2)}) = \left(\frac{-A_{1}- \sqrt{((A_{1})^{2}-4A_{2})\frakb_{2}/\frakb_{1}}}{2}, \frac{-A_{1}+ \sqrt{((A_{1})^{2}-4A_{2})\frakb_{1}/\frakb_{2}}}{2}\right).
  \]
  The weighted discriminant locus $\calS_{\vec\frakb}=\{\vec A: 4A_{2}-(A_{1})^{2}=0\}$ is independent of $\vec \frakb$ in this special
  case, and corresponds to solutions on the diagonal $z_{1}=z_{2}=-\frac{A_{1}}{2}$. The map
\[
\calF: \CC^{2}\rightarrow \CC^{2}, \qquad (z_{1}, z_{2})\mapsto (A_{1}, A_{2})
\]
is a $2$-to-$1$ branched cover ramifying along the diagonal $\{z_1 = z_2\}$. 
Finally, the two local inverses to $\calF$ are given by the explicit formul\ae\ above. These also imply that
$|z_{i}|\leq C \max\{|A_{1}|, |A_{2}|^{1/2}\}$ for any fixed $\vec \frakb$. For the construction of $\calE_{2}^{(2)}$, see previous subsection.

\quad

\noindent\textbf{Example 3 \ \  $2\pi\beta_{0}\in (6\pi, 8\pi):$}  Fix an admissible triple $(B_0^1, B_0^2, B_0^3)$ and set $\frakb_{i}=3\frac{B_{0}^{i}-1}{\beta_{0}-1}$, $i=1,2,3$.  The functions $z_i(\vec A)$ in 
$$
v(\vec A;z)=\sum_{i=1}^{3}\frakb_{i}\log|z-z_{i}(\vec A)|,
$$
satisfy the set of equations
\[
\begin{split}
\sum b_{i}z_{i}&=\sum \lambda_{i}=-A_{1}\\
\sum b_{i}z_{i}^{2}&=\sum \lambda_{i}^{2}=(A_{1})^{2}-2A_{2}\\
\sum b_{i}z_{i}^{3}&=\sum \lambda_{i}^{3}=-(A_{1})^{3}+3A_{1}A_{2}-3A_{3}
\end{split}
\]

If $\vec \frakb \notin \widehat\Delta$, i.e., $\frakb_i + \frakb_j \neq 0$ for $i \neq j$, these equations have $3! = 6$ solutions
in $\CC^{3}$, counted with multiplicity.  The map 
$$
\calF: \CC^{3}\setminus\{\mbox{partial diagonals}\}\rightarrow \CC^{3}\setminus\calS_{\vec \frakb}
$$
is a $6$-to-$1$ covering. Unfortunately, it is no longer so easy to find an explicit expression for the weighted discriminant
locus $\calS(\vec \frakb)$ in this case. 

However, in the special case that all $\frakb_{j}=1$, the six solutions $\vec Z^{(i)}$ are the rearrangements of the roots
of the polynomial $z^{3}+A_{1}z^{2}+A_{2}z+A_{3}=0$.  The bound $\max\{|z_{j}|\}\leq C \max\{|A_{i}|^{1/i}\}$
follows from the explicit formula for the roots of a cubic. 

We now compute the asymptotic expansions of $z_{i}$. 
Write $z_{i}=\sum_{j=1}^{3} \rho^{j}c_{ij}+\calO(\rho^{4})$ and plug into the equations~\eqref{e:Rell}, then we get
\[
\begin{split}
-\tilde A_{1}\rho^{3}&=\sum_{i=1}^{3} \frakb_{i} [ c_{i1}\rho+ c_{i2}\rho^{2}+c_{i3}\rho^{3}]\\
-2\tilde A_{2}\rho^{3}&=\sum \frakb_{i} [c_{i1}^{2}\rho^{2}+2 c_{i1}c_{i2}\rho^{3}+(2c_{i1}c_{i3}+c_{i2}^{2})\rho^{4}]\\
-3\rho^{3}e^{i\theta}&=\sum \frakb_{i} [c_{i1}^{3}\rho^{3}+3c_{i1}^{2}c_{i2}\rho^{4}+(3c_{i1}^{2}c_{i3}+3c_{i1}c_{i2}^{2})\rho^{5}]
\end{split}
\]
Then one can solve for $c_{ij}$ iteratively as described above. Below we give an explicit computation for the case when all $\frakb_{i}=1$.

We first solve for $\{c_{i1}\}_{i=1}^{3}$ which satisfy
$$
\sum_{i}  c_{i1}=0, \ \sum_{i}  c_{i1}^{2}=0, \ \sum_{i}  c_{i1}^{3}=-3e^{i\theta}
$$ 
We choose one of the six solutions (which come from permutations):
$$
c_{j1}=-e^{i\theta/3}\tau^{j}, \ j=1,2, 3
$$
where $\tau=\frac{-1+\sqrt{3}i}{2}$.
We then solve for $\{c_{i2}\}$ which satisfy
$$
\sum c_{i2}=0, \ \sum 2c_{i1}c_{i2}=-2\tilde A_{2}, \ \sum 3 c_{i1}^{2}c_{i2}=0
$$
which gives
$$
c_{12}= -\frac{1 + i \sqrt{3}}{6}\tilde A_{2} e^{-i\theta/
   3},\ c_{22} = \frac{3 i + \sqrt{3}}{3 (-3 i + \sqrt{3})} \tilde A_{2}  e^{-i\theta/
   3}, \ c_{32} = \frac{1}{3} \tilde A_{2} e^{-i\theta/
   3}
$$
Then the equations for $\{c_{i3}\}$ are
$$
\sum  c_{i3}=-\tilde A_{1}, \ \sum (2c_{i1}c_{i3}+c_{i2}^{2})=0, \ \sum (3c_{i1}^{2}c_{i3}+3c_{i1}c_{i2}^{2})=0
$$
which gives
$$
c_{13}=c_{23}=c_{33}=-\frac{\tilde A_{1}}{3}.
$$


\section{The obstruction subbundle and projected solutions}
Our next step is to construct families of solutions of the Liouville equation modulo the finite dimensional space of eigenfunctions
\[
E_2 = \{\phi \in \calD^{m,\alpha}_{\Fr}(M_{\frakp_0}):  \Delta_{g_0} \phi = 2\phi\}.
\]
These will be called projected solutions. The remainder of the argument, in the next section, consists in identifying the subfamilies
which can be deformed to exact spherical conic metrics.  

The difficult questions surrounding the parametrization of $K$-tuples of points by vectors $\vec A$ described in the last section
do not play a role here, so we are able to work exclusively on $\calE_K$ and $\calC_K$ here, and lift to $\tilde{\calE}_K$ and
the corresponding space $\widetilde{\calC}_K$ only at the end. 

\subsection{The fibers near the central face}
Following \S\ref{ss:global}, consider a spherical cone metric $g_0$ with conic data $\frakp_{0}=(p_{1}, \dots, p_{k})$ and $\vec \beta$. This uniquely 
determines an `exploded point'
$$
\frakq_0 = (\underbrace{p_{1}, \dots, p_{1}}_{[\beta_{1}]}, \ \underbrace{p_{2}, \dots, p_{2}}_{[\beta_{2}]}, \ \dots, \ \underbrace{p_{k_{0}}, 
\dots, p_{k_{0}}}_{[\beta_{k_{0}}]}, \  p_{{k_{0}+1}}, p_{k_{0}+2}, \dots, p_{k})\in M^{K}.
$$
Any nearby point $\frakq \in M^K$ determines a set of clusters $\frakq^{(i)}_j$, $i = 1, \ldots, k_0$, $j = 1, \ldots, [\beta_i]$, 
where $q^{(i)}_1, \ldots, q^{(i)}_{[\beta_i]}$ all lie in a small neighborhood of $p_i$, along with the remaining isolated points 
$q_i$, $i = k_0+1, \ldots, k$, each lying near the corresponding point $p_i$.
The set of lifts of $\frakq_0$ to $\calE_K$ fills out a corner 
\begin{equation}\label{e:F0}
F_{0}:=\bigcap_{i=1}^{k_{0}} F_{12\dots [\beta_{i}]}^{i} \subset \calE_K,
\end{equation}
where $F^{i}_{12\dots[\beta_{i}]}$ is the face arising from blowing up the partial diagonal $\{q^{i}_{1}=\dots=q^{i}_{[\beta_{i}]}\}$. We denote points on $\calE_{K}$ by $\frakq$. The corresponding 
lift $\pi_K^{-1}(\frakq) \subset \calC_K$ (where $\pi_{K}: \calC_{K}\rightarrow \calE_{K}$) is a union of hemispheres, each lying over $F_0$,  attached 
in succession to the punctured surface $M_{\frakp_{0}}$, cf. \S 3 and, for further details, \cite{MZ}.  

We fix a neighborhood $\calU$ of  $F_0$ in $\calE_{K}$, and set $\calV=\pi_{K}^{-1}(\calU)$.  If $\frakq\in \calU$, then the fiber $\pi_{K}^{-1}(\frakq)$ 
contains, as one of its constituents, the surface $M$ blown up at the points $q_j$ of $\beta_{K}(\frakq)\in M^{K}$. These points lie in two classes: 
\begin{itemize}
\item  
when $i > k_0$, the cone angles at the initial points $p_i$ are less than $2\pi$, so the corresponding points $q_i$ move without splitting; 
\item on the other hand, if $i \leq k_0$, then the cone angles at $p_i$ and at the points of the associated cluster $q_i^{(j)}$, $j = 1, \ldots, [\beta_{j}]$
which have split from $p_i$, are greater than $2\pi$.  
\end{itemize}

Fix an admissible set of cone angle parameters $\vec B \in B$ (see Definition~\ref{def:B} for the definition of being admissible). We now produce, for each $\frakq\in \calU$, a spherical conic metric $g_{\frakq, \vec B}$ 
on the regular part of the fiber $\pi_{K}^{-1}(\frakq)$ which solves the Liouville equation modulo a certain finite dimensional obstruction subspace.

In the following, we use weighted $b$-H\"older spaces $\calC^{m,\alpha}_b$ on each fiber; these are the restrictions
of the space $\calC_{b}^{m,\alpha}(\calC_{K})$ to that fiber, see \cite[Lemma 5]{MZ}.

\subsection{The first approximation}\label{ss:initial}
Fix a smooth (nonconic) metric $h_{0}$ on $M$ which is flat in balls $\calB_j(\epsilon)$ containing
$p_j$, $j \leq k_0$.  Next define the family of (nonconstant curvature) conic metrics $g_{1} = g_{1}(\vec B, \frakq)$,
parametrized by $\vec B$ and $\frakq\in \calU$, which equals $g_{1}=e^{2v(\vec B, \frakq)}h_{0}$ in each $\calB_j$; here
\begin{equation}\label{e:v}
v(\vec B,\frakq)=\sum_{i=1}^{[\beta_{j}]} (B_{i}^{j}-1)\log |z-q_{i}^{j}|.
\end{equation}
We can arrange, for simplicity, that near each $p_j$, $j > k_0$, $g_1$ is spherical with cone angle $2\pi B^j$. This family of
metrics is polyhomogeneous on $\pi_{K}^{-1}(\calU)$, cf.\ Theorem 1 in~\cite{MZ}.

We next modify $g_1$ to a new family $g_2(\vec B, \frakq)=e^{2\tilde v(\vec B,\frakq)} g_{1}$ which has the same conic data, but 
such that $g_2$ is the original conic metric $g_{0}$ on the constituent $M_{\frakp_{0}}$ in the singular fiber $\pi_{K}^{-1}(\frakq_0)$.
The function $\tilde v$ on $\pi_K^{-1}(\calU)$ is first determined on the central fibers over $F_{0}$ by the equation
\begin{equation}\label{e:vtilde}
e^{2\tilde v(\vec B, \frakq_{0})}g_{1}|_{M_{\frakp_{0}}}=g_{0}, 
\end{equation}
and then extended to be polyhomogeneous over the whole neighborhood $\calV$. We can also arrange for this extension to satisfy 
$\tilde v(\rho)-\tilde v(0)=\calO(\rho_{j}^{J+\epsilon})$ where $\rho_{j}$ is the boundary defining function of front face $F^{j}$ 
corresponding to the cluster of $\beta_{j}$ and $J=\max\{[\beta_{j}], 1\}$.

\subsection{Projected solution family}
We now solve the Liouville equation up to a finite rank error on all fibers near $\pi_K^{-1}(\frakq_0)$.   To do this, we first 
construct a vector bundle $\EE_2$ over $\calU \subset \calE_K$ which extends the $\ell$-dimensional eigenspace 
$E_2 = \ker (\Delta_{g_0} - 2)$ on $M_{\frakp_0}$. Specifically, first pull back $E_2$ to a trivial rank $\ell$ vector bundle
on the face $F_0$ lying over $\frakp_0$. Then extend this trivial bundle smoothly to $\calU$.  This extension is not yet
well adapted to the family of Laplacians, so we arrange this next. 

For any $\frakq \in \calU$, consider the resolvent $(\Delta_{g_2(\frakq)} - \lambda)^{-1}$ of the Friedrichs extension of the Laplacian.
By standard eigenvalue perturbation theory, there exists some small $\epsilon > 0$ so that, shrinking $\calU$ if necessary,
then the spectrum of $\Delta_{g_2(\frakq)}$ does not intersect the loop $\gamma = \{|\lambda - 2| = \epsilon\}$.  We then define
\[
\Pi_{\frakq}=(2\pi i)^{-1} \int_{\gamma} (\Delta_{g_{2}( \frakq)} -\lambda)^{-1} \, d\lambda;
\]
this is an $L^2$-orthogonal projector onto the sum of eigenspaces for all eigenvalues inside $\gamma$. Its range $\EE_2$ 
is a smooth rank $\ell$ bundle, with $\EE_2|_{\frakq} \subset \calD^{m,\alpha}_{\Fr}(M_{\pi(\frakq)})$. Furthermore,
$\Pi_{\frakq}^\perp = \mbox{Id} - \Pi_{\frakq}$ projects onto the complementary finite codimensional subspace 
$\EE_2|_{\frakq}^\perp \subset \calC^{m,\alpha}_b(M_{\pi(\frakq)})$. 

\begin{proposition}\label{p:eigen}
For each $\frakq\in \calU$, there exists a unique $u\in \EE_2^{\perp}|_{\frakq}$ and $f\in \EE_2|_{\frakq}$ such that 
\begin{equation}\label{e:eigenerror}
\Delta_{g_{2}}u - e^{2u}+K_{g_{2}}=f.
\end{equation}
Both $u$ and $f$ depend smoothly on $\frakq \in \calU$. 
\end{proposition}
\begin{proof}
By construction, if $\frakq \in \calU$, the linearization in $u$ of 
\[
(\frakq,u) \mapsto N(\frakq,u) := \Pi_{\frakq}^\perp \circ (\Delta_{g_2(  \frakq)} u - e^{2u} + K_{g_2})
\]
is an isomorphism $\Pi^\perp_{\frakq} \calD^{m,\alpha}_\Fr(M_{\pi(\frakq)}) \to \Pi_{\frakq}^\perp\calC^{m,\alpha}_b(M_{\pi(\frakq)})$. 
Furthermore, $N(\frakq_0, 0) = 0$.  Since $N$ depends continuously on $\frakq$, $N(\frakq,0)$ remains small in norm for
$\frakq \in \calU$. Using the invertibility of this linearization, a standard contraction argument produces both the solution $u$ and 
the error term $f$.   Obviously, we could simply invoke the inverse function theorem, but for the arguments in the next subsection
it is helpful to recall that this relies on a contraction. 
\end{proof}

\subsection{Polyhomogeneity of projected solution family}\label{ss:poly}
\begin{proposition}\label{p:polyhomo}
The solution $u$ to~\eqref{e:eigenerror} is polyhomogeneous on $\calC_K$.
\end{proposition}
The detailed proof of this same result for flat and hyperbolic metrics appears in the lengthy \cite[Section 6]{MZ}.
The present setting differs only very slightly, because of the finite corank projection. Hence we shall sketch
the argument only briefly since the modifications needed are very minor.  In fact, near conic points which do not split, 
that proof carries over verbatim. Thus we focus on a neighborhood of some point which splits.  For simplicity
we describe the proof in the simplest situation where one conic point splits into two.  The general case
proceeds in the same way, but more steps. 

\begin{lemma}\label{l:2pt}
If $\beta_{1}\in (2,3)$, which implies that $p_1$ splits into two points, then the solution family $u$ is polyhomogeneous on $\calC_K$. 
\end{lemma}
\begin{proof}
We use the coordinates as labelled in Figure~\ref{f:coordinates}, and recall how we can produce the successive terms in the polyhomogeneous 
expansion of the solution family. 
\begin{figure}[h]
\centering
 \includegraphics[width=0.5\textwidth]{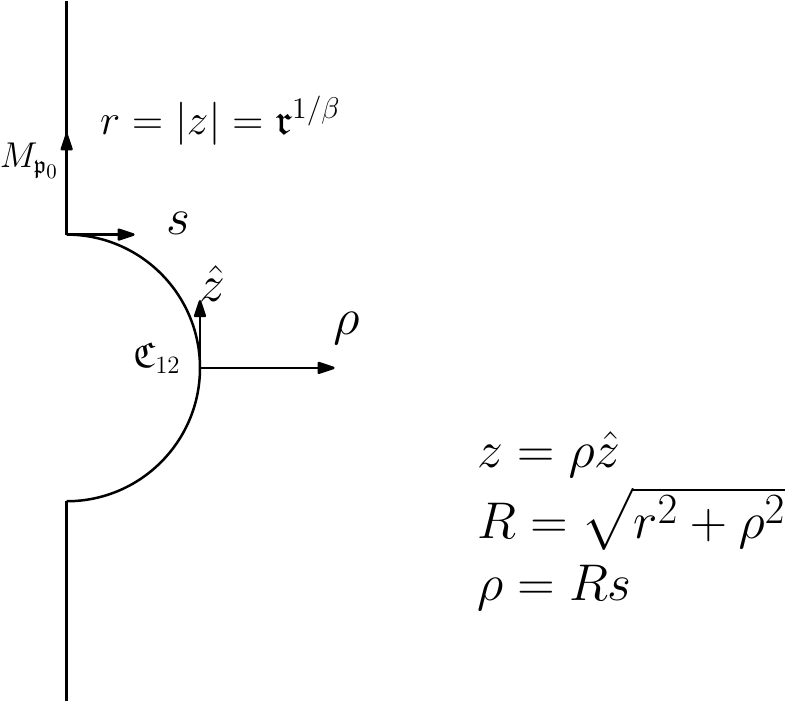}
 \caption{coordinates used in the computation of expansion near the singular fiber $\frakC_{12} \bigcup M_{\frakp_{0}}$}
 \label{f:coordinates}
\end{figure}

\noindent\textbf{Step 1: The initial metric}
The metric $g_1(\vec B, \frakq) = e^{2v(\vec B, \frakq)}h_0$ is flat near $\frakC_{12} \subset \calC_2$.  This is precisely the 
situation in \cite[Eqn. (48)]{MZ} (where this metric was called $g_{0,\frakp}$). Its polyhomogeneity is proved there. 

\noindent\textbf{Step 2:  Expansion at the central face $M_{\frakp_{0}}\cap \frakC_{12}$ in $M_{\frakp_{0}}$.} 
The solution of the curvature equation on the central fiber $M_{\frakp_{0}}$
$$
\Delta_{g_1}\tilde v-e^{2\tilde v}+K_{g_1}=0;
$$
provides the conformal factor $\tilde v$ on $M_{\frakp_{0}}$ in~\eqref{e:vtilde}.  The only difference with~\cite{MZ} is
that here $e^{2\tilde v}g_{1}=g_{0}$ is a priori fixed on the central fiber.  
Choosing appropriate local holomorphic coordinates near the cone point, we obtain the expansion of $\tilde v$ 
on approach to this boundary in $M_{\frakp_{0}}$ as in~\cite[Lemma 6]{MZ}, to get
$$
\tilde v|_{M_{\frakp_{0}}}\sim \sum_{j\in \NN_{0}}a_{j}\frakr^{2j}.
$$ 

\noindent\textbf{Step 3: Expansion at the face $M_{\frakp_{0}}$.}
We next extend $\tilde v$ away from this face and solve the projected curvature equation near $M_{\frakp_{0}}$ but 
away from $\frakC_{12}$. This uses the invertibility of operator $(\Delta_{g_{2}}-2)|_{M_{\frakp_{0}}}$ on 
$\EE_2^\perp$, cf.\ Proposition~\ref{p:eigen}.  Proceeding on as in \cite{MZ}, we obtain that 
$$
\tilde v+ u\sim \sum_{j=0}^{\infty} s^{j}\tilde u_{j}',
$$
where, for $j\geq 1$,
$$
\tilde{u}_{j}'\sim \sum_{\ell\in \NN} r^{j+\ell} a_{j\ell 0}(\phi) + \sum_{\ell, k \in \NN, \ell\geq 0, k\geq 1} r^{j+\ell+2k \beta} a_{j\ell k}(\phi).
$$
Compatibility with the previous step is the fact that $\tilde u_{0}'$ equals the function $\tilde v|_{M_{\frakp_{0}}}$ in Step 2. 

\noindent\textbf{Step 4: Expansion at $\frakC_{12}$}
To extend to an expansion at $\frakC_{12}$, recall from \cite{MZ} that the writing things in terms of the projective coordinates near this face
rescales the linearized Liouville equation at $\frakC_{12}$ to the Laplacian for a flat conic metric on that face, see \cite[Lemmas 7-11]{MZ}. 
The regularity theory for solutions of this equation yields
\begin{equation}\label{e:u}
\tilde v+u\sim \sum_{\alpha=j+2k\beta, \ell\leq j}  R^{\alpha}s^{\ell}u_{\alpha \ell}(\phi).
\end{equation}

\noindent\textbf{Step 5: Polyhomogeneity on $\calC_{K}$}
Putting these steps together, we obtain the entire expansion for $u$.  What remains is to show that this is the actual
asymptotic expansion for the solution family we obtained earlier.  In other words, we must show that the difference
of $u$ and any finite partial sum of this expansion is conormal and vanishes at a rate just larger than the last term
in this partial sum.  This too is carried out in the same way as in \cite[Theorem 1]{MZ}.
\end{proof}
The general case of Proposition~\ref{p:polyhomo} is proved along the same lines, cf.\ \cite[Theorems 2 and 3]{MZ}. We extend the 
expansion iteratively on each of the faces along the tower of hemispheres in the singular fibers of $\calC_K$. Each step is carried 
out essentially the same as in the two-point case above. 

Now let us return to the fact, discussed at the very beginning of this section, that we actually need to consider solutions
over the final iterated blowup $\widetilde{\calE}_K$ rather than just $\calE_K$.  More specifically, the extended configuration
family $\calC_K$ can be lifted by the blowdown map $\widetilde{\calE}_K \to \calE_K$ to a space we call $\widetilde{\calC}_K$.
There is a lifted map $\widetilde{\pi}_K: \widetilde{\calC}_K \to \widetilde{\calE}_{K}$, and the added complexity is just in the base.

The fact that solution family $u$ is polyhomogeneous already on $\calC_K$ immediately implies that its lift
$\widetilde{\pi}_K^* u$ is polyhomogeneous on $\widetilde{\calC}_K$.    This is the fact that will be needed in the next section.

\section{The solution space}
The final step is to identify the points $\frakq \in \calE_K$ where the error term $f$ in~\eqref{e:eigenerror} vanishes. These correspond 
to the configurations of conic points such that the projected solution $g=e^{2u}g_{2}$ is actually spherical.     The way we do
this is as follows.  In the last section we found solutions up to a finite dimensional error, so the problem reduces to one of
understanding when this error vanishes.   Suppose, for the moment, that $\EE_2$ is one-dimensional, and spanned by
the eigenfunction $\phi$ for $\frakq \in \calU$.  The defect is then of the form $\Lambda(\frakq) \phi_\frakq$, where
$\Lambda(\frakq)$ is a scalar function.  Clearly $\Lambda$ vanishes on $F_0$, and we seek to compute where it vanishes
in the interior.  We compute the derivative of $\Lambda$ along any one of the curves $\vec Z(t) = \calF^{-1}(t\vec A)$ discussed at length in
\S 6. This derivative turns out to have an exceptionally pretty form: it is given by a symplectic pairing between the asymptotic
coefficients of the eigenfunction $\phi$ and of the (flat) conformal factor $v$ with the real and imaginary parts of the
constants $\vec A$.  Of course, in order for this parametrization in terms of $\vec A$ to be nonsingular, we must pass from $\calE_K$
to $\widetilde{\calE}_K$.  The end result is that the kernel of this pairing (for fixed $\phi$ and $v$) determines a codimension
one submanifold which meets the front face $\widetilde{F}_0$ transversely.  We also show that $\Lambda$ vanishes only
to first order along this submanifold. Taken together, this establishes the existence of a smooth codimension one
submanifold in $\widetilde{\calE}_K$ where the problem is solved exactly.  If the rank of $\EE_2$ equals $\ell > 1$,
the same considerations lead to the existence of a codimension $\ell$ submanifold of exact solutions. 

Suppose then that $\ell=1$, i.e., the eigenspace $E_2$ for $g_{0}$ is one-dimensional, so $\EE_{2}$ is a (real) line bundle over $\calU$.  
As will become clear below, it is necessary to work over the base $\widetilde{\calE}_K$, so we lift the solution family $u$,
the neighborhood $\calU$ and the line bundle $\EE_2$ up to this larger space via the blowdown map.   If $\frakp$ is a $K$-tuple
of points on $M$, possibly with multiplicities, we write $\frakq$ for some point over it in $\calE_K$ and $\tilde{\frakq}$ for
a point over it in $\widetilde{\calE}_K$.   As usual, write $[M; \frakp] = M_{\frakp}$, 
and denote by $\phi_{\frakp}$ the associated eigenfunction, which is unique up to scale because $\ell = 1$. We normalize it to have 
$L^2$ norm $1$. To simplify notation below, $v + \tilde{v} = \hat{v}$; this depends on $\frakq \in \calU$.  
Since $K_{g_{2}}=\Delta_{g_{2}} \hat{v}+e^{-2\hat{v}}K_{h_{0}}$, we write \eqref{e:eigenerror} as
\begin{equation}\label{e:g1'}
\Delta_{g_{2}}u - e^{2u}+\Delta_{g_{2}} \hat{v}+e^{-2\hat{v}}K_{h_{0}}=\Lambda_{\frakq}\phi_{\frakq}.
\end{equation}
This identifies the error term $f$ as $\Lambda_{\frakq}\phi_{\frakq}$ for some $\Lambda_{\frakq}\in \RR$.  As above,
we lift this equation and all these functions up to $\widetilde{\calE}_K$. In particular, we regard $\Lambda_{\tfrakq}$ 
as a function on $\widetilde{\calU}$. 

Again for simplicity, we consider the special case where there is only one cluster of points; we explain how to
carry this over to the general case at the end. 

Our goal is to find the entire locus where $\Lambda_{\tfrakq} = 0$. Note that $\Lambda_{\tfrakq_{0}}=0$ on the
face $\tilde{F}_0$, so if there is an additional submanifold $V$ which is transverse to this face and on which
$\Lambda_{\tfrakq}$ vanishes, then this function must vanish to second order at $V \cap \tilde{F}_0$.   

Differentiate $\Lambda_{\tfrakq}$ with respect to $\tfrakq$.  For the moment, fix a nonzero vector $\vec A \in \CC^K$ as in
\eqref{e:AK}, such that $\sum_i |A_{i}|^{2}=1$ and let $\vec A(t)=(tA_{1}, \dots, tA_{K})$, be a path in $\CC^{K}_{A}$ and $\tilde A(t)$ the 
lifted path in $\calA_{K}(M)$.   We assume that $\tilde A(t)$ intersects the front face $\calA_{K}(M)$ in the open set $\Omega$ (i.e., the 
complement of the finite number of codimension two subvarieties) where the lifted parametrization $P^{(K)}$ of $\widetilde{\calE}_K$ by $\vec A$ 
is nonsingular, see \eqref{e:PJ}. Define $\gamma(t)=P^{(K)}(\tilde A(t))$ to be the corresponding path in $\nwBase$. Since $\gamma(t)$ 
lies in the interior of $\nwBase$ for $t>0$, we can write $\gamma(t)=\vec z(t) =(z_{1}(t), \dots, z_{K}(t))\in M^{K}$.  Assume the first $(z_{1}. \dots, z_{J})$ gives the cluster of points. 

We now use the chain rule to compute that the derivative of 
$$
v(\gamma(t)):=\sum_{i=1}^{J}\frakb_{i}\log|z-z_{i}(t)|=\log|z^{J}+t(A_{1}z^{J-1}+\dots+A_{J})|+\calO(t^{1+\epsilon}). 
$$
with respect to $t$ equals
$$
\frac{d v(\gamma(t))}{dt}|_{t=0}=\Re\sum_{\ell=1}^{J} \frac{A_{\ell}}{z^{\ell}}.
$$

Using dots to indicate infinitesimal variation with respect to $t$, and dropping various subscripts (like $\tfrakq$) to
unclutter notation, and recalling that $e^{-2\hat{v}}\Delta_{h_{0}}=\Delta_{g_{2}}$, we get 
\[
\left(-2\dot{\hat{v}} \Delta_{g_2}(u+\hat{v}) + \Delta_{g_2}(\dot u + \dot{\hat{v}})\right) 
- 2\dot u e^{2u} -2\dot{\hat{v}} e^{-2\hat{v}}K_{h_{0}}  =\dot \Lambda \phi + \Lambda\dot \phi.
\]
Taking the inner product with $\phi$ and integrating yields
\begin{multline*}
\int_{M} -2\dot{\hat{v}} \Delta_{g_{2}}(u+\hat{v}) \phi \, dA_{g_{2}}+ \int_{M} 
\Delta_{g_{2}}(\dot u + \dot{\hat{v}}) \phi\, dA_{g_{2}} \\
 - \int_{M} 2\dot u e^{2u} \phi \, dA_{g_{2}} -\int_{M} 2\dot {\hat{v}} e^{-2\hat{v}}K_{h_{0}}  \phi \, dA_{g_{2}}\\
=\dot \Lambda \int_{M} |\phi|^{2} \, dA_{g_{2}} + \Lambda \int_{M} \dot \phi \phi \, dA_{g_{2}}.
\end{multline*}
At any point on $\tilde F_0$, $\hfrakq$ projects to $\frakp_0$ and $\Lambda= u = 0$, so at this face, 
\begin{equation}
\label{e:Adot}
\begin{split}
\dot \Lambda & =  \int -2\dot{\hat{v}} (\Delta_{g_{2}} \hat{v})\phi \, dA_{g_{2}}  \\ & + \int \Delta_{g_{2}}(\dot u + 
\dot{\hat{v}}) \phi \, dA_{g_{2}} - 2 \int (\dot u + \dot{\hat{v}} e^{-2\hat{v}}K_{h_{0}})  \phi \, dA_{g_{2}}.
\end{split}
\end{equation}

By Green's theorem, 
\begin{equation*}
\begin{split}
\int_{M_\frakp} \Delta_{g_{2}}\dot{\hat{v}} \phi \, dA_{g_{2}} =  \int_{M_\frakp} 
\dot{\hat{v}} \Delta_{g_{2}} \phi \, dA_{g_{2}} 
+ \lim_{\epsilon \to 0} \int_{r=\epsilon} \left(\dot{\hat{v}} \del_r \phi - \del_r  \dot{\hat{v}} \phi\right) \, r d\theta.
\end{split}
\end{equation*}
Inserting this into~\eqref{e:Adot} and using the two equalities $\Delta_{g_2} \phi = 2\phi$ and $-\Delta_{g_{2}}\hat{v} +1- e^{-2\hat{v}}K_{h_{0}} = 0$ 
(since $K_{g_2} = K_{g_0} = 1$), we see that almost all terms cancel and we are left simply with
\begin{equation}\label{e:preLambda}
\dot \Lambda=\int (\Delta_{g_{2}}\dot u-2\dot u)\phi + \lim_{\epsilon \to 0} \int_{r=\epsilon} \left( \dot{\hat{v}} \del_r \phi - 
\del_r \dot{\hat{v}} \phi \right) \, r d\theta.
\end{equation}

We now show that the first term vanishes, i.e.,
\begin{equation}\label{e:udot}
\lim_{\tfrakq\rightarrow \tfrakq_{0}}  \int (\Delta_{g_{2}}\dot u-2\dot u)\phi =0.
\end{equation}
For this we use the expansion of $u$ from the previous section. Recall that if $\rho$ is a boundary defining function for $\tilde{F}_0$, then 
$\rho=t^{1/J}$. 

Let $u'=\frac{\partial u}{\partial \rho}$ and note that $\dot u=u'  (\del \rho/\del t) = J^{-1} u' t^{1/J-1}$.  We claim that
\begin{equation}\label{e:uprime}
\int (\Delta_{g_{2}}u'-2u')\phi=\calO(\rho^{J-1+\epsilon}).
\end{equation}
If this is true, then by the chain rule, 
$$
\lim_{t\rightarrow 0} \int (\Delta_{g_{2}}\dot u-2\dot u) \phi=\lim_{t\rightarrow 0} \frac{\partial \rho}{\partial t} \int (\Delta_{g_{2}}u'-2u')\phi =0
$$
since the variables $\rho$ and $t$ are constant on the fibers of $\nwTotl$, and hence commute with $\Delta_{g_{2}}$.

We now prove the claim \eqref{e:uprime}. 
Observe first that
$$
\int (\Delta_{g_{2}}-2)u' \phi= \lim_{\epsilon\rightarrow 0} \int_{r=\epsilon}(u'\partial_{r}\phi - \partial_{r}u' \phi)rd\theta+ \int u' (\Delta_{g_{2}}-2)\phi.
$$
The integral in the middle vanishes since $u'$ decays sufficiently near each cone point.  On the other hand, the integral on the right equals 
$\int u' (\lambda-2)\phi$, so it suffices to show that the function $\rho \mapsto \lambda(\rho)$ satisfies
\begin{equation}\label{e:lambda}
\lambda_\rho-2=\calO(\rho^J),
\end{equation}
where $\lambda_\rho = \del_\rho \lambda$.    We prove this by inductively showing that each of the first $J-1$ derivatives of $\lambda$
vanish at $\rho = 0$.  This in turn relies on the
\begin{lemma}\label{l:vdot}
$$
\frac{\partial^{k}\hat v}{\partial\rho^{k}}(0,z) \equiv 0, k=1, \dots, J-1.
$$
\end{lemma}
\begin{proof}
We initially extended $\tilde v$ from the front face so that $\tilde v(\rho,z)-\tilde v(0,z)=\calO(\rho^{J+\epsilon})$. Thus it suffices 
to prove the vanishing of the first $J-1$ derivatives of the flat conformal factor $v$.

Recalling
$$
v(\rho)=\log|z^{J}+\rho^{J}(A_{1}z^{J-1}+\dots+A_{J})|+\calO(\rho^{J+\epsilon}),
$$
then direct computation gives that
\begin{multline*}
\frac{\partial^{k} v}{\partial\rho^{k}}  = \frac{J!}{(J-k)!} \rho^{J-k}\, \Re \sum_{\ell=1}^{J} \frac{A_{1}z^{J-1}+\dots+A_{J}}{z^{J}+\rho^{J}(A_{1}z^{J-1}+\dots+A_{J})}  + \calO(\rho^{J-k+\epsilon}),
\end{multline*}
and this clearly vanishes at $\rho=0$ when $k \leq J-1$. 
\end{proof}

Now let us prove the corresponding fact for the eigenvalue. Differentiate the equation
$$
(\Delta_{g_{2}}-\lambda)\phi=0, g_{2}=e^{2\hat v} h_{0}
$$
with respect to $\rho$; this gives
$$
(-2\hat v_{\rho} \Delta_{g_{2}}-\lambda_{\rho}) \phi + (\Delta_{g_{2}}-\lambda) \phi_{\rho}=0.
$$
As in~\cite{MZ}, the eigenfunction $\phi$ is polyhomogeneous on $\calC_K$, and hence lifts to be polyhomogeneous 
on $\nwTotl$.  Now multiply this expression by $\phi$ and integrate to get
\begin{equation}\label{e:lrho}
\lambda_{\rho}=-\int 2\lambda\hat v_{\rho} |\phi|^{2}.
\end{equation}
Using $\hat v_{\rho}|_{\rho=0} =0$, we obtain $\lambda_{\rho}(0)=0$.  

Similarly, taking another derivative gives 
$$
\lambda_{\rho\rho}=-2\lambda_{\rho} \int \hat v_{\rho}|\phi|^{2}
-4\lambda \int \hat v_{\rho} \phi \phi_{\rho} - 2\lambda \int\hat  v_{\rho\rho}|\phi|^{2},
$$
so by Lemma~\ref{l:vdot} again, $\lambda_{\rho\rho}(0)=0$.  More generally, 
\[
\frac{d^{k}\lambda}{d\rho^{k}} = -2 \sum_{k_{1}+k_{2}+k_{3}=k-1 \atop k_{i}\geq 0} 
(\del_\rho^{k_1} \lambda) \, (\del_\rho^{k_2 + 1} \hat v)  \, (\del_\rho^{k_3} (|\phi|^2))
\]
Evaluating at $\rho = 0$ and assuming by induction that $\del_\rho^j \lambda(0) = 0$ for $j \leq k-1$, the only terms remaining in
the expression above are those with with $k_{1}=0$. However, in that case $k_2 + 1 \leq k-1$ so $\del_\rho^{k_2+1} \hat v(0)=0$ and
all terms in the sum vanish.

\medskip

We now obtain from \eqref{e:preLambda} and~\eqref{e:udot} that
\begin{equation}\label{e:dotA}
\dot \Lambda=\lim_{\epsilon \to 0} \int_{r=\epsilon} \left(\dot{\hat{v}} \del_r \phi - \del_r \dot{\hat{v}} \phi \right) \, r d\theta.
\end{equation}
To evaluate this more explicitly, we next show that only $\dot{v}$ (in the decomposition $\dot{\hat{v}} = \dot{\tilde{v}} + \dot v$) contributes.
Indeed, this follows by differentiating $\tilde v(\rho)-\tilde v(0)=\calO(\rho^{J+\epsilon}) = \calO(t^{1+\epsilon'})$ with respect to $t$. 
We conclude that
\begin{equation}
\dot \Lambda= \sum_{p_j} \lim_{\epsilon\rightarrow 0}\int_{\{r=\epsilon\}} ( \dot v \del_r \phi - \phi \del_r \dot v)\, rd\theta.
\end{equation}
(The sum indicates that we must sum the appropriate quantity over all conic points.) 

Suppose first that $\beta_{j}>1$. Then 
\begin{equation}\label{e:q}
  \begin{split}
\phi & \sim \sum_{\ell=0}^{[\beta_j]} (a_{\ell}'\cos( \ell\theta)+a''_{\ell}\sin (\ell\theta)) r^{\ell/\beta_j} + \calO(r^{1+\epsilon}), \\ 
\dot v & \sim \sum_{m=0}^{[\beta_j]}[e_m'\cos( m\theta)+e''_m\sin (m\theta)] r^{-m/\beta_j} + \calO(r^{\epsilon}),
\end{split}
\end{equation}
where \eqref{e:v0} relates $\{e_{m}', e_{m}''\}$ to $\{A_{i}:i=1, \dots, [\beta_{j}]\}$.  A brief computation then shows that this integral equals
$$
2\pi \sum_{\ell=1}^{[\beta_{j}]} \ell(a_{\ell}'e_{\ell}'+a_{\ell}''e_{\ell}'')
$$
in the limit as $\epsilon \to 0$. 

On the other hand, when $\beta_{j}<1$, 
\begin{equation}\label{e:phi}
\begin{split}
\phi & \sim c_{j0}'+ [c_{j1}'\cos( \theta)+c''_{j1}\sin (\theta)] r^{1/\beta_{j}}+ \calO(r^{2}), \\
\dot v & \sim d_{j0}' + [d_{j1}'\cos( \theta)+d''_{j1}\sin (\theta)] r^{-1/\beta_{j}} + \calO(r^{\epsilon}),
\end{split}
\end{equation}
and this leads to the contribution
$$
2\pi(c_{j1}'d_{j1}'+c_{j1}''d_{j1}'').
$$
We have now proved the formula 
\begin{equation}\label{e:balancing}
  \dot \Lambda=2\pi \sum_{j=1}^{k_0} \sum_{\ell=1}^{[\beta_j]} \ell (a'_{j \ell}e'_{j \ell}+ a''_{j \ell}e''_{j \ell}) +
  2\pi \sum_{j=k_0+1}^{k} (c_{j1}'d_{j1}'+c_{j1}''d_{j1}'').
\end{equation}

The key feature of \eqref{e:balancing} is the fact that the eigenfunction $\phi$ determines the constants 
$\{ a_{j\ell}', a_{j\ell}''\}$ and $\{c_{j1}', c_{j1}''\}$.   We prove below that if these all vanish, then $\phi = 0$.
Assuming that, then the equation $\dot \Lambda = 0$ defines a codimension $1$ subspace of the data 
$(e_{j\ell}',e_{j\ell}'', d_{j1}', d_{j1}'')$ for $v$.  Both sets of data lie in $\RR^{2K}$, $K = \sum_{j=1}^{k_0} [\beta_{j}]+(k-k_0)$.

If the dimension of $E_2$ has dimension $\ell > 1$, the computation is similar.  Since $\EE_2$ is smooth, there exists a local 
smooth orthonormal basis $\{\phi_{1}, \dots, \phi_{\ell}\}$, and the error term $f_{\hfrakq}$ in~\eqref{e:eigenerror}
is a linear combination of these sections at every point $\hfrakq$. Calculating the derivative $\dot{f}$ along any curve $\vec Z(t)$
emanating from $\tilde\frakq_0$ just as before, and pairing with each element of this basis at $\tilde\frakq_0 \in \tilde F_0$, we see that 
\begin{equation}
\int_{M} \dot f \phi_{i} \, dA_{g_{2}}=\lim_{\epsilon\rightarrow 0}\int_{\{r=\epsilon\}} 
(\dot v \del_r \phi_i - \phi_i \del_r \dot v)\, rd\theta.
\end{equation}
Each $\phi_{i}$ has an expansion as in \eqref{e:q} with coefficients $\{a_{jim}', a_{jim}''\}$, $j = 1, \ldots, k_0$, and
an expansion as in \eqref{e:phi} with coefficients $\{c_{ji m}', c_{ji m}''\}$ when $j > k_{0}$.   Similarly,
$\dot v$ has expansions with coefficients $\{e_{j m}', e_{j m}''\}$ and $\{d_{jm}', d_{jm}''\}$ for $j \leq k_0$ and $j > k_0$,
respectively. The same computation shows that the infinitesimal variation of $f$ vanishes provided
\begin{equation}
0= 2\pi\left(\sum_{j=1}^{k_0} \sum_{m=1}^{[\beta_j]} m(a_{jim}'e_{jm}'+a_{jim}''e_{jm}'') + \sum_{j=k_{0}+1}^{k} (c_{ji1}'d_{j1}'+c_{ji1}''d_{j1}'')\right)
\end{equation}
for each $i = 1, \ldots, \ell$. 

We summarize all of this in the following
\begin{definition}
Fix any basis $\phi_1, \ldots, \phi_\ell$ for $E_2$ and define the coefficient pairs $(a_{jim}', a_{jim}'')$ in the expansion
for $\phi_i$ near each $p_j$, $j = 1, \ldots, k_0$ and triples $(c_{ji0}, c_{ji1}', c_{ji1}'')$ near $p_j$ with  $j > k_0$.
Suppose that $\dot v\in r^{-1/\beta_{k}} \calC_{b}^{\ell,\delta}(M_{\frakp_{0}}) \cap \calA_{\mathrm{phg}}$, with
coefficient pairs $(e_{jm}', e_{jm}'')$ for each $j = 1, \ldots, k_0$ and $m \leq [\beta_j]$ and triplets $(d_{j0}, d_{j1}', d_{j1}'')$ 
for $p_j$ with $j > k_0$. 

Define the bilinear form $B: \RR^{2K} \times E_2 \to \RR$,  
\begin{equation}\label{e:balancing2}
\begin{split}
B(\phi_i, \dot{v}) = \sum_{j=1}^{k_{0}}\sum_{m=1}^{[\beta_{j}]} & m (a_{jim}'e_{jm}'+a_{jim}''e_{jm}'')  \\
& + \sum_{j=k_{0}+1}^{k} (c_{ji1}'d_{j1}'+c_{ji1}''d_{j1}''), \ \ i=1, \dots, \ell.
\end{split}
\end{equation}

We say a vector $\dot v$ is a solution if $B(\phi_i, \dot{v})=0$ for $i=1, \dots, \ell$. Denote the vector space of all such $\dot v$ as $V$. 
\end{definition}

Our next goal is to prove that $\ell = \dim E_2$ is not too large. 
\begin{lemma}
Let $g_{0}$ be a spherical cone metric on $M = \bS^2$ with $k \geq 3$ cone points. If $(\Delta_{g_0} - 2) \phi = 0$, i.e.,
$\phi \in E_2$, and furthermore, near each $p_j$, $\phi = \mbox{const.} + \calO(r^{1+\epsilon})$, then $\phi \equiv 0$.
\end{lemma}
\begin{proof}
We restate the assumption as saying that near a cone point with $\beta_{j}>1$ all coefficients of the terms $r^{m/\beta_j}$ vanish, $m = 1, \ldots, [\beta_j]$.
The proof of \cite[Proposition 13]{MW} can then be applied verbatim.  Indeed, recall: the absence of these terms 
validates the integration by parts
$$
\int_{M} \langle \Delta_{1}d\phi, d\phi \rangle=\int_{M} (|\nabla d\phi|^{2}+|d\phi|^{2});
$$
here $\Delta_{1}=\nabla^{*}\nabla + 1$ is the Hodge Laplacian for 1-form.  Next, using $\Delta_{1} d\phi=d \Delta\phi=2d\phi$ and 
the Cauchy-Schwarz inequality, $|\nabla d\phi|^{2}\geq \frac{1}{2}|\Delta \phi|^{2}$, we see that
\begin{multline*}
2\|d\phi\|^{2}=\int_{M} \langle  d\Delta_{0}\phi, d\phi\rangle =\int_{M} \langle \Delta_{1} d\phi, d\phi\rangle = \int_{M} (|\nabla d\phi|^{2}+|d\phi|^{2})\\\geq  \int_{M}(\frac{1}{2}|\Delta \phi|^{2}+|d\phi|^{2})=\int_{M} (\frac{1}{2} \langle  \Delta_{1} d\phi,d\phi\rangle + |d\phi|^{2})=2\|d\phi\|^{2}.
\end{multline*} 
Hence the inequalities are actually equalities, so $|\nabla d\phi|^{2}=\frac{1}{2}|\Delta \phi|^{2}=2|\phi|^{2}$, i.e., the Hessian of $\phi$ is pure trace:
\[
\nabla d\phi=-\frac{1}{2}\Delta \phi \cdot g=-\phi\cdot g.
\]

If $J$ is the complex structure on $M$, then define the vector field $J\nabla \phi$. For any vector fields $X,Y$ on M, 
\[
\begin{split}
\calL_{J\nabla \phi}g(X,Y) & =(J\nabla \phi)g(X,Y)-g([J\nabla \phi, X], Y)-g(X,[J\nabla \phi, Y])\\
& =g(\nabla_{J\nabla \phi}X-[J\nabla \phi, X],Y) + g(X,\nabla_{J\nabla \phi}Y-[J\nabla \phi,Y])\\
& =g\left(\nabla_{X}(J\nabla \phi), Y\right) + g\left(X, \nabla_{Y}(J\nabla \phi)\right)\\
& =-g(\nabla_{X}\nabla \phi,JY)-g(\nabla_{Y}\nabla \phi,JX)\\
& =-(\nabla d\phi)(X,JY)- (\nabla d\phi)(Y,JX)\\ & =\phi\cdot [g(X,JY)+g(Y,JX)]=0,
\end{split}
\]
i.e., $\nabla \phi$ is a Killing field on $M_{\frakp}$.  It also extends over each $p_j$ as a conformal Killing field since
it vanishes at these points. However, no such field exists since $k \geq 3$. Hence $\nabla \phi \equiv 0$ and so
$\phi$ is constant; but $\Delta \phi = 2\phi$ so $\phi \equiv 0$. 
\end{proof}

\begin{lemma}
Let $(M, g_0)$ be a spherical cone metric with $k \geq 3$ conic points (so $M$ is not a spherical football). Then the rank of the 
linear system is precisely $\ell$, hence $\ell = \dim E_2 \leq 2K_{0}$ where $K_{0}=\sum_{j=1}^{k_{0}} [\beta_{j}]$.  Furthermore: 
\begin{itemize}
\item if $\ell<2K$, then there is a $(2K-\ell)$ dimensional space of solutions $(e_{jm}', e_{jm}'', d_{j1}', d_{j1}'')$ in $V\subset \RR^{2K}$;
\item if $\ell=2K_{0}=2K$ (so $k_{0}=k$), then $V$ is trivial. 
\end{itemize}
\end{lemma}
\begin{proof}
To prove the first assertion, suppose that $\ell > 2K_0$. Then there exists some linear combination of elements in $E_2$ 
 which vanishes like $r^{1+\epsilon}$ modulo constants for each $p_j$, $j \leq k_0$.  However, as we have just shown, any such 
$\phi$ vanishes identically, which is a contradiction.  The remaining statements are elementary.
\end{proof}

\begin{remark}
When $(M,g_{0})$ is a spherical football of angle $2\pi\beta$, then by direct computation, $K_{0}=2[\beta]$, so 
$$
k_0 =0, K =2, \text{ if }\beta<1; \ k_0 = 2, K=2[\beta], \text{ if }\beta\geq 1.
$$
Since $\ell = 1$ if $\beta\notin\NN$ and $\ell=3$ if $\beta\in\NN$, the football always lies in the first case above, with $\ell < 2K$. 
\end{remark}

We now state our final and main theorem. Recall from Definition~\ref{d:psub} in \S 6.2 the definition of $p-$submanifolds,
and from Definition~\ref{def:B} in \S 6.3 the set $B$ of admissible angles. 
\begin{theorem}
Let $(M,g_0)$ be a spherical cone metric.  With all the notation as above, and in particular setting
$$
K_{0}=\sum_{j=1}^{k_{0}} [\beta_{j}], \ \ K= \sum_{j=1}^k  \max\{[\beta_{j}],1\} = K_{0}+ (k-k_{0}),
$$
then: 
\begin{enumerate}
\item[a)] (The unobstructed case) if $2\notin \spec(\Delta_{g_{0}})$, there is a smooth neighborhood in the space of spherical
cone metrics around $g_{0}$ parametrized by $(\fraks, \frakp, \vec \beta) \in \Metcc \times (\calE_k)^\circ \times \RR^k_+$; 
\item[b)] (Partial rigidity) if $1 \leq \dim E_2 = \ell < 2K$, then for any $\vec B\in B$ and $\fraks \in \Metcc$ near the dataset 
for $g_0$, there exists a $2K-\ell$ dimensional $p$-submanifold $X \subset \nwBase$ such that for each point $\tfrakq\in X$, 
there exists a spherical cone metric, i.e., a solution to $K(e^{2u}h_{0})-1=0$; this family of solutions depends on a choice
of a branch of roots as described in \S 6, but given that choice, locally unique for the specific conic data; 
\item[c)] (Rigidity) if $K_{0}=K$ and $\ell = 2K_0$, then for any $\vec B\in B$ and $\fraks \in \Metcc$, there is 
a neighborhood $\calU \in \nwBase$ such that $\frakp_{0}$ is the only configuration admitting a spherical cone metric.
In other words, there is no nearby spherical cone metric obtained by moving or splitting the conic points of $g_0$.
\end{enumerate}
\end{theorem}
\begin{remark}
We have not explicitly stated the smooth dependence of these solutions on the underlying smooth conformal class. Indeed, this is
a bit complicated since even if a given conic metric $(M,g_0)$ is obstructed, i.e., $2$ is in the spectrum of its scalar Laplacian 
so we are case b), then for generic nearby classes and conic metrics, $2$ will no longer be in the spectrum. More generally, if the multiplicity
of the eigenvalue $2$ is $\ell$, there is a local stratification of the space of nearby conic metrics according  to the multiplicity
of this eigenvalue. The family of solution metrics is smooth within each stratum. However,  this is somewhat subtle 
and we do not prove this here. 
\end{remark}
\begin{proof}
Case 1 is just Theorem~\ref{t:no2}.  

Next, for prove Case 2, we show that the set of configurations $\tfrakq \in \tilde \calE_K$ near $\tfrakq_0$ for which the 
error term $f$ in~\eqref{e:eigenerror} vanishes is a $p$-submanifold.  Since $\ell < 2K$, there is a $(2K-\ell)$-dimensional 
subspace $V\in \RR^{2K}$ for which $\dot f(0)=0$. In terms of the map $P^{(K)}$ from~\eqref{e:PK}, we 
define
$$
\vec \Lambda: \calA_K \rightarrow \EE_{2}, \ \tilde A \mapsto \vec \Lambda(P^{(K)}(\tilde A)). 
$$ 
The computation in~\eqref{e:balancing2} determines the $2K-\ell$ dimensional kernel of $d\vec\Lambda$. 

The front face of $\calA_K$ is locally diffeomorphic to $\widetilde{F}_0$.  
Writing the error $f = \sum \Lambda_i \phi_i$, then near this front face, then by the computation leading to~\eqref{e:balancing2},  for each $i$, 
$$
\Lambda_{i}=t\dot \Lambda_{i}(0) + \calO(t^{1+\epsilon}),
$$
hence 
$$
(\Lambda_{i}/t)|_{t=0}=\sum_{j=1}^{k_{0}}\sum_{m=1}^{[\beta_{j}]} m (a_{jim}'e_{jm}'+a_{jim}''e_{jm}'')+ \sum_{j=k_{0}+1}^{k} (c_{ji1}'d_{j1}'+c_{ji1}''d_{j1}'')
$$
on this front face. 
Note that $\vec \Lambda(t)=0$ on the
face $\tilde{F}_0$, so if there is an additional submanifold $X$ which is transverse to this face and on which
$\vec \Lambda(t)$ vanishes, then this function must vanish to second order at $X \cap \tilde{F}_0$. If we could show that $\vec \Lambda /t$ changes sign when crossing a submanifold on $\tilde F_{0}$, hence there exists such a p-submanifold $X$ where $\Lambda=0$. 
Since the projectivisation of $\{\vec x=(e_{*}', e_{*}'', d_{*}', d_{*}'')\}$ give coordinates on this front face, we 
conclude that there exist directions such that $\del_x (\vec \Lambda/t)= 0$, and other directions for which
$\del_x(\vec \Lambda/t) \neq 0$.  

The implicit function theorem now provides the existence of a $p$-submanifold in $\widetilde{\calE}_K$, as claimed.

Finally, Case 3 is almost the same as Case 2.  Here the linear system~\eqref{e:balancing2} has no nontrivial solutions, so 
$d\vec\Lambda$ is invertible, and the conclusion follows from the inverse function theorem.
\end{proof}
\begin{remark}
Case 2 definitely can occur. The reader is referred to \cite{Zhu18}, which discusses the example of two glued footballs. There the dimension
of $E_2$ is positive but not maximal.  We are not aware of any examples which fall under Case 3, but it is not excluded by any results
at present.  
\end{remark}

\begin{proposition}
In Cases 1 or 2 above, the solution $u(\fraks, \vec B, \tfrakq)$ lies in $\calC_{b}^{m,\alpha}(\tilde\calC_{K})$ 
for all $m$ and is polyhomogeneous on $\nwTotl$. 
\end{proposition}
\begin{proof}
This follows from Proposition~\ref{p:polyhomo}. The solution modulo the obstruction bundle is polyhomogeneous, hence so
is its restriction to any smooth $p$-submanifold.  However, the restriction to the particular $p$-submanifold identified 
in this Theorem corresponds to the family of actual solutions.
\end{proof}

\noindent\textbf{Acknowledgement: }
The authors are grateful to Alexandre Eremenko, Andrea Malchiodi, Richard Melrose, Gabriele Mondello, Dmitri Panov, Yanir Rubinstein, Song Sun, and Bin Xu for helpful conversations. R.M. is supported by the NSF grant DMS-1608223. X. Z. is supported by the NSF grant DMS-2041823.

\end{document}